\documentclass{article}

\usepackage{amssymb,amsthm}
\usepackage{amsmath}
\usepackage{algorithm,algorithmic}
\usepackage{appendix}
\usepackage{graphicx}
\usepackage{float}
\usepackage{subcaption}
\usepackage{color}
\usepackage[normalem]{ulem}

\usepackage[unicode=true,pdfusetitle,
 bookmarks=true,bookmarksnumbered=false,bookmarksopen=false,
 breaklinks=false,pdfborder={0 0 1},backref=false,colorlinks=false]
 {hyperref}

\allowdisplaybreaks

\DeclareMathOperator*{\argmin}{argmin}

\DeclareMathOperator*{\tr}{Tr}

\DeclareMathOperator*{\unif}{\mathbb{U}}
\DeclareMathOperator*{\poly}{Poly}

\newcommand{\E}{\mathbb{E}}
\newcommand{\Prob}{\mathbb{P}}
\newcommand{\R}{\mathbb{R}}

\newcommand{\SD}{\mathcal{S}^d}
\newcommand{\SN}{\mathcal{S}^n}

\newcommand{\T}{^{\top}}
\newcommand{\bq}{\overline{Q}}
\newcommand{\br}{\overline{R}}
\newcommand{\bbs}{\overline{\|B\|^2}}
\newcommand{\abs}{\overline{\|A-BL\|^2}}
\newcommand{\bnr}{\overline{\|R\|}}
\newcommand{\bqlrl}{\overline{Q+L^{\top} R L}}
\newcommand{\bqlrlp}{\overline{Q+(L')^{\top} R L'}}

\newcommand{\bnxs}{\overline{\|x_0\|^2}}
\newcommand{\blrlp}{\overline{(L')^{\top} R L'}}
\newcommand{\blrl}{\overline{L^{\top} R L}}
\newcommand{\rl}{\overline{R+B\T P_L B}}
\newcommand{\bbpla}{\overline{B\T P_L A}}
\newcommand{\us}{\underline{\sigma}}
\newcommand{\FF}{\mathcal{F}}
\newcommand{\TT}{\mathcal{T}}
\newcommand{\AAA}{\mathcal{A}}
\newcommand{\BBB}{\mathcal{B}}
\newcommand{\BBbb}{\mathbb{B}}
\newcommand{\SSbb}{\mathbb{S}}
\newcommand{\xut}{\begin{bmatrix}x_t\\u_t\end{bmatrix}}
\newcommand{\xutt}{\left[x_t\T,  u_t\T\right]}
\newcommand{\xuti}{\begin{bmatrix}x_t^{(i)}\\u_t^{(i)}\end{bmatrix}}
\newcommand{\hdelta}{h_{\Delta}}
\newcommand{\hgrad}{h_{\text{grad}}}
\newcommand{\bhdelta}{\overline{\hdelta}}

\newcommand{\hsnsop}{h_{\text{sample,NSOP}}}
\newcommand{\hsnsf}{h_{\text{sample,NSF}}}

\newcommand{\hshf}{h_{\text{sample,HF}}}

\newcommand{\hsbf}{h_{\text{sample,F}}}
\newcommand{\hs}{h_{\text{sample}}}
\newcommand{\hd}{\widetilde{\nabla C} _r(L)}
\newcommand{\hdl}{\widetilde{\nabla C} _r^{(l)}(L)}
\newcommand{\td}{\widehat{\nabla C}(L)}
\newcommand{\hrgd}{h_{\text{r,GD}}}
\newcommand{\hlgd}{h_{\text{l,GD}}}
\newcommand{\cc}{\mathcal{C}}

\newtheorem{theorem}{Theorem}

\newtheorem{lemma}{Lemma}[section]
\newtheorem{corollary}{Corollary}[section]
\newtheorem{assumption}{Assumption}
\newtheorem{definition}{Definition}

\newtheorem{assumptionalt}{Assumption}[assumption]
\newenvironment{assumptionp}[1]{
  \renewcommand\theassumptionalt{#1}
  \assumptionalt
}{\endassumptionalt}

\newtheoremstyle{TheoremNum}
{\topsep}{\topsep}              
{\itshape}                      
{}                              
{\bfseries}                     
{.}                             
{ }                             
{\thmname{#1}\thmnote{ \bfseries #3}}
\theoremstyle{TheoremNum}
\newtheorem{thmn}{Theorem}

\numberwithin{equation}{section}

\theoremstyle{definition}
\newtheorem{remark}{Remark}[section]

\renewcommand{\algorithmicrequire}{ \textbf{function:}}

\title{Policy Gradient Methods for Discrete Time Linear Quadratic Regulator with Random Parameters}

\date{November 2023}

\begin{document}
\author{Deyue Li\thanks{Shanghai Center for Mathematical Sciences; Fudan University, Shanghai, China 200433 ({lidy21@m.fudan.edu.cn})} 
}
\maketitle

\begin{abstract}
	This paper studies an infinite horizon optimal control problem for discrete-time linear system and quadratic criteria, both with random parameters which are independent and identically distributed with respect to time. 
	In this general setting, we apply the policy gradient method, a reinforcement learning technique, to search for the optimal control  without requiring knowledge of statistical information of the parameters.   
	We investigate the sub-Gaussianity of the state process and establish global linear convergence guarantee for this approach based on assumptions that are weaker and easier to verify compared to existing results. 
	Numerical experiments are presented to illustrate our result.

 \noindent
 {\bf Key words}:
 linear quadratic optimal control, random parameters,  reinforcement learning, model-free policy gradient method, sub-Gaussianity

\noindent
{\bf AMS subject classification}: \, 49N10, 68W40, 93E35 
\end{abstract}

\section{Introduction}
Linear Quadratic (LQ) control problem for discrete time with random parameters whose study goes back to Kalman~\cite{kalman1961}  finds applications in a wide range of practical problems, 
such as random sampling of a diffusion process in digital control~\cite{tiedemann1984}, sampling of a system with noise~\cite{drenick1964} and economic systems~\cite{aoki1976}. 
Consequently, extensive results has been carried out in this area~\cite{drenick1964, athans1977, de1982, morozan1983, beghi1998}. 
However, the literatures cited above assumes a priori knowledge of model parameters, which is unrealistic in many practical scenarios. 
Therefore, solving such problem without statistical information of model parameters are of great importance from both theoretical and practical perspectives.

Recent years  have witnessed  a huge growth in  learning approaches, among which the reinforcement learning (RL) method has garnered a great deal of attention from researchers~\cite{ fazel2018, tu2019, gravell2020, hambly2021, du2022, lai2023}. 
There are two categories of RL-methods: the model-based RL and the model-free RL.
The model-based RL approach estimates the transition dynamics by observing or conducting experiments and then designs the control policy using the estimated parameters~\cite{simchowitz2018, dean2020}. 
On the other hand, the model-free RL optimizes the control policy by interacting with the system directly without inferring the model parameters. 
Model-free RL has several appealing features. Firstly, it can solve control problems even when the parameter distribution is unknown. Secondly, it is more robust and can tolerate some mis-specification of a complex model. Thirdly, it is relatively easy to implement.

Due to the good nature of LQ problem, 
optimal control has linear feedback form~\cite{de1982, du2022}. 
A natural way is to employ gradient descent to search for the optimal feedback matrix over a certain matrix space. Such method is called policy gradient method and is a kind of model-free method. 
Fazel et al.~\cite{fazel2018} was the first to prove the global convergence result of policy gradient within the LQ problem framework. 
Their research primarily focused on an infinite horizon LQ problem where parameter matrices remained constant and time-invariant. 
In this scenario, the only source of randomness comes from the initial state. 
To achieve global convergence, their theoretical analysis required the initial state to be almost surely bounded. 
This result has been extended in various directions. 
In the realm of finite horizon setting, Hambly et al.~\cite{hambly2021} investigated the policy gradient method 
for LQ problem with additive noise. 
They not only incorporated additive noise into the system transition dynamics, but also 
relaxed the requirement of almost surely boundedness of the initial state.  
In their setting, achieving global convergence required 
the sub-Gaussianity of the initial state and the additive noise term.  
In the realm of infinite horizon setting,
Gravell et al.~\cite{gravell2020} were the first to consider the presence of multiplicative noise in system transition dynamics. 
They provided a proof of global convergence of policy gradient method, thereby further enhancing the algorithm's robustness.
However, the randomness in parameter matrices also have a specific structure. 
To ensure global convergence, they assumed that the noise term in parameters and the initial state were 
almost surely bounded and satisfied a technical condition, which is challenging to verify.

In this study, we consider a generic setting with structureless random parameter matrices 
in both transition dynamics and cost. 
Due to the lack of structural information about the randomness to exploit, 
we need to perform a more detailed probabilistic analysis of relevant random matrices and state process. 
Furthermore, the presence of randomness in cost imposes further uncertainty on our gradient estimator, which  
entails more careful treatments and novel proofs of algorithm's convergence. 
We not only consider a more general model,
but also relax the convergence condition proposed in~\cite{gravell2020} in the following directions.
First, similar to~\cite{hambly2021}, we assume that the initial state is sub-Gaussian instead of almost surely boundedness. 
Second, we eliminate the technical conditions from their convergence condition of the algorithm.

RL methods for solving LQ problem with random parameters are less developed.  
One existing approach is a model-free Q-learning method proposed by Du et al.~\cite{du2022} to learn the optimal policy. 
However, to the best of our knowledge, there are no theoretical guarantees for the convergence of policy gradient methods in this generic setting. 
Filling this research gap is meaningful for broadening the class of models to which the policy gradient method can be applied 
and improving the robustness of the policy gradient algorithm.   

\subsection{Problem formulation and preliminaries}\label{sec:pre}
  We consider the following infinite-horizon LQ problem for discrete time systems with random parameters. Given the initial state $x_0 \in \R^n$ , the system evolves as 
  \begin{equation}
      x_{t+1}=\Lambda_{t+1}\xut ,\quad t=0,1,2\ldots,
  \end{equation}
  where $x_t$ is the system state at time $t$, and  $u_t\in \R^m$ is the control at time $t$. The cost function is defined as 
  \begin{equation}
      J(x_0,u_.)=\sum_{t=0}^{\infty}\xutt N_{t+1} \xut,
  \end{equation}
  where random parameters 
  \[
  \begin{bmatrix}
		\Lambda_t\T&N_t
	\end{bmatrix}	:=\begin{bmatrix}
	A_t\T&Q_t&0\\
	B_t\T&0&R_t
\end{bmatrix},\quad t=1,2,3\ldots
  \]
  affecting the system form time $t$ to $t+1$  are not exposed until time $t+1$.

  Our goal is to find admissible controls that solve the following optimization problem:
  \begin{subequations}
     \begin{align}
	&\text{minimize} \quad \E \left[J(x_0,u_.)\right],\label{cost}	\\
    &\text{such that} \quad x_{t+1}=\Lambda_{t+1}\xut,\quad x_0\sim \mathcal{D},\label{sys}
	\end{align}
  \end{subequations}
  where $x_0\in \R^n $ is the initial state drawn from initial distribution $\mathcal{D}$ with finite second order moment 
  $\E\left[x_0x_0\T\right]$, and  $N_{t+1}$ is
  positive semi-definite. The random matrices 
  $\left[\Lambda_t\T, N_t\right]$ 
  are assumed to be independent and identically distributed. 
  In sequel, we use $\left[\Lambda\T, N\right]$ to denote an independent copy of $\left[\Lambda_1\T, N_1\right]$.
  
  As we are dealing with an infinite horizon problem, it is necessary to consider the question of stability. 
  A control $u_.$ is said to be mean-square stabilizing if the state process under control satisfies $\lim_{t\to+\infty} \E\left[x_t\T x_t\right]=0$. 
  The system~\eqref{sys} or $\left(A,B\right)$ is said to be mean-square stabilizable if there exists a matrix $L$ such that the linear feedback control of the form $u_.=-L x_.$ is mean-square stabilizing. 
  In this case, $L$ is said to mean-square stabilize the system.

  Given a positive semi-definite matrix $P\in \SD$ with $d=n+m$, we use the following partition to denote specific sub-matrices:
  \[P=\begin{bmatrix}
		P_{xx}&P_{xu}\\
		P_{ux}&P_{uu}
	\end{bmatrix}\quad	\text{with} \quad P_{xx}\in\R^{n\times n}.\]
	And two mappings are defined as follows:
	\begin{equation}
	\label{op}
	\begin{aligned}
		\Pi(P):=&P_{xx}-P_{xu}P_{uu}^{\dag}P_{ux},\\
		\Gamma(P):=&-P_{uu}^{\dag}P_{ux},
	\end{aligned}
	\end{equation}
    where $P_{uu}^{\dag}$ denotes the Moore-Penrose pseudo-inverse of $P_{uu}$.
  
  To ensure the optimization problem \eqref{cost}-\eqref{sys} is non-trivial, 
  we introduce some conditions on parameters to guarantee the existence of solution.
  
  \begin{assumption}
    \label{ass:1}
  Assume that $\E\left[\|x_0\|^2\right]$, $\E\left[\|N\|\right]$ and $\E\left[\|\Lambda\|^2\right]$ are finite and that $\E\left[N\right]$ is positive definite and that $\left(A,B\right)$ is mean-square stabilizable.
  \end{assumption}
  
  Under Assumption~\ref{ass:1} we can define matrix $K$ as the unique positive definite solution to the stochastic algebraic Riccati equation (ARE)~\cite{de1982,du2022}:
  \begin{equation}
  \label{eq:ARE}
      K=\Pi(\E\left[N+\Lambda\T K \Lambda\right]).
  \end{equation}
Moreover,  the value function, defined as $$V(x)\colon = \inf_u J(x,u_.),$$ has a quadratic form of $V(x)=x\T K x$, and  the optimal control is given by linear feedback form~\cite{de1982,du2022}
     \begin{equation}
        \label{eq:feedbackcontrol}
         u_t=\Gamma(\E\left[N+\Lambda\T K \Lambda\right])x_t.
     \end{equation}
     
	In cases where the statistical information of parameters is known and mathematical expectation can be calculated explicitly, the fix point iteration method can be used to solve~\eqref{eq:ARE}~\cite{athans1977,ku1977}

\eqref{eq:feedbackcontrol} suggests that we only need to focus on the following class of linear feedback policies :
\begin{equation}
    \label{eq:linearcontrol}
    u_t=-L x_t
\end{equation} 
under which \eqref{sys} is equivalent to the closed loop system
\begin{equation}
    \label{eq:cls}
    x_{t+1}=(A_t-B_t L)x_t.
\end{equation}

We say that $L$ is admissible if $L$  mean-square stabilize the system $\left(A,B\right)$. 
To describe admissibility of matrix $L$ we introduce a linear transformation $\FF_L : \SN \to \SN $ defined as 
\begin{equation}
    \FF_L(X):=\E\left[(A-BL) X (A-BL)\T\right].
\end{equation}
As shown in~\cite{de1982}, $L$ is admissible if and only if the spectral radius $\rho(\FF_L)<1$. 
Using this result, we can narrow the admissible set as follows:
\begin{equation}
\label{eq:ad}
    U_{ad}: = \left\{L\in \R^{m\times n}\big|  \rho(\FF_L)<1 \right\}.
\end{equation}

\subsection{Our contributions}
This paper further generalizes the results of Gravell et al.~\cite{gravell2020}.  
Compared with previous work, the novelties are as follows:
\begin{itemize}
	\item The matrices occurring in both transition dynamics and quadratic criteria are allowed to be random, and the randomness has no special structure, which entails more careful treatments and novel proofs.
	\item More advanced analyses of concentration property of state process provide weaker and easily verifiable conditions for ensuring the performance of model-free algorithms. 
	\item The overall performance of the model-free algorithm is elaborated which allows the algorithm's performance to be controlled directly and precisely. 
\end{itemize}

\subsection{Notations}
Throughout this paper, we use $\R^n$ to denote the Euclidean $n$-space, where the inner product is given by $\left<x,y\right>_{\R^n}=x\T y$ and the norm is $\|x\|_2=(x\T x)^{1/2}$. 
We use $\R^{n\times m}$ to denote the Banach space of all real $n\times m$ matrices with operator norm $\|M\|=\max_{\|x\|_2=1}\|Mx\|_2$, 
and use $\sigma_{\min}(M)$ to denote the smallest singular value of matrix $M$,
where $M\in \R^{n\times m}$ and $x\in\R^m$. 
The identity matrix in $\R^{n\times n}$ is denoted by $I_n$. 
We also use $\SN$ to denote the the Hilbert space of all real symmetric $n\times n$ matrices, equipped with the Frobenius inner product $\left<A,B\right>_F=\tr(A\T B)$. 
Note that the operator norm and Frobenius norm $\|M\|_F=\left(\tr(M\T M)\right)^{1/2}$ are equivalent norm on $\R^{n\times m}$ and $\SN$.  
We use $\mathcal{J}^n$ to denote the Banach space of all linear transformations $\TT \colon \SN  \to \SN $, with norm $\|\TT\|=\max_{\|M\|=1}\|\TT M\|$. The spectral radius of $\TT\in\mathcal{J}^n$ is denoted by $\rho(\TT)$. 
We consider the probability space $\left(\Omega,\mathcal{F},\Prob\right)$ where the expectation of random element $X$ is denoted by $\E\left[X\right]$ or $\overline{X}$. 
We use $L^p\left(\Omega,\mathcal{F},\Prob\right)$ to denote the Banach space of all $p$-th order integrable real random variables, with norm $\|\xi\|_{L^p}=\left(\E\left[|\xi|^p\right]\right)^{1/p}$, where $\xi\in L^p\left(\Omega,\mathcal{F},\Prob\right)$.
The model parameters are those quantities that are only relevant to the intrinsic nature of the problem~\eqref{cost}-\eqref{sys}, 
such as $m$, $n$, $\overline{\|A\|^2}$, $\bbs$, $\overline{\|Q\|}$,  $\overline{\|R\|}$, $\overline{\|x_0\|_2^2}$,
$\sigma_{\min}(\Sigma_0)$,$\sigma_{\min}(\bq)$,$\sigma_{\min}(\br)$. $\|\Sigma_{L^*}\|$, $C(L^*)$, and their reciprocals.  
We use $\poly(\cdot)$ to denote a polynomial depending on model parameters and other dependent parameters in parentheses. 
We use $\cc$ to denote absolute constants. 
When constants depend on model parameters, 
we denote this dependency by using subscripts, for example, $\cc_{A}$, $\cc_{Q}$, and $\cc_{\Sigma_0}$. 
These constants may vary from line to line.

\subsection{Outline}
The rest of this paper is organized as follows. 
Section~\ref{sec:2} introduces the algorithm and presents its convergence result. 
In Section~\ref{sec:num}, we present numerical examples to illustrate the performance of the proposed algorithm. 
Section~\ref{sec:CAOMBA} aims to prove the convergence of model-based algorithm in Theorem~\ref{thm:mb} as an auxiliary result. 
Finally, Section~\ref{sec:samplebased} provides the proof of our main result in Theorem~\ref{thm:samplebased}. 
\section{Main results}\label{sec:2}
The most noteworthy result in this paper is Theorem~\ref{thm:samplebased}, which guarantees the global convergence of the model-free policy gradient method for LQ problems whose systems and quadratic criteria both have random parameters. 
In Section~\ref{sec:modelbased}, Theorem~\ref{thm:mb}, serving as an auxiliary theorem,  presents the convergence result of the model-based gradient method,  paving the way for the analysis of the model-free setting. 
In Section~\ref{sec:m2}, the convergence result in model-free setting and its sample complexity analysis are provided.
\subsection{Model-based policy gradient method}\label{sec:modelbased}
In this section we assume that the distribution information of all parameters $\left\{ \Lambda_t \right\}_{t=1}^{\infty}$,$\left\{N_t\right\}_{t=1}^{\infty}$ are known, and $L\in U_{ad}$. 

As shown in~\eqref{eq:feedbackcontrol}, admissible control is linearly characterized.
Accordingly, the cost function can be expressed as follows: 
\begin{equation}
    C(L):=\E\left[J(x_0,-Lx_.)\right]=\E\left[\sum_{t=0}^{\infty}x_t \T(Q_{t+1}+L\T R_{t+1} L)x_t\right].
\end{equation}
Thus, the optimization problem~\eqref{cost}-\eqref{sys} is equivalent to the following~\eqref{costl}-\eqref{feal}:
\begin{subequations}
    \begin{align}
    &\text{minimize} \quad  C(L),\label{costl}\\ 
    &\text{such that} \quad L\in U_{ad}. \label{feal}
    \end{align}
\end{subequations}
The optimal policy, $L^*:=\argmin_{L\in U_{ad} } C(L)$, is given by $L^*=\Gamma(\E\left[N+\Lambda\T K \Lambda\right]) $ according to~\eqref{eq:feedbackcontrol}.

To search $L^*$ directly, we consider the following model-based gradient decent with constant step size $\eta$:
\begin{equation}
\label{eq:gradientupdate}
    L_{k+1}=L_k-\eta\nabla C(L_k).	
\end{equation}
However, the convergence analysis of \eqref{eq:gradientupdate} is not immediately clear. As pointed out in~\cite{fazel2018}, the feasible set $U_{ad}$ is generally non-convex, and there is no guarantee that the gradient method will successfully solve a non-convex problem. Fortunately, similar to the deterministic case, we show that the cost function possesses two essential properties: {\it the gradient dominance} and {\it the almost smoothness}, which facilitate the analysis of non-convex gradient method. 
For these two properties to hold, the only extra condition that needs to be added is the positive definiteness of initial covariance matrix 
$\Sigma_0:=\E\left[x_0 x_0\T\right].$
\begin{assumption}
\label{ass:2}
We assume that $x_0\sim \mathcal{D}$ such that $\Sigma_0$ is positive definite.
\end{assumption}

We introduce some notations that will be used throughout the analysis. 
\begin{align}
    \mu&:=\sigma_{\min}(\Sigma_0),\\
    \us_{\bq}&:=\sigma_{\min}(\bq),\\
    \us_{\br}&:=\sigma_{\min}(\br).
\end{align}  
Note that Assumption~\ref{ass:1} implies $\us_{\bq}>0,$ $\us_{\br}>0$ and  Assumption~\ref{ass:2} implies $\mu>0.$ 


Let $\left\{x_t\right\}_{t=0}^{\infty}$ be a trajectory generated by policy matrix $L$. 
We can write the state covariance matrix as 
\begin{equation}
    \Sigma_t:=\E\left[x_t x_t\T\right] 
\end{equation}
and aggregate state covariance matrix as  
\begin{equation}
    \Sigma_L:=\E\left[\sum_{t=0}^{\infty}x_t x_t\T\right]=\sum_{t=1}^{\infty}\Sigma_t.
\end{equation}

We can then state the following convergence result for the model-based setting:
\begin{theorem}[Global Convergence of Model-based Policy Gradient]
	\label{thm:mb}
	Assume Assumptions~\ref{ass:1} and~\ref{ass:2} hold, and let $L_0\in U_{ad}$. 
	Then for any $\epsilon>0$, the model-based policy gradient update
	\begin{equation}\label{eq:mbg}
        L_{k+1}=L_k-\eta \nabla C(L_k)	
    \end{equation}
	with a constant step size $0<\eta\leq \eta^*$ 
	converges globally to the optimal matrix $L^*$ at a linear rate:
	\[
	C(L_{k+1})-C(L^*)\leq \left(1-2\eta \frac{\mu^2\us_{\br}}{\|\Sigma_{L^*}\|}\right)	(C(L_k)-C(L^*)),
	\]
	where $\eta^* = \cc_{A,B,Q,R,\mu} \left(1+C(L_0)\right)^{-5}$.

	Moreover, if the number of iteration $k$ satisfies
	\[
	k \geq  \frac{\|\Sigma_{L^*}\|}{2\eta\mu^2\us_{\br}}	\log\frac{C(L_0)-C(L^*)}{\epsilon}, 
	\] 
	then the model-based policy gradient method~\eqref{eq:mbg} enjoys the following performance bound:
	\[
	C(L_k)-C(L^*)<\epsilon.	
	\]
\end{theorem}
Theorem~\ref{thm:mb} demonstrates that despite the non-convexity of problem~\eqref{costl}-\eqref{feal},  the model-based gradient method~\eqref{eq:mbg} still achieves global convergence at linear rate.
However, the algorithm~\eqref{eq:mbg} requires knowledge of the distributions of model parameters to calculate the gradient at each iteration, which can be computationally expensive or practically unrealistic.
Although the model-based algorithm is not practical enough for real-world applications, this algorithm serves as a crucial building block for developing the model-free algorithm in next section. 
\subsection{Model-free policy gradient method}\label{sec:m2}
In this section, the statistical information of model parameters is unknown. 
The controller can only access the model through simulation or experiments. 
To address this challenge, we employ zeroth-order optimization and use a model-free estimator $\widehat{\nabla C}(L_k)$ to approximate the genuine gradient $\nabla C(L_k)$ with arbitrary accuracy.
Although some perturbation occurs, we show that for the random parameter LQ problem, the model-free policy gradient method still achieves global convergence with high probability. 
Moreover, the sample and computation complexities are both polynomial. 
 
 In each step $k$ the policy matrix is updated as:
 \begin{equation}
    \label{eq:samplebasedgradientmethod}
    L_{k+1}=L_k-\eta \widehat{\nabla C}(L_k),
 \end{equation}
 where $\widehat{\nabla C}(L_k)$ is generated by calling Algorithm~\ref{alg:1}.
 \begin{algorithm}[H]
	\caption{Model-Free Policy Gradient Estimation}
	\label{alg:1}
	\renewcommand{\algorithmicrequire}{\textbf{Input:}}
	\renewcommand{\algorithmicensure}{\textbf{Output:}}
	\begin{algorithmic}[1]
		\REQUIRE Current feedback matrix $L_k$, number of samples $N$, rollout length $l$, exploration radius $r$ 
		\FOR{$i=1,\ldots, N$}
		\STATE Generate a sample feedback matrix $\widehat{L}_k^{(i)}=L_k+U_i$ where $U_i$ is drawn uniformly at random over matrices with Frobenius norm $r$.
		\STATE Simulate $l$ steps trajectory $\{x_t^{(i)},c_t^{(i)}\}_{t=0}^{l-1}$ using linear feedback policy $\widehat{L}_k^{(i)}$ with $x_0^{i}\sim \mathcal{D}.$ Let $\widehat{C}_i$ be the empirical cost estimator 
		\[
		\widehat{C}_i=	\sum_{t=0}^{l-1}c_t^{(i)}
		\] 
		\ENDFOR
		\ENSURE  Model-free Gradient 
		\begin{equation}
            \label{eq:zoo}
            \widehat{\nabla C} (L_k)=	\frac{1}{N}\sum_{i=1}^N\frac{mn}{r^2}\widehat{C}_iU_i.
        \end{equation}
	\end{algorithmic}  
\end{algorithm}

As described in~\cite{fazel2018},  zeroth-order optimization is applied to approximate the genuine gradient term in~\eqref{eq:mbg}. 
This method allows us to access the gradient term using zeroth-order information, i.e. values of the cost function, which can be easily obtained by experiment or simulation. 
However, as the admissible set $U_{ad}$ is not the 
entire space $\R^{n\times m}$, the classical Gaussian-smoothing technique 
cannot be applied directly. This is avoidable by smoothing over the sphere of a ball with Frobenius norm $r$. 
We can show that, with suitable choice of step-size $\eta$, 
the updated policy $L_k$ remains in the interior 
of $U_{ad}$ at each iteration, making local smoothing well-defined.   

In order to guarantee global convergence of the model-free gradient method, additional conditions on the distributions of the initial state $x_0$ and model parameters $\left[\Lambda\T, N\right]$ are necessary. First, we introduce the following definitions:
\begin{definition}  
	We introduce the following definitions for random variables and vectors:
        \begin{enumerate}
            \item A centered random variable $X$ is said to be sub-Gaussian with parameter $\sigma$, denoted by $X\sim SG(\sigma^2)$, if its moment generating function satisfies $\E[e^{\lambda X}]\leq \exp(\sigma^2 \lambda^2/2 ),$ for any $\lambda\in \R$. 
            \item A centered random variable $X$ is said to be sub-exponential with parameter $(v,\alpha)$, denoted by $X\sim SE(v^2,\alpha)$, if its moment generating function satisfies $\E[e^{\lambda X}]\leq \exp(v^2 \lambda^2/2),$ for any $|\lambda|<1 / \alpha.$
            \item A centered random vector $X\in \R^n$ is said to be sub-Gaussian, denoted by $X\sim SG(\sigma^2)$, if for any $u\in\R^n$ with $\|u\|_2=1$, the standard Euclidean inner product $\left<u,X\right>$ is a sub-Gaussian random variable with parameter $\sigma$. 
        \end{enumerate} 
\end{definition}
\begin{remark}
    The above definitions can be extended to non-centered random variables or vectors. 
	Specifically, a random variable or vector $X$ is said to be sub-Gaussian (or sub-exponential) if $X-\E\left[X\right]$ is sub-Gaussian (or sub-exponential).
\end{remark}
The Sub-Gaussian distribution is a versatile class of probability distributions that covers a wide range of commonly used distributions, including the normal distribution, symmetric Bernoulli distribution, and all distributions with finite support. 
As a result, it finds applications in various fields, such as statistics, machine learning, economics, signal processing, and analysis of algorithms. 
A notable property of the Sub-Gaussian distribution is the concentration property, which is utilized to bound the probability of large deviations from the mean of a random variable. 

In this context, we leverage the concentration property to ensure that the gradient estimator~\eqref{eq:zoo} concentrates around the exact gradient term. 
Consequently, when the exact gradient term is replaced by the gradient estimator~\eqref{eq:zoo}, the convergence result of the exact gradient method in Theorem~\ref{thm:mb} is slightly affected. 
To achieve this goal, we make the following assumption regarding the initial state and model parameters: 
\begin{assumption}$\quad$
\label{ass:3}
    \begin{enumerate}
        \item The initial state $x_0$ is sub-Gaussian random vector with parameter $\sigma_0$. 
        \item The random matrix $\left[\Lambda\T, N\right]$ are almost surely bounded.
    \end{enumerate}
\end{assumption}
Assumption~\ref{ass:3} is made to guarantee the sub-Gaussian 
property of the state process, which is crucial for the analysis of concentration property in model-free setting. 
Specifically, the boundedness assumption on model parameters 
enables the sub-Gaussian property to pass from the 
initial state to subsequent states. 

We now present our main theorem:
\begin{theorem}
    \label{thm:samplebased}
    Assume Assumptions~\ref{ass:1},~\ref{ass:2} and~\ref{ass:3} hold
    and $L_0\in U_{ad}$. For any $\epsilon>0$ and failure probability 
	$\delta>0$, the model-free policy 
    gradient method as described in~\eqref{eq:samplebasedgradientmethod}
    \[
    L_{k+1}=L_k-\eta \widehat{\nabla C}(L_k)
	\]
    with $0<\eta<\eta^*$ and $\widehat{\nabla C}(L_k)$ is generated by calling Algorithm~\ref{alg:1} with parameters $\left(N,l,r\right)$ such that 
	$r\leq \hrgd$, $l\geq \hlgd$ and $N\geq \hs^*$, 
	where 
	$$
	\begin{aligned}
		&\eta^*= \cc_{A,B,Q,R,\mu} \left(1+C(L_0)\right)^{-5},\\
		&\hrgd=\poly(\frac{1}{1+C(L_0)})\epsilon,\\
		&\hlgd=\poly(1+C(L_0))\frac{1}{\epsilon r},\\
		&\hs^*=\poly(1+C(L_0))^{4l}\frac{l^2}{\epsilon^2 r^2}\left(\log\log(1/\epsilon)+\log(1/\delta)+\log(1/\eta)+\log(l)\right).
	\end{aligned}
	$$

	Then, with at least $1-\delta$ probability, at most at the ($k+1$)-th step, we have the following performance guarantee:
	 \[
	 C(L_{k+1})-C(L^*)<\epsilon,
	 \] 
	where 
	\[
		k= \left\lceil \frac{\|\Sigma_{L^*}\|}{\eta\mu^2\us_{\br}}	\log\frac{C(L_0)-C(L^*)}{\epsilon}\right\rceil.	
	\]
	
\end{theorem}
Theorem~\ref{thm:samplebased} demonstrates that, with high probability, the model-free policy gradient method achieves an accuracy level of $\epsilon$ at a linear rate. The three assumptions required by the theorem are easy to verify in practice.
In particular, the verification of assumption~\ref{ass:3} is easier compared to those technical conditions that are deliberately introduced to circumvent the probabilistic analysis related to concentration inequalities.  

It's noteworthy that the existence of parameters $(\eta^*,\hs^*,\hlgd,\hrgd)$ only plays a theoretical role. 
In the practice of model-free algorithms, 
the exact values of these parameters can't be calculated explicitly,
because the calculation involves the unknown model parameters. 
Therefore, in the following numerical experiments, these parameters are obtained through parameter tuning.

Regarding sample complexity, for a given accuracy level $\epsilon$ and failure probability tolerance $\delta$, 
the parameters $\left(N,l,r,k\right)$ in Theorem~\ref{thm:samplebased} are polynomial in $C(L_0)$, $\eta$, $\epsilon$, $\delta$, and model parameters. 
Specifically, when initial policy $L_0$ is fixed, 
in order to achieve the performance $\Prob(C(L_{k+1})-C(L^*)<\epsilon)>1-\delta$   
the exploration radius $r$ is $\mathcal{O}(\epsilon)$, 
the rollout length $l$ is $\mathcal{O}(1/\epsilon^2)$, 
the total number of sample  required for gradient estimation $N$ is 
$\mathcal{O}(\epsilon^{-8}\left(\log\log(1/\epsilon)+\log(1/\delta)+\log(1/\eta)+\log(1/\epsilon^2)\right))$, 
and the number of iteration $k$ is $\mathcal{O}((1/\eta) \log(1/\epsilon))$.   
Note that the number of sample trajectories $N$ and the step-size upper bound $\eta^*$ given by the theoretical analysis can be rather conservative and far from optimal. 
However, the sample complexity provided by our theorem is necessarily greater than existing results. 
This is because of the presence of randomness in both the transition dynamics and quadratic criteria and because of the lack of structural information about the noise to exploit.  
The proof of Theorem~\ref{thm:samplebased} is developed in Section~\ref{sec:samplebased}.

\section{Numerical experiments}\label{sec:num}
In this section we illustrate the effectiveness of our main results with numerical experiments. 
The first experiment is conducted on a system with $n=3$ and $m=2$.  
We not only demonstrate the convergence of our algorithms 
but also explore the impact of parameters on the numerical behavior.
Furthermore, we compare the numerical behavior when adaptive step sizes are employed.  
Finally, higher dimensional experiments with $n=20$ and $m=10$ is presented.  
The experiments were conducted on a PC laptop with single-core Intel i5 10400 2.9GHz CPU, 8-GB RAM; no GPU computing was utilized.

We use the following relative error to evaluate the performance of a given policy matrix $L$:
\[
\text{Relative error}=\frac{|C(L)-C(L^*)|}{C(L^*)}	
\]
where the baseline optimal cost $C(L^*)$ was computed by solving the ARE~\eqref{eq:ARE} using fix point iteration with high precision. 

\subsection{Low dimensional experiment}
\noindent {\bf Set up:} We consider the LQ problem~\eqref{cost}-\eqref{sys} with $n=3$, $m=2$. 
The system parameters are given by
\[
\begin{aligned}
A_t &= \begin{bmatrix}
    0.7&0.3&0.2\\
    -0.2&0.4&0.5\\
    -0.4&0.2&-0.3
\end{bmatrix} + \xi_t^{(1)}\begin{bmatrix}
    1&0&0\\
    0&0&0\\
    0&0&1
\end{bmatrix} + (\xi_t^{(1)})^2\begin{bmatrix}
    0&1&0\\
    0&0&0\\
    0&1&0
\end{bmatrix} + \xi_t^{(2)}\begin{bmatrix}
    1&0&0\\
    0&0&1\\
    0&0&0
\end{bmatrix} + (\xi_t^{(2)})^3\begin{bmatrix}
    0&0&1\\
    0&1&0\\
    1&0&0
\end{bmatrix},\\
B_t & = \begin{bmatrix}
    0.5&-0.3\\
    0.8&0.3\\
    0.1&0.9
\end{bmatrix} + \omega_t^{(1)}\begin{bmatrix}
    1&0\\
    0&1\\
    1&0
\end{bmatrix} + (\omega_t^{(1)})^2\begin{bmatrix}
    1&0\\
    1&0\\
    1&0
\end{bmatrix} + (\omega_t^{(1)})^3\begin{bmatrix}
    0&0\\
    0&1\\
    1&1
\end{bmatrix} + (\omega_t^{(2)})^2\begin{bmatrix}
    0&1\\
    1&1\\
    0&0
\end{bmatrix},\\
Q_t &= I_3 + \zeta_t^2\begin{bmatrix}
    1&0&0\\
    1&1&0\\
    0&0&1
\end{bmatrix} + 
\zeta_t^3
\begin{bmatrix}
    0&0&0\\
    0&1&1\\
    0&1&0
\end{bmatrix},\\
R_t &= I_2 + \psi_t 
\begin{bmatrix}
    0&1\\
    1&1
\end{bmatrix} + \psi_t^2 
\begin{bmatrix}
    0&1\\
    1&0
\end{bmatrix}.
\end{aligned}
\]
where the noise terms are given by
\begin{equation}\label{eq:noise}
\begin{aligned}
&\xi_t^{(1)}\overset{i.i.d}{\sim}\unif([-0.01,0.01]),\quad \xi_t^{(2)}\overset{i.i.d}{\sim}\unif([-0.012,0.012]),\\
&\omega_t^{(1)}\overset{i.i.d}{\sim}\unif([-0.015,0.015]),\quad \omega_t^{(2)}\overset{i.i.d}{\sim}\unif([-0.011,0.011]),\\
&\zeta_t\overset{i.i.d}{\sim}\unif([-0.015,0.015]),\quad \psi_t\overset{i.i.d}{\sim}\unif([-0.011,0.011]).
\end{aligned}
\end{equation} 
We choose the noise terms from uniform distributions $\mathbb{U}$, which ensures that the boundedness assumption~\ref{ass:3} is satisfied. 

It is worth noting that this example cannot be accommodated within the framework proposed by Gravell et al.~\cite{gravell2020}. 
Not only because of the presence of noise terms in $Q_t$ and $R_t$, 
but also because the random noise terms in their setting are required to be mutually independent, while in our case, the dependency is allowed. 
Moreover, the randomness can even be completely structureless. 
We consider this setting for the sake of programming convenience. 
\subsubsection{Constant step size}
\begin{figure}[htbp]
	\centering
	\begin{subfigure}[b]{0.45\textwidth}
        \includegraphics[width=\textwidth]{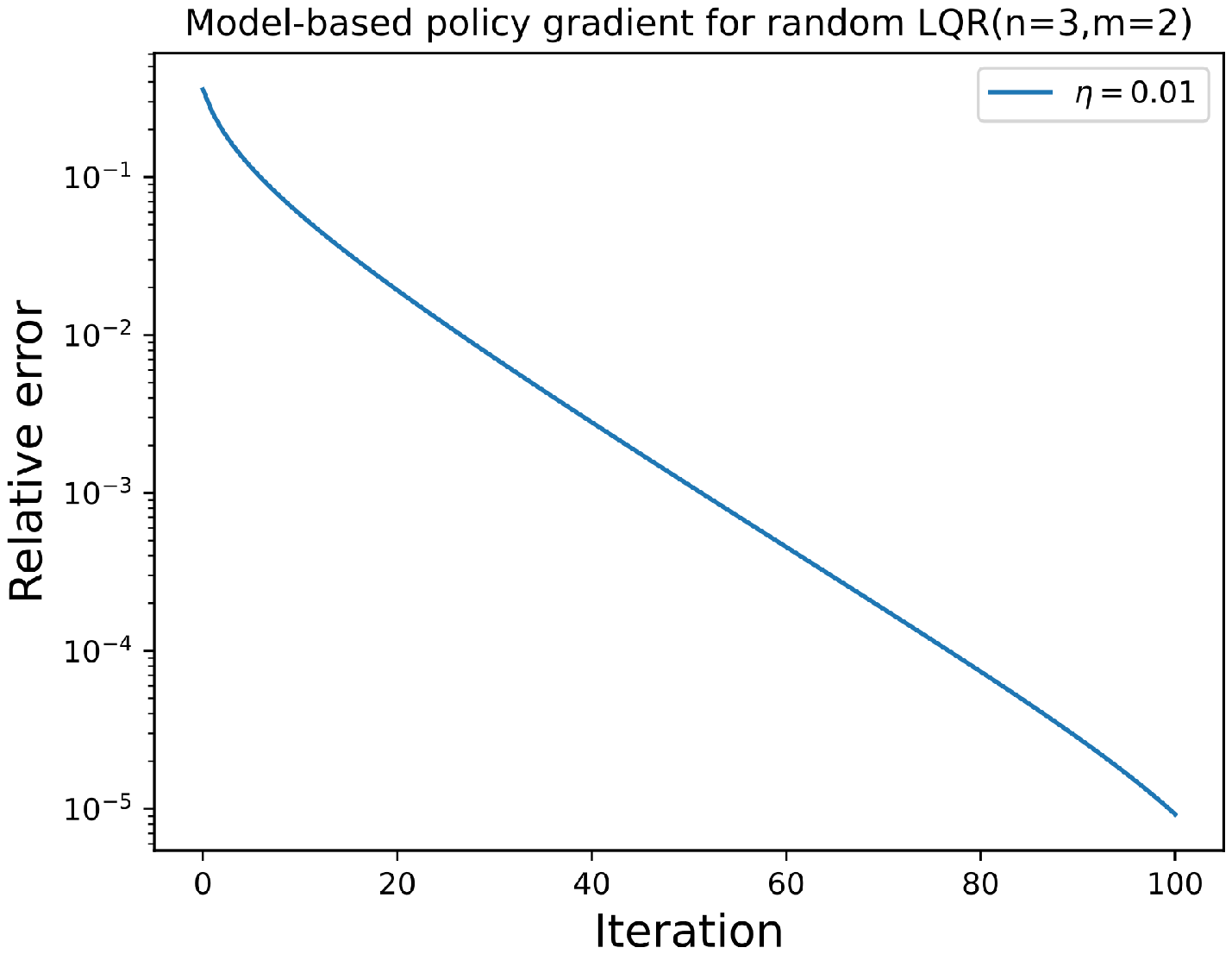}
        \caption{Model-based policy gradient method($\eta = 10^{-2}$)}
        \label{fig:sub1}
    \end{subfigure}
    \begin{subfigure}[b]{0.45\textwidth}
        \includegraphics[width=\textwidth]{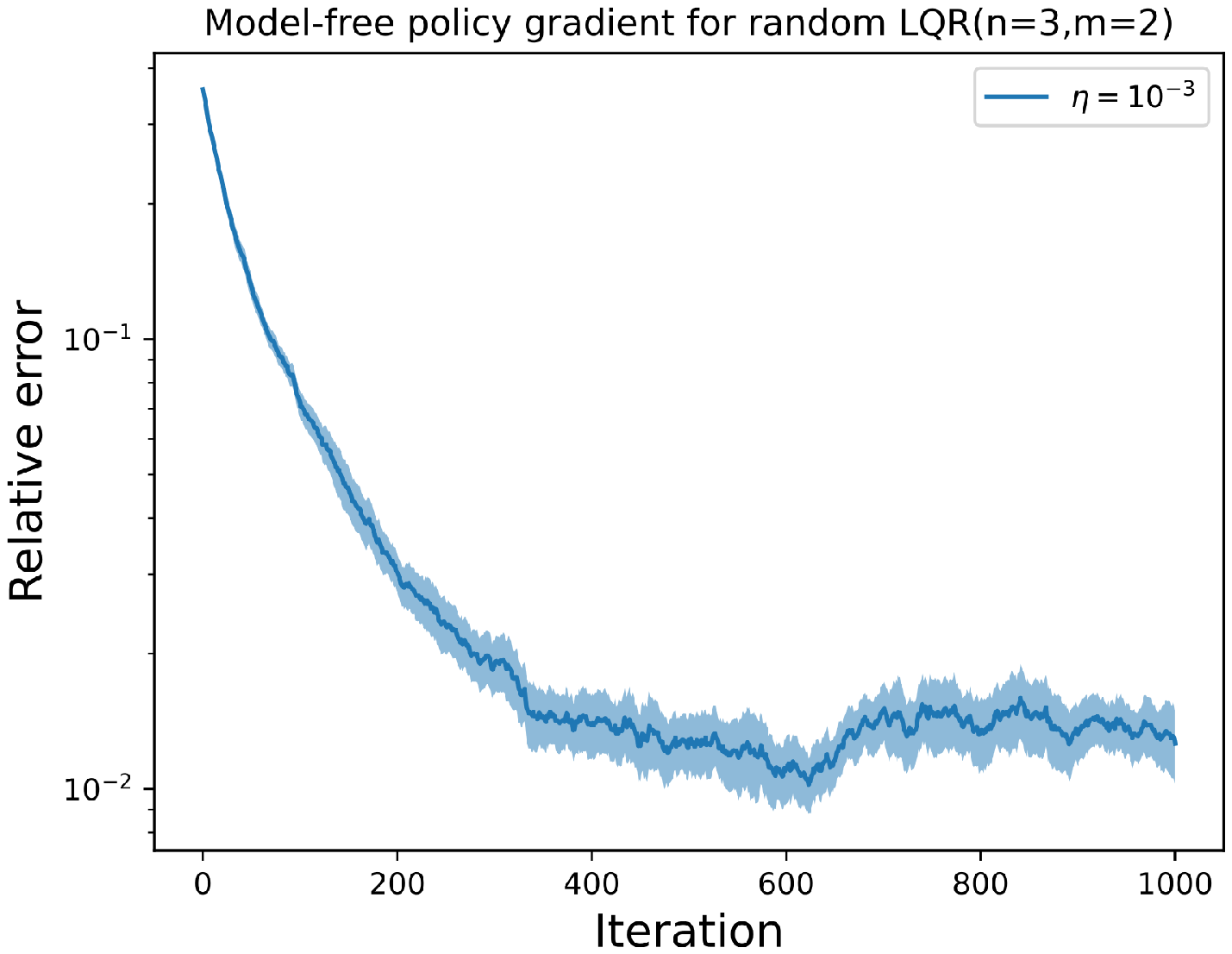}
        \caption{Model-free policy gradient method($\eta = 10^{-3}$)}
        \label{fig:sub2}
    \end{subfigure}
	\caption{Relative error vs. iteration during policy gradient method on the $n=3, m=2$ example.}
	\label{fig:1}
\end{figure}
\noindent{\bf The model-based experiment: } 
The step size was set to $\eta=1\times 10^{-2}$, 
and we executed a fixed number of iterations, which was $100$.

\noindent{\bf The model-free experiment: } 
In this case, we set the step size of gradient method to $\eta=1\times 10^{-3}$.
We estimated the model-free gradient with an exploration radius of $r=0.1$ and 
conducted $N=500$ rollouts per iteration. 
For each rollout, we set the rollout length to $l=30$ and the initial state $x_0\sim \mathcal{N}(0,I_3)$. 
We simulated $30$ trajectories. The shaded area represents $95\%$ confidence interval.

\noindent{\bf Convergence: } 
The results of our numerical experiments are presented in Fig.~\ref{fig:1}. 
The model-based method shows rapid linear convergence within $100$ iterations, 
while the model-free method exhibits slower convergence and more fluctuation.
Both the model-based and model-free methods converge to a reasonable neighborhood of the optimal policy at a linear rate, 
which is consistent with our theoretical analysis. 

\subsubsection{Impact of parameters}
In this subsection, we investigate  
the impact of varying parameters on the numerical behavior.  
When analyzing a specific kind of parameter, all other parameters are held unchanged. 
In the policy gradient experiments presented in Fig.~\ref{fig:var}, for each parameter setting, we simulated $10$ trajectories. 
The shaded area represents $95\%$ confidence interval. 
In the gradient estimation experiments recorded in Table~\ref{tb:1}, 
for each $N$, we estimated gradient $100$ times and recorded the mean CPU time consumption and the mean relative error.

\begin{figure}[htbp]
	\centering
	\begin{subfigure}[b]{0.45\textwidth}
        \includegraphics[width=\textwidth]{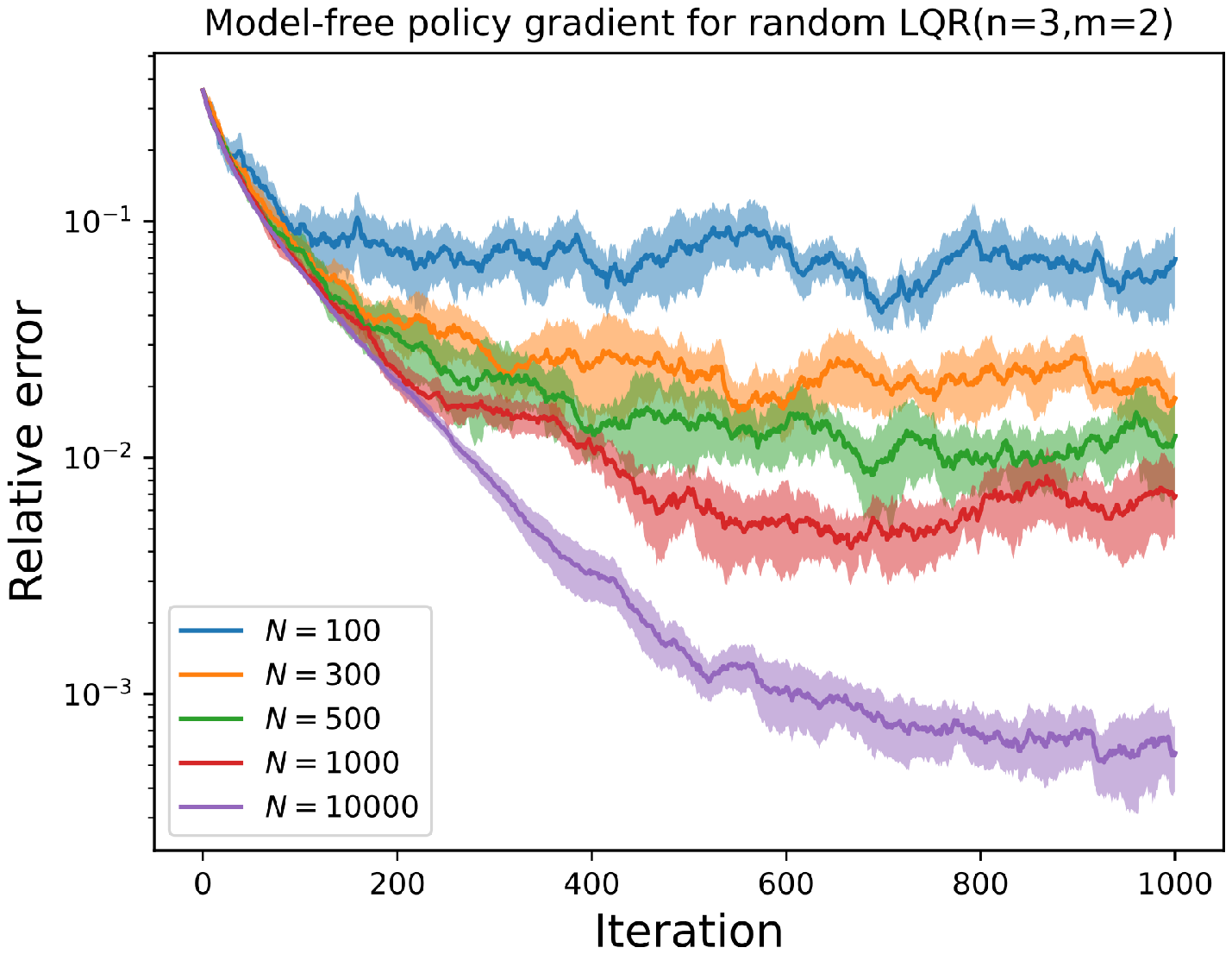}
        \caption{Under different $N$ but same $\eta=10^{-3}$}
        \label{fig:varn}
    \end{subfigure}
    \begin{subfigure}[b]{0.45\textwidth}
		\includegraphics[width=\textwidth]{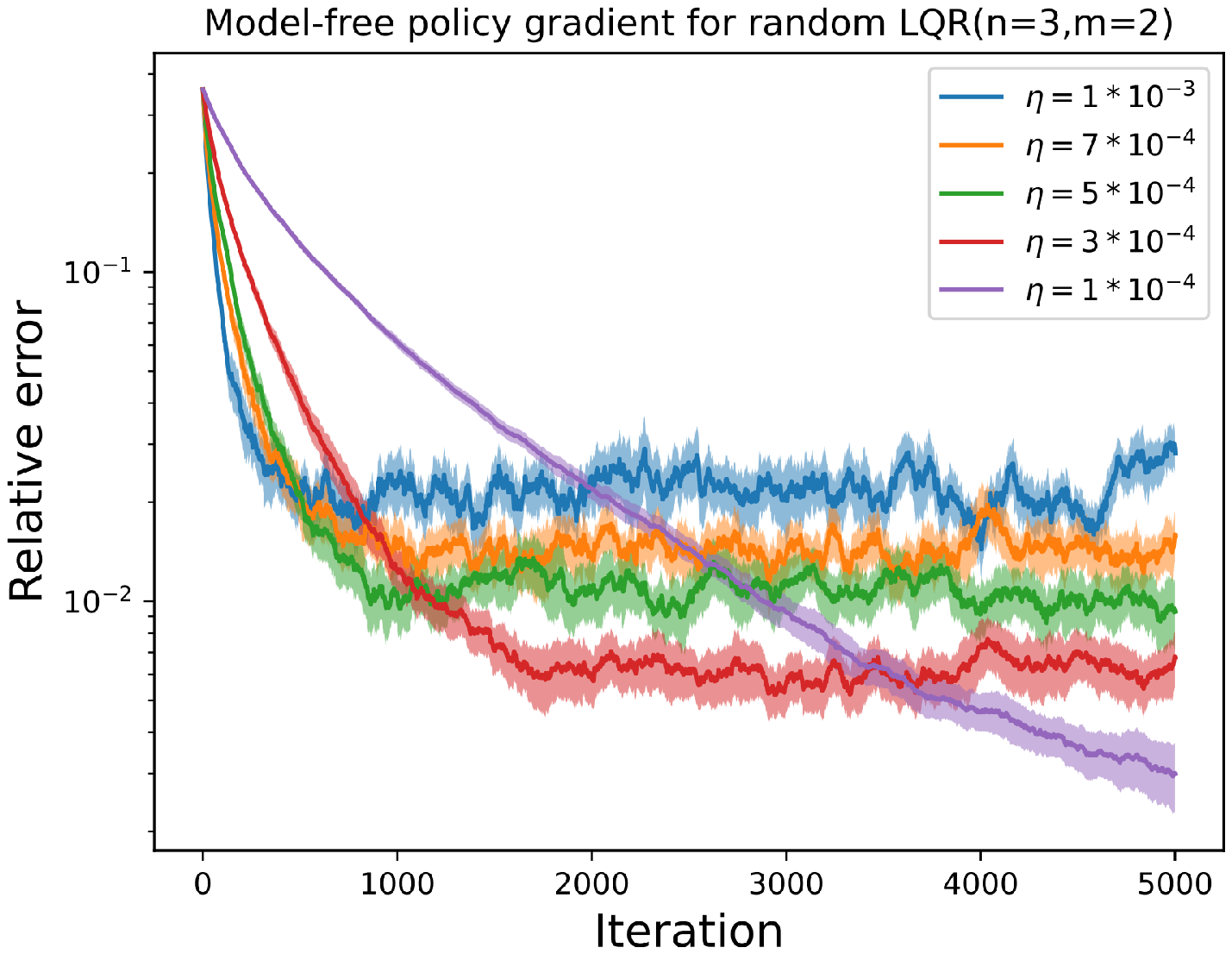}
        \caption{Under different $\eta$ but same $N=500$}
        \label{fig:vareta}
    \end{subfigure}
	\caption{Numerical performance of varying parameters}
	\label{fig:var}
\end{figure}
\begin{table}[htbp]
    \centering
    \begin{tabular}{|c|c|c|c|c|c|}
		\hline
		$N$ & $100$ & $1000$ & $10000$ & $100000$ & $1000000$   \\
		\hline
		CPU time(s) & 0.11092 & 1.08474 & 10.85027 & 110.35008 & 1112.69388  \\
		\hline
		$\|\widehat{\nabla C} - \nabla C\|/\|\nabla C\|$ & 6.00636 & 1.96197 & 0.57168 &  0.20274  &  0.05614 \\
		\hline
	\end{tabular}
	\caption{Impact of sample size $N$ on gradient estimation}
	\label{tb:1}
\end{table}

\noindent{\bf Impact of sample number \texorpdfstring{$N$}{}:  } 
We observe from Fig.~\ref{fig:varn} that 
the larger number of samples $N$ for estimating the gradient, the more accurate the algorithm is. 
However, it does not suggest that the larger $N$, the better numerical behavior. 
We can see from Table~\ref{tb:1}. 
As the number of samples increases, the gradient estimation becomes more accurate, 
but the CPU time required per iteration also increases almost in proportion to $N$.

\noindent {\bf Impact of step size  \texorpdfstring{$\eta$}{}:  }
Since the number of sample $N$ required for estimating the gradient is unchanged.  
Thus, based on the above discussion, the number of iterations is directly proportional to the CPU time consumption.
As depicted  in Fig.~\ref{fig:vareta}, a larger the step size leads to a steeper slope of the relative error with respect to the iteration. 
And we observe that the model-free policy gradient algorithm with smaller step size 
exhibit higher precision but slower convergence.  
This phenomenon aligns with our theoretical results 
and is a common trait in machine learning algorithms. 
A smaller step size enhances convergence accuracy, 
but it also increases the number of iterations, 
resulting in more CPU time consumption to attain stability.  

\subsubsection{Adaptive Step size}
In this subsection we show the numerical performance of the model-free policy gradient method using adaptive step sizes.  
All experiments illustrated in Fig.~\ref{fig:diminish} keep a consistent number of  $N=500$ for estimating gradient. 
For each type of step size, $10$ trajectories were simulated. The shaded area reflects the $95\%$ confidence interval. 
In the experiment using diminishing step size, our step size varies according to $\eta_k = 1/(500+5k)$, where $k$ is the number of iterations. 
In the experiment using backtrack line search, our criterion for step size contraction uses the Armijo type criterion:
\(
C(L_k-\eta\widehat{\nabla C}(L_k) )	\leq C(L_k)-0.01\eta\|\widehat{\nabla C}(L_k)\|^2.
\)
The CPU times in Table~\ref{tb:2} represent the average time taken for executing the first $100$ iterations. 
The relative error in Table~\ref{tb:2} reflects the average accuracy achieved after executing the first $100$ iterations.
\begin{figure}[htbp]
	\centering
    \includegraphics[width=0.6\textwidth]{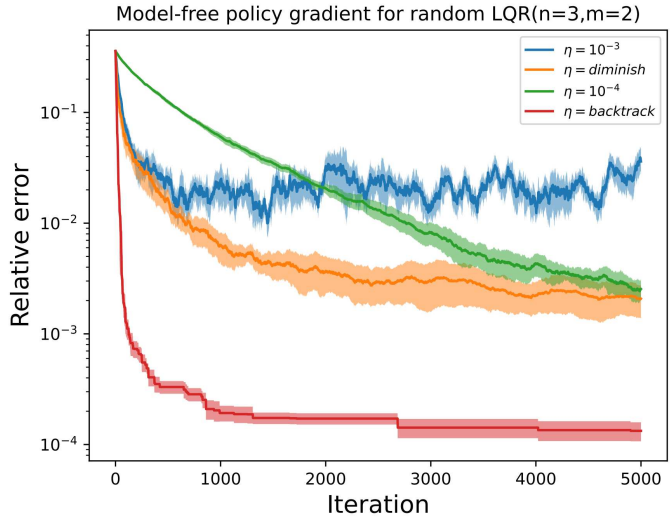}
    \caption{Numerical performance when using adaptive step size($N=500$)}
    \label{fig:diminish}
\end{figure}
\begin{table}[htbp]
    \centering
    \begin{tabular}{|c|c|c|c|}
		\hline
		Selection of $\eta$ & constant $\eta=0.001$ & diminishing step & backtrack line search    \\
		\hline
		CPU time(s) & 54.72431& 55.56956 & 57.28105   \\
		\hline
		$|C(L_{100}) - C^*|/|C^*|$ & 0.07059& 0.05535 & 0.00063  \\
		\hline
	\end{tabular}
	\caption{Numerical performance of different step sizes within first $100$ iterations}
	\label{tb:2}
\end{table}

We can see from Fig.~\ref{fig:diminish} that both algorithms using adaptive step size 
exhibit better numerical performance  compared to those using fixed step size. 
The algorithm using diminishing step size shows the combined benefits of rapid convergence for large step sizes and high accuracy for smaller step sizes.

In the experiment of using backtrack line search, 
due to the protection of Armijo type criterion, the default step size can be chosen to be relatively large $\eta_{\text{default}}=0.01$. 
Consequently, the algorithm using backtrack line search decreases faster at the beginning of the optimization process compared to other algorithms. 
Moreover, the relative error is monotone decreasing with fewer fluctuation, due to the safeguarding of the Armijo type criterion.

Although the line search process adds to the CPU time in each iteration, 
we observe from Table~\ref{tb:2} that within the first $100$ iterations, 
despite the additional $2$-$3$ seconds of CPU time,  
the line search process significantly improves accuracy by a factor of 100.

\subsection{Higher dimensional experiment}
\begin{figure}[htbp]
	\centering
	\begin{subfigure}[b]{0.45\textwidth}
		\includegraphics[width=\textwidth]{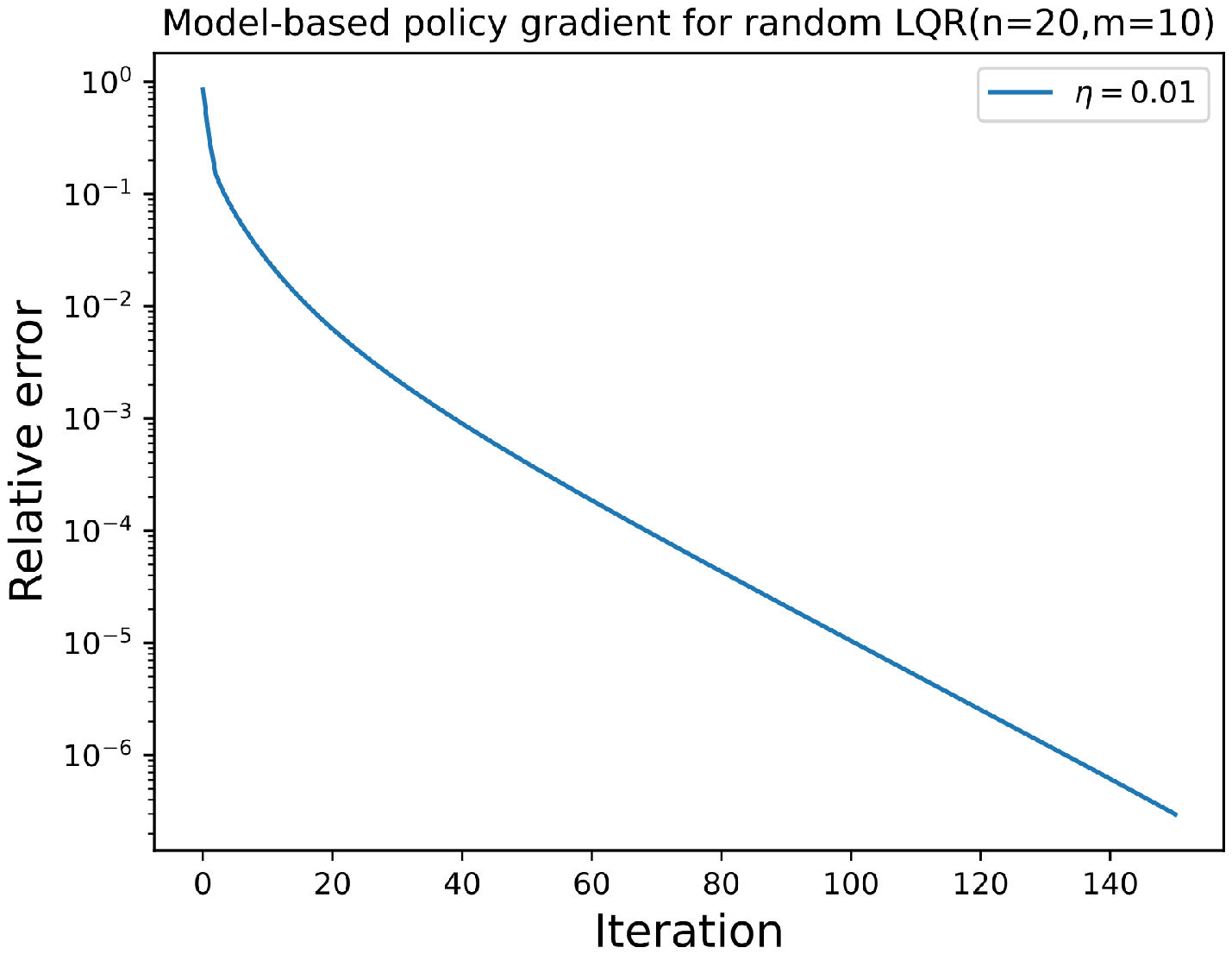}
		\caption{Model-based policy gradient method($\eta = 10^{-2}$)}
		\label{fig:sub11}
	\end{subfigure}
	\begin{subfigure}[b]{0.45\textwidth}
		\includegraphics[width=\textwidth]{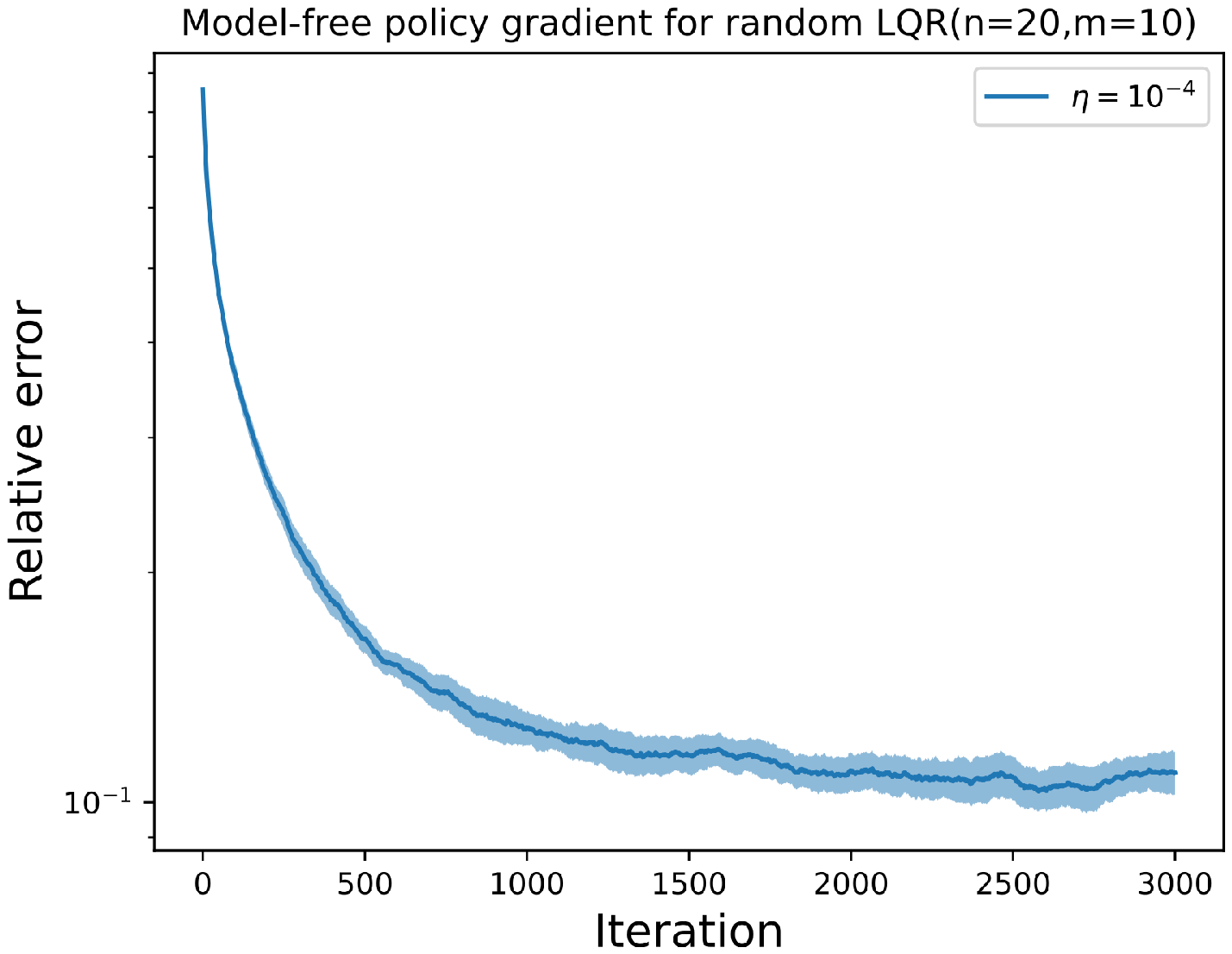}
		\caption{Model-free policy gradient method($\eta = 10^{-4}$)}
		\label{fig:sub22}
	\end{subfigure}
	\caption{Relative error vs. iteration during policy gradient method on the $n=20, m=10$ example.}
	\label{fig:2}
\end{figure}

\noindent {\bf Set up: }  The explicit system parameter matrices are omitted due to space constraint. 
In the higher dimensional case, we use the same noise structure as in lower dimensional case. 

\noindent{\bf The model-based experiment: } 
The step size was set to $\eta=1\times 10^{-2}$, 
and we executed a fixed number of iterations, which was $150$.

\noindent{\bf The model-free experiment: } 
In this case, we set the step size to $\eta=1\times 10^{-4}$.
We estimated the model-free gradient with an exploration radius of $r=0.1$ and 
conducted $N=500$ rollouts per iteration. 
For each rollout, we set the rollout length to $l=30$ and the initial state are drawn such that $x_0\sim \mathcal{N}(0,I_{20})$. 

\noindent{\bf Convergence: } 
Fig.~\ref{fig:2} shows that both model-based and model-free methods converge to a reasonable neighborhood of the optimal value in the higher dimensional example. 
Due to the impact of dimension, it can be observed that when the number of samples $N$ for gradient estimation keep unchanged, 
the algorithm in the higher dimensional setting can not achieve the same level of precision as in the low dimensional setting. 

\section{Convergence analysis of model-based algorithm}\label{sec:CAOMBA}
This Section aims to prove Theorem~\ref{thm:mb}. 
Firstly, we analyze two key properties of the LQ problem~\eqref{cost}-\eqref{sys} in Section~\ref{sec:cigd}. 
Secondly,  in Section~\ref{sec:paos} we control the perturbation of $\Sigma_L$ and show that the admissible set $U_{ad}$ is an open set.  
Finally, we provide the proof of Theorem~\ref{thm:mb} in Section~\ref{sec:caombm}

In the following we always assume Assumptions~\ref{ass:1},~\ref{ass:2} hold and $L\in U_{ad}$ unless otherwise stated. 
Due to space constraints, the proofs of the lemmas in this section are deferred to Appendix~\ref{sec:tech}.
\subsection{Gradient dominance and almost smoothness}\label{sec:cigd}
In this subsection, we analyze two key properties of the cost function $C(L)$: {\it gradient dominance} and {\it almost smoothness}.

The gradient dominance, also known as the Polyak-\L{}ojasiewicz condition , bounds the difference between $C(L)$ and the optimal cost $C(L^*)$ by $\|\nabla C(L)\|_F$, ensuring that all stationary points are global minimizers even in non-convex settings. 
The almost smoothness is a property similar to gradient Lipschitz continuity, but with a small deviation;     
gradient Lipschitz continuity enables the gradient-based methods to search for the global minimum efficiently. 
Therefore, the gradient dominance and almost smoothness of the cost function are essential for the convergence of our algorithms.

\begin{lemma}[Gradient domination]
    \label{lemma:gd}
	Let $L\in U_{ad}$ and $L^*$ denote the optimal policy. 
    Then, we have   
    \begin{equation}
        \label{eq:gd}
        C(L)-C(L^*)\leq \frac{\|\Sigma_{L^*}\|}{4\mu^2\us_{\br}}\|\nabla C(L)\|_F^2.	
    \end{equation}
    Moreover, the difference can be lower bounded as
    \[
        \frac{\mu}{\|R_L\|}\tr(E_L\T E_L) \leq C(L)-C(L^*)
    \]
	where $E_L=R_L L-\bbpla$ and $R_L=\rl$.
\end{lemma}


We say that $f$ is {\it gradient $\beta$-Lipschitz} continuous if
\begin{equation}
	\label{eq:c1}
	f(x)-f(y)-\langle\nabla f(y),x-y\rangle 	\leq \frac{\beta}{2}\|x-y\|^2, \quad \forall x,y \in U_{ad}
\end{equation}
The following result shows that $C(L)$ almost satisfies~\eqref{eq:c1} locally in the admissible set $U_{ad}$. 
\begin{lemma}[Almost smoothness]
	\label{lemma:as}
	Let $L,L'\in U_{ad}$ and $\left\{x_t'\right\}$ be the state process generated by feedback matrix $L'$ with $x_0'=x_0$. 
	Then, we have 
	\begin{equation}
        \label{eq:as}
        C(L')-C(L)=2\tr(\Sigma_{L'}(L'-L)\T E_L)+\tr(\Sigma_{L'}(L'-L)\T R_L (L'-L)), 	
    \end{equation}
    where $\Sigma_{L'}=\E\left[\sum_{t=0}^{\infty}x_t'(x_t')\T\right]$.
\end{lemma} 
\begin{remark}[Explanation of the name]
    \label{rk:1}
    As appointed in~\cite{fazel2018,gravell2020}, we call Lemma~\ref{lemma:as} `almost smoothness' because if $\Sigma_L$ is stable under the perturbation of $L$, i.e.,  $\Sigma_{L'} \approx \Sigma_L$ when $L'$ is close to $L$, then we obtain  
\[
    \begin{aligned}
    C(L')-C(L)&=2\tr(\Sigma_{L'}(L'-L)\T E_L)+\tr(\Sigma_{L'}(L'-L)\T R_L (L'-L))\\
    &\approx 2\tr(\Sigma_{L}(L'-L)\T E_L)+\tr(\Sigma_{L}(L'-L)\T R_L (L'-L))\\
    &=\tr((L'-L)\T \nabla C(L))+\tr(\Sigma_{L}(L'-L)\T R_L (L'-L))\\
    &\leq \left<\nabla C(L), (L'-L)\right>+\frac{2\|\Sigma_{L}\| \|R_L\| }{2}\|L'-L\|_F^2,
    \end{aligned}	    
\]   
which is exactly~\eqref{eq:c1} with $\beta(L) = 2\|\Sigma_{L}\| \|R_L\|$. 

Although $\beta(L)$ may blow up as $L$ approaches $\partial U_{ad}$, Lemma~\ref{cor:2} ensures that $\beta(L)$ can be bounded on any sublevel set of cost $C(L)$. 
Therefore, $C(L)$ is indeed almost gradient Lipschitz on any nonempty sublevel set as long as the stability issue of $\Sigma_L$ is addressed, which is the topic of the next subsection.
\end{remark}
 
\subsection{Perturbation analysis of  \texorpdfstring{$\Sigma_L$}{} and openness of \texorpdfstring{$U_{ad}$}{}}\label{sec:paos}
In this subsection, we  bound  the perturbation of $\Sigma_L$  as discussed in Remark~\ref{rk:1}. 
Moreover, we demonstrate that the admissible set $U_{ad}$ is an open set, 
which guarantees that any admissible $L$ has a neighborhood of matrices that are also admissible, 
allowing the gradient algorithm to take a step forward.

Let $B_{op}(L,r):=\left\{M\in \R^{m\times n} \big| \quad \|M-L\|\leq r \right\}$ denote the closed ball of radius $r$ in the operator norm centered at $L$.
The following lemma is the main result in this subsection.
\begin{lemma}
    \label{lemma:paos}
	For any $L\in U_{ad}$ fixed, the following hold:
    \begin{enumerate} 
        \item 
	The closed ball $B_{op}(L,\hdelta(L))$ is contained in $U_{ad}$, where 
    \[
		\hdelta(L):=\frac{\us_{\bq}\cdot \mu}{4(\bbs)^{1/2}\left(1+(\abs)^{1/2}\right)C(L)}.
	\]
    \item For any $L'\in B_{op}(L,\hdelta(L))$, we have  
	\[
		\begin{aligned}
			\|\Sigma_{L'}-\Sigma_L\|\leq 4\left(\frac{C(L)}{\us_{\bq}}\right)^2\frac{(\bbs)^{1/2}\left(1+(\abs)^{1/2}\right)}{\mu}\|L-L'\|
			\leq \frac{C(L)}{ \us_{\bq} }
		\end{aligned}
	\]
    \end{enumerate}
\end{lemma}
Note that a direct corollary of the first statement in Lemma~\ref{lemma:paos} is that $U_{ad}$ is an open set, 
which ensures that the gradient algorithm can choose a suitable step size for the next update, and the updated matrix remains in the admissible set.

\subsection{Proof of theorem~\ref{thm:mb}}\label{sec:caombm}
In the previous section, we demonstrated that there exists a space surrounding every admissible matrix that enables the gradient algorithm to take a step forward.  
The following lemma quantifies this step size and demonstrates that the cost function decreases in a single step update. However, it should be noted that the step size depends on the matrix $L$.
 \begin{lemma}[One step update of model-based gradient decent]
    \label{lemma:osu}
    Consider the update
	\[
	L'=L-\eta \nabla C(L)	
	\]  
	with step size 
	\begin{equation}
		\label{eq:ss}
				0<\eta\leq \frac{1}{16}\min\left\{\left(\frac{\us_{\bq}\cdot \mu}{C(L)}\right)^2\frac{1}{\|\nabla C(L)\|(\bbs)^{1/2}\left(1+(\abs)^{1/2}\right)},\frac{\us_{\bq}}{C(L)\|R_L\|}\right\}.
	\end{equation}
	Then, we have
	\begin{equation}
        \label{eq:up}
        	C(L')-C(L^*)\leq \left(1-2\eta \frac{\mu^2\us_{\br}}{\|\Sigma_{L^*}\|}\right)	(C(L)-C(L^*)).
    \end{equation}
 \end{lemma}
 \begin{remark}\label{rk:thm1}
Note that the step size $\eta$ in the model-based policy 
gradient update~\eqref{eq:gradientupdate} is constant. 
Thus, we need to find the positive lower bound $\eta^*$ for the RHS of~\eqref{eq:ss}. 
Any positive step size dominated by $\eta^*$ will satisfy the step size 
condition~\eqref{eq:ss} at each update position, thereby 
enabling the update performance~\eqref{eq:up} to be 
path independent. 
Unfortunately, the RHS of~\eqref{eq:ss} cannot be lower bounded over the entire admissible set $U_{ad}$, 
since the $C(L)$ term in the denominator is not upper bounded on $U_{ad}$.
However, since the cost $C(L_k)$ decreases during the 
optimization process, the update 
path $\{L_k\}$ will remain in the $C(L_0)$-sub-level set of cost. 
Consequently, we only need to ensure that the 
RHS of~\eqref{eq:ss} can be lower bounded on the $C(L_0)$-sub-level set of cost. 
This task can be accomplished using Lemma~\ref{lemma:bocag} and Lemma~\ref{lemma:boss}, which are provided in the appendix.
\end{remark}

\begin{proof}[Proof of Theorem~\ref{thm:mb}]
    As discussed in Remark~\ref{rk:thm1}, to establish the existence of a positive $\eta^*$ such that the step size condition~\eqref{eq:ss} 
    holds at each update position, we need to establish a positive lower bound on the right hand side of \eqref{eq:ss}. 
	The detailed proof of the existence of such positive lower bound $\eta^*= \cc_{A,B,Q,R,\mu} \left(1+C(L_0)\right)^{-5}$ 
	can be found in Lemma~\ref{lemma:eta}.


	Next, we choose $0<\eta\leq \eta^*$ and prove the theorem by induction. 

	At the first step $k=0$, by Lemma \ref{lemma:osu} we have 
	\[
		C(L_{1})-C(L^*)\leq \left(1-2\eta \frac{\mu^2\us_{\br}}{\|\Sigma_{L^*}\|}\right)	(C(L_0)-C(L^*)),
	\]
	which implies $C(L_1)\leq C(L_0)$ and the base case is true. 

	Now suppose our theorem is true at $k-1$. Inductive hypotheses gives $C(L_k)\leq C(L_0)$, which means that $L_k$ remains in $C(L_0)$-sub-level set of cost. 
	Then the step size condition in~\eqref{eq:ss} is still satisfied. 
    Thus we can again apply Lemma~\ref{lemma:osu} at $L_k$ to obtain 
	\[
		C(L_{k+1})-C(L^*)\leq \left(1-2\eta \frac{\mu^2\us_{\br}}{\|\Sigma_{L^*}\|}\right)	(C(L_k)-C(L^*)).
	\]   
\end{proof}
 \section{Convergence analysis of model-free case}\label{sec:samplebased}
 This section presents a proof of Theorem~\ref{thm:samplebased} through a sequence of lemmas. 
 First, in Section~\ref{sec:acaswfh}, we show that the cost function $C(L)$ and the aggregate covariance matrix $\Sigma_L$ can be approximated by their finite horizon counterparts. 
 Next, in Section~\ref{sec:paogac}, we demonstrate that $C(L)$ and its gradient $\nabla C(L)$ are stable under perturbation with respect to $L$. 
 Then, in Section~\ref{sec:sgosp}, we establish the sub-Gaussianity of state process $\{x_t\}$ under the Assumption~\ref{ass:3} and derive relevant concentration inequalities. 
 Moving on to Section~\ref{sec:saag},  we show that the gradient estimator generated in Algorithm~\ref{alg:1} can well approximate the exact gradient with high probability. 
 Finally, in Section~\ref{sec:pot2}, Theorem~\ref{thm:samplebased} is proved  based on the previously established lemmas. 

Unless otherwise stated, throughout this section, we assume that Assumptions~\ref{ass:1} and~\ref{ass:2} hold and that $L\in U_{ad}$.
\subsection{Approximating \texorpdfstring{$C(L)$}{} and  \texorpdfstring{$\Sigma_L$}{} with finite horizon}\label{sec:acaswfh}
In the model-free setting, only finite horizon samples can be observed or simulated. 
In this subsection, we demonstrate that infinite horizon quantities $C(L)$ and $\Sigma_L$ can be approximated by their finite horizon estimations with arbitrary accuracy, 
provided that the rollout length $l$ is sufficiently large.
 
 The finite horizon estimations of $\Sigma_L$ and $C(L)$ are defined as follows:
	\[
	\Sigma_L^{(l)}:=\E\left[\sum_{t=0}^{l-1}x_tx_t\T\right]	,\quad C^{(l)}(L):=\E\left[\sum_{t=0}^{l-1}\begin{bmatrix}x_t\\u_t\end{bmatrix}\T N_{t+1}\xut\right].
	\]  
 \begin{lemma}[Approximating $C(L)$ and $\Sigma_L$ with finite horizon rollouts]
	\label{lemma:acaswfhr}
	For any  small tolerance $\epsilon>0$, the following hold:
    \begin{itemize}
        \item If  $l\geq \frac{nC(L)^2}{\epsilon \mu\cdot (\us_{\bq})^2 }$, then 
	\[
	\|\Sigma_L-\Sigma^{(l)}_L\|	\leq \epsilon.
	\]
        \item  If $l\geq h_l(L,\epsilon)$, then 
	\[
	| C(L) - C^{(l)}(L) |	\leq \epsilon,
	\]
	where 
	\[
	h_l(L,\epsilon)	:=\frac{nC(L)^2(\|\bq\|+\bnr \|L\|^2)}{\epsilon \mu\cdot (\us_{\bq})^2}.
	\]
    \end{itemize}
\end{lemma}
The proof is deferred to Appendix~\ref{sec:po4.1}. 
\subsection{Perturbation analysis of  \texorpdfstring{$\nabla C(L)$}{} and  \texorpdfstring{$C(L)$}{}}\label{sec:paogac}
In this subsection, we demonstrate the stability of $C(L)$ and $\nabla C(L)$ under perturbation with respect to $L$. 
The proofs of Lemmas in this subsection can be found in Appendix~\ref{sec:po4.2}. 

\begin{lemma}[Perturbation of $C(L)$ ]
	\label{lemma:cp} 
	Let $\Delta(L)$ be an arbitrary positive constant such that $\|L'-L\|\leq \Delta(L) \leq \hdelta(L)$. 
	Then, the cost difference is bounded as 
	\[
	|C(L')-C(L)|\leq h_c(L, \Delta(L)) C(L)\|L-L'\|	
	\]
	where
	\[
		h_c(L, \Delta(L)):=\frac{\bnxs}{\mu\cdot \us_{\bq}}\left[2\bnr (\Delta(L)+2 \|L\|)+4\frac{C(L)}{\mu\cdot \us_{\bq}}(\bbs)^{1/2}(1+(\abs)^{1/2})(\|\bq\|+\bnr\|L\|^2)\right],
	\]
    and $h_{\Delta}(L)$ is defined in Lemma~\ref{lemma:paos}.
\end{lemma}

\begin{lemma}[Perturbation of $\nabla C(L)$ ]
	\label{lemma:gp} 
	Let $\Delta(L)$ is an arbitrary positive constant such that $\|L'-L\|\leq \Delta(L) \leq \hdelta(L)$. 
	Then, the gradient difference is bounded as 
	\begin{align}
        &\|\nabla C(L')-\nabla C(L)\|\leq \hgrad(L,\Delta(L))	\|L'-L\|,\\
        &\|\nabla C(L')-\nabla C(L)\|_F \leq \widehat{\hgrad}(L,\Delta(L))	\|L'-L\|_F, 
    \end{align}
	where 
	\[
		\begin{aligned}
			\hgrad(L,\Delta(L)):=&4\frac{C(L)}{\us_{\bq}}\left(\left(\overline{\|B\|\|A\|}+\bbs( \Delta(L)+\|L\|) \right)\frac{C(L)}{\mu\us_{\bq}}\left[2\bnr (\Delta(L)+2 \|L\|)\right.\right.\\
			& \left. \left.+4\frac{C(L)}{\mu\cdot \us_{\bq}}(\bbs)^{1/2}(1+(\abs)^{1/2})(\|\bq\|+\bnr\|L\|^2)\right]+\|\br\|+\bbs\frac{C(L)}{\mu} \right)\\
			&+8\sqrt{\frac{\|R_L\|}{\mu}(C(L)-C(L^*))}\left(\frac{C(L)}{\us_{\bq}}\right)^2\frac{(\bbs)^{1/2}(1+(\abs)^{1/2})}{\mu}
		\end{aligned}	
	\]
	and $\widehat{\hgrad}(L,\Delta(L))=\sqrt{m\land n}\cdot \hgrad(L,\Delta(L)).$
\end{lemma}

\subsection{Sub-Gaussianity of states process}\label{sec:sgosp}
Briefly speaking, the technique in model-free setting is to demonstrate that the gradient estimation will concentrate around the true gradient with high probability, 
thereby having little influence on the convergence result in the model-based setting. 
To achieve this goal, this subsection establishes the sub-Gaussianity of state process $\{x_t\}$ and derives some useful concentration inequalities. 
The lemmas in this section are proven in Appendix~\ref{sec:po4.3}. 

According to section 2.5 in~\cite{ver2018}, $X\sim SG(\sigma^2)$ if and only if the moments of $X$ satisfy  
\[\|X\|_{L^p}\leq c\sigma \sqrt{p} ,\quad \forall p\geq 1,\] where $c$ is an absolute constant. 
Thus,   Assumption~\ref{ass:3} is equivalent to the following Assumption~\ref{ass:3p}.
\begin{assumptionp}{\ref*{ass:3}$'$}
    \label{ass:3p}
    $\quad$

    \begin{enumerate}
        \item The initial state $x_0\in \R^n$ is a sub-Gaussian random vector with parameter $\sigma_0$:
        \begin{equation}
            \label{eq:ass31}
            \|\left<u,x_0-\overline{x_0}\right>\|_{L^p}\leq \sigma_0 \sqrt{p},\quad \forall p\geq 1,     
        \end{equation}
        holds for all $u\in \R^n$ with $\|u\|_2=1$.
        \item Random parameter matrices are almost surely bounded by a positive constant $K$, i.e., 
        \begin{equation}
            \label{eq:ass32}
              \max\{\|A\|,\|B\|,\|Q\|,\|R\|\}\leq K  \quad \text{a.s.}.
        \end{equation}
    \end{enumerate}
\end{assumptionp}
From this point onwards, we assume that $L\in U_{ad}$ and that Assumptions~\ref{ass:1},~\ref{ass:2},~\ref{ass:3p} hold. 

The following lemma shows that the sub-Gaussianity of the initial state $x_0$ can be propagated to the subsequent states $x_t$.
\begin{lemma}[$x_t$ is sub-Gaussian]
    \label{lemma:xtisg}
    Suppose the state process $\{x_t\}$ is generated by feedback matrix $L$. 
	Then, for each $t\in \mathbb{N}$, $x_t$ is sub-Gaussian random vector:    
	\[
        \|\left<u,x_t-\overline{x_t}\right>\|_{L^p}\leq \sigma_t(L) \sqrt{p},\quad \forall p\geq 1,
    \]
    for all $u\in \R^n$ with $\|u\|_2=1$, where 
    \[
      \sigma_t(L) =  \left[4K(1+\|L\|)\right]^{t}\left(\sigma_0+2\|\overline{x_0}\|_2\right).
    \]
\end{lemma}

According to classical textbook~\cite{ver2018},   
squared sub-Gaussian random variable is sub-exponential. 
This result can be extended to random vectors as well. 
Lemma~\ref{lemma:xmx} shows that if $x_t$ is a sub-Gaussian random vector, 
then $x_t\T M x_t$ is sub-exponential for any constant symmetric matrix $M$.
\begin{lemma}[$x_t\T M x_t$ is sub-exponential]
    \label{lemma:xmx}
	Let $\{x_t\}$ be the state process generated by feedback matrix $L$,  
	and let $M\in\SN$ be an arbitrary constant symmetric matrix. 
    Then, for each $t\in \mathbb{N}$, $x_t\T M x_t$ is sub-exponential random variable:
    $$x_t\T M x_t\sim SE((2{\rm e}\gamma_t(L,M))^2,2{\rm e}\gamma_t(L,M) ),$$
	where 
    \[
		\begin{aligned}
			&\gamma_t(L,M):=4 \|M\| \left(n \beta_t(L)+[(1+\|L\|)K]^{2t}\E[\|x_0\|_2^2]\right),\\
			&\beta_t(L):=16\left(\sigma_t(L)^{2}+\left[K(1+\|L\|)\right]^{2t}\E\left[\|x_0\|_2^2\right]\right).
		\end{aligned}
    \] 
    \end{lemma}
When controlling the probability that the model-free gradient estimator deviates from the genuine gradient, 
we can break down the task by dealing with the concentration properties of several random matrices. 
Specifically, these random matrices take the form of $\frac{1}{N}\sum_{i=1}^N (x_t^{(i)})\T M^{(i)} x_t^{(i)}U_i$, where $U_i$ are constant $m\times n$ matrices and $M^{(i)}\in \SN$. 
According to Lemma~\ref{lemma:xmx}, the quadratic form of state $(x_t^{(i)})\T M^{(i)} x_t^{(i)}$ is a sub-exponential random variable.   
This fact motivates us to establish a Bernstein-type concentration inequality for these random matrices, which is presented in the following lemma.
\begin{lemma}[Concentration of $\frac{1}{N}\sum_{i=1}^N  \xi_i U_i$]
	\label{lemma:concen}
	Let $\xi_1, \ldots, \xi_N$ be independent centered sub-exponential random variables satisfying
	\[
	\|\xi_i\|_{L^p}\leq \gamma p \quad \forall p\geq 1,
	\]
	and let $U_1, \ldots, U_N$ be constant matrices in $\R^{m \times n}$ with $\|U_i\|_F = r$.  
	Then, for any $\epsilon>0$,  there exists a Bernstein-type inequality for $\frac{1}{N}\sum_{i=1}^N \xi_i U_i$: 
	\[
		\Prob\left(\|\frac{1}{N}\sum_{i=1}^N \xi_i U_i\|\geq \epsilon\right)\leq 2(m+n)\exp\left(-\frac{\epsilon^2N}{2\left(\epsilon (2{\rm e}r\gamma)+(2{\rm e}r\gamma)^2\right)}\right). 
	\]
\end{lemma}
\begin{corollary}\label{cor:concen}
	Under the same conditions as Lemma~\ref{lemma:concen}, for any $\epsilon>0$, if 
	\[
	N\geq \hsbf(\epsilon,\delta):=\frac{2 m \land n}{\epsilon^2}\left((2{\rm e}r\gamma)^2+\frac{\epsilon}{\sqrt{m \land n}} (2{\rm e}r\gamma)\right)\log(\frac{2(m+n)}{\delta})	,
	\]
	then 
	\[
	\Prob\left(\left\|\frac{1}{N}\sum_{i=1}^N  \xi_i U_i\right\|_F > \epsilon\right)\leq \delta.	
	\]
\end{corollary}

\subsection{Smoothing and approximating gradient}\label{sec:saag}
This section analyzes the smoothing procedure and proves that the gradient term can be well approximated by the model-free gradient estimator generated in Algorithm~\ref{alg:1} with high probability.
Throughout this subsection, we assume that $L\in U_{ad}$ and that Assumptions~\ref{ass:1},~\ref{ass:2},~\ref{ass:3p} hold. 

Let $B_F(0,r)=\{X\in \R^{m\times n} : \|X\|_F\leq r\}$ denote the closed ball in $\R^{m\times n}$ in Frobenius norm centered at $0$ with radius $r$, 
and let $S_F(0,r)=\partial B_F(0,r)$ denote the surface of  $B_F(0,r)$.  
We also use $\mathbb{B}_r$ to denote the uniform distribution over $B_F(0,r)$ and $\mathbb{S}_r$ to denote the uniform distribution over $S_F(0,r)$.

Following the approach in~\cite{fazel2018}, to apply zeroth-order optimization techniques, 
the model-free algorithm performs gradient descent on a smoothed version of the cost function
\begin{equation}
    \label{eq:smoothing}
    C_r(L)=\E_{U\sim \BBbb_r}[C(L+U)].	
\end{equation}
The gradient of the smoothed cost function can be computed as
\begin{equation}\label{eq:zooo}
\nabla C_r(L)=\frac{mn}{r^2}\E_{U\sim \SSbb_r}[C(L+U)U]	.
\end{equation}

In the following, we will present a series of estimators, namely $\{\nabla C_r(L), \hd,\hdl,\td\}$,  whose randomness increases gradually. 
Note that the last estimator, $\td$, is our model-free gradient estimator.  
The remaining Lemmas in this subsection demonstrate that these estimators can effectively approximate the genuine gradient $\nabla C(L)$. 

Since the proofs of Lemma~\ref{lemma:egwfmihrr},~\ref{lemma:egwfmihr},~\ref{lemma:5.8} are generally similar to those presented in previous work~\cite{fazel2018, gravell2020}, with suitable modifications to $C(L)$ and $\nabla C(L)$. For brevity, we defer these proofs to Appendix~\ref{sec:appsaag}.

To begin,  Lemma~\ref{lemma:egwfmihrr} demonstrates that the gradient of the smoothed cost $\nabla C_r(L)$ can approximate the genuine gradient with arbitrary accuracy, provided that the exploration radius $r$ is chosen to be sufficiently small.
\begin{lemma}
	\label{lemma:egwfmihrr}
    Let $\Delta(L)$ be an arbitrary positive constant such that $0<\Delta(L)\leq \hdelta(L)$. 
	For any $\epsilon>0$, suppose the exploration radius $r$ is chosen such that 
	\[
	r\leq h_r(\epsilon,L,\Delta(L)):=\min\left\{\Delta(L),\frac{1}{h_c(L,\Delta(L))},\frac{\epsilon}{\widehat{\hgrad}(L,\Delta(L))}\right\},	
	\]
	Then, the following error bound holds:
	\[
	\left\|\nabla C(L)-\nabla C_r(L)\right\|_F\leq \epsilon.
	\]	
\end{lemma}
In the model-free setting, calculating expectations is generally not possible.
However, the Law of Large Numbers implies that the sample mean can be used as an approximation of the expectation. 
Specifically, we can define the sample mean of $\nabla C_r(L)$ as follows:
\[
\hd =\frac{1}{N}\sum_{i=1}^N\frac{mn}{r^2}C(L+U_i)U_i.	
\]
where $\{U_i\}_{i=1}^N$ are independent and identically distributed samples drawn from $\SSbb_r$. 

The following Lemma~\ref{lemma:egwfmihr} shows that the gradient term can be approximated by $\hd$, provided the sample size $N$ is sufficiently large. 
\begin{lemma}
	\label{lemma:egwfmihr}
	Adopting the same notations as above. 
	For any $\epsilon>0$ and failure probability $\delta>0$,
	suppose the exploration radius $r$ is chosen such that  
	\[
	r\leq h_r(\frac{\epsilon}{2},L,\Delta(L)),
	\]
	and the sample size $N$ is chosen such that 
	 $N\geq \hs(\frac{\epsilon}{2},\delta,L)$ where 
	\[
		\hs(\frac{\epsilon}{2},\delta,L)=8\frac{n\land m}{\epsilon^2}\left[(\frac{2mnC(L)}{r})^2+(\frac{\epsilon}{2}+h_1(L))^2+\frac{\epsilon (\frac{2mnC(L)}{r}+\frac{\epsilon}{2}+h_1(L))}{2\sqrt{m\land n}}\right]\log(\frac{2(m+n)}{\delta})	.
	\]
	Then, with probability at least $1-\delta$, we have the following error bound:
	\[
	\left\|\nabla C(L) - \hd \right\|_F \leq \epsilon .	
	\]
\end{lemma}
	In practice, experiments and observations are limited to a finite time horizon.   
	Lemma~\ref{lemma:acaswfhr} provides a solution to this limitation by enabling us to substitute the infinite horizon quantity $C(L+U_i)$ in $\hd$ with its finite horizon counterpart $C^{(l)}(L+U_i)$. 
	Specifically, we define 
	\[
	\hdl=\frac{1}{N}\sum_{i=1}^N\frac{mn}{r^2}C^{(l)}(L+U_i)	U_i.
	\]
	The following Lemma~\ref{lemma:5.8} shows that this substitution is valid as long as the rollout length $l$ is sufficiently large.
\begin{lemma}
	\label{lemma:5.8}
	Adopting same notations as above.  
	For any $\epsilon>0$ and failure probability $\delta>0$, 
	suppose the exploration radius $r$ is chosen such that 
	\[
		r\leq h_r(\frac{\epsilon}{4},L,\Delta(L)),
	\]
	the sample size $N$ is chosen such that  
	$N\geq \hs(\frac{\epsilon}{4},\delta,L)$,
	and the rollout length $l$ is chosen as
	$$ l\geq \tilde{h}_l(L,\frac{r\epsilon}{2mn}):=\frac{4mn^2C(L)(\|\bq\|+\bnr (\|L\|+r)^2)}{r\epsilon\mu(\us_{\bq})^2}.$$
	Then, with probability at least $1-\delta$, we have the following error bound:
	\[
		\left\|\nabla C(L) - \hdl \right\|_F	\leq \epsilon.
	\]
\end{lemma}
The computation of $C^{(l)}(L)$ involves the calculation of expectations, which in turn requires access to distributional information of model parameters. In the model-free setting, such information is unavailable. 
Therefore, in the final step to approximate the true gradient, we resort to using single trajectory data as a substitute for $C^{(l)}(L)$, 
resulting in our model-free gradient estimator $\td$.

For simplicity, let $L_i' = L + U_i$. We note that
\begin{equation}\label{eq:app}
\begin{aligned}
	\hdl-\td&=\frac{1}{N}\sum_{i=1}^N\frac{mn}{r^2}\left(C^{(l)}(L_i')-\left[\sum_{t=0}^{l-1}(x_t^{(i)})\T Q_{t+1}^{(i)} x_t^{(i)}+(u_t^{(i)})\T R_{t+1}^{(i)} u_t^{(i)}\right]\right)U_i\\
	&=\frac{mn}{r^2} \frac{1}{N} \sum_{i=1}^N \sum_{t=0}^{l-1} \left( \E\left[ (x_t^{(i)})\T \left( Q_{t+1}^{(i)} + (L_i')\T R_{t+1}^{(i)} L_i' \right) x_t^{(i)} \big| U_i \right] \right. \\
	&\qquad \left. - (x_t^{(i)})\T \left( Q_{t+1}^{(i)} + (L_i')\T R_{t+1}^{(i)} L_i' \right) x_t^{(i)}  \right)U_i \\
	&=\frac{mn}{r^2} \sum_{t=0}^{l-1}\left[ \frac{1}{N} \sum_{i=1}^N \left(\E\left[ H_i^t  \big| U_i \right]  - H_i^t  \right)U_i \right],
\end{aligned}	
\end{equation}
where we define
	\[
	H_i^t = (x_t^{(i)})\T \left(Q_{t+1}^{(i)} + (L_i')\T R_{t+1}^{(i)} L_i'\right) x_t^{(i)}.	
	\]
If we can demonstrate that, for any $\{U_i\}$, the norm of random matrix $\frac{1}{N} \sum_{i=1}^N \left(\E\left[ H_i^t  \big| U_i \right]  - H_i^t  \right)U_i$ concentrates around $0$ with high probability, 
then we can conclude that the model-free gradient estimator $\td$ approximates the true gradient. 
\begin{lemma}
	\label{lemma:5.9}
	Adopting same notations as above. 
	Let $r\leq \hdelta(L)$ be fixed.
	For any $\epsilon>0$ and failure probability $\delta>0$, if 
	\[
	N\geq \hshf(\epsilon,\delta,L,l,r):=\frac{2(mnl)^2 m \land n}{\epsilon^2 r^2}\left((2{\rm e}\widetilde{\gamma_l}(L,r))^2+\frac{\epsilon}{mnl\sqrt{ m \land n}} (2{\rm e}r \widetilde{\gamma_l}(L,r))\right)\log(\frac{2(m+n)l}{\delta})	,
	\]
	where 
	\[
		\begin{aligned}
			\widetilde{\gamma_l}(L,r)&= 4 K\left( 1 + (\|L\|+r)^2\right) \left(n\widetilde{\beta_l}(L,r)+[(1+r+\|L\|)K]^{2l}\E[\|x_0\|_2^2]\right),\\
			\widetilde{\beta_l}(L,r)&= 16\left(\widetilde{\sigma_l}(L,r)^{2}+\left[K(1+r+\|L\|)\right]^{2l}\E\left[\|x_0\|_2^2\right]\right),\\
			\widetilde{\sigma_l}(L,r)&=\left[4(1+r+\|L\|)K\right]^l\left(\sigma_0+2\|\overline{x_0}\|\right).
		\end{aligned}
	\]
	Then we have 
	\[
	\Prob\left(\left\|\hdl-\td \right\|_F \leq \epsilon\right)\geq 1-\delta.	
	\]
\end{lemma}
\begin{proof}
	First, we introduce some notation. 
	Let  $D_t$  be the sets defined as follows:
	\[
	D_t:=\left\{\left\| \frac{1}{N} \sum_{i=1}^N \left(\E\left[ H_i^t \big| U_i \right]  - H_i^t \right)U_i  \right\|_F \leq \frac{\epsilon r^2}{mnl}\right\}.
	\]
	And we use the notation $\left\{U_i ,Q_{t+1}^{(i)} ,R_{t+1}^{(i)}\right\}$ to denote the set of all random matrices $U_i$, $Q_{t+1}^{(i)}$, and $R_{t+1}^{(i)}$, where $i$ and $t$ run over $1, \ldots, N$ and $0, \ldots, l-1$, respectively.

	Using~\eqref{eq:app} and triangular inequality, we get
	\[
	\|\hdl-\td\|_F	\leq \frac{mn}{r^2}\sum_{t=0}^{l-1}\left\| \frac{1}{N} \sum_{i=1}^N \left(\E\left[ H_i^t \big| U_i \right]  - H_i^t \right)U_i  \right\|_F.
	\]
	Thus, we have the following inclusion relationship
	\[
	\left\{ \|\hdl-\td\|_F > \epsilon \right\} \subseteq  \cup_{t=0}^{l-1}D_t^c .	
	\]
	If we can prove that  
	\begin{equation}\label{eq:111}
		\Prob\left( D_t^c \bigg| \left\{U_i ,Q_{t+1}^{(i)} ,R_{t+1}^{(i)}\right\}\right)\leq  \frac{\delta}{l},\quad \forall t\in\{0,\ldots,l-1\}, 
	\end{equation}
	for all $0\leq t <l$. 
	Then, we can deduce that 
	\[
		\begin{aligned}
			\Prob\left(\|\hdl-\td\|_F \leq \epsilon\right)&=1-\Prob\left(\|\hdl-\td\|_F > \epsilon\right)\\
			&=1-\E\left[\Prob\left(\|\hdl-\td\|_F > \epsilon \bigg| \left\{U_i ,Q_{t+1}^{(i)} ,R_{t+1}^{(i)} \right\}\right)\right]\\
			&\geq 1 - \E\left[\sum_{t=0}^{l-1}\Prob\left( D_t^c \bigg| \left\{U_i ,Q_{t+1}^{(i)} ,R_{t+1}^{(i)}\right\}\right)\right]\\
			&\geq 1 - \delta .	
		\end{aligned}
	\]
	Hence, it suffices to prove~\eqref{eq:111}.

	In the rest of the proof, we implicitly use the conditional probability measure $\Prob\left(\quad \cdot \quad \bigg| \left\{U_i ,Q_{t+1}^{(i)} ,R_{t+1}^{(i)} \right\}\right)$.

	We define 
	\(
	M_t^{(i)}:= Q_{t+1}^{(i)} + (L_i')\T R_{t+1}^{(i)} L_i', 
	\)
	for all $i=1,\ldots, N$ and $0 \leq t <l $.
	Then, $H_i^t$ can be expressed as 
	\(
	H_i^t = (x_t^{(i)})\T M_t^{(i)} x_t^{(i)}	
	\)

	By Applying Lemma~\ref{lemma:xmx}, we have 
	\(
		H_i^t 	\sim SE((2{\rm e}\gamma_t(L_i',M_t^{(i)}))^2,2{\rm e}\gamma_t(L_i',M_t^{(i)})),
	\)
	where
	\[
		\gamma_t(L_i', M_t^{(i)})=4 \|M_t^{(i)}\| \left(n \beta_t(L_i')+[(1+\|L_i'\|)K]^{2t}\E[\|x_0\|_2^2]\right)\leq \widetilde{\gamma_l}(L,r),
	\]
	for all $0 \leq t<l$, and $i=1\ldots N$.

	Thus, under the conditional probability $\Prob\left(\quad \cdot \quad \bigg| \left\{U_i ,Q_{t+1}^{(i)} ,R_{t+1}^{(i)} \right\}\right)$, we have 
	\[
		H_i^t \sim SE((2{\rm e}\widetilde{\gamma_l}(L,r))^2,2{\rm e}\widetilde{\gamma_l}(L,r)),
	\]
	for all $0 \leq t<l$, and $i=1\ldots N$.

	Finally, by applying Corollary~\ref{cor:concen}, we complete the proof.
\end{proof}
With the above preparation, the following lemma ensures that the model-free gradient estimator $\td$ can approximate the true gradient with arbitrary accuracy by choosing suitable hyperparameters $(r, l, N)$. 
\begin{lemma}
	\label{lemma:egwfmfhr}
	Adopting same notations as above. 
    For any $\epsilon>0$ and failure probability $\delta>0$, 
	we choose the exploration radius $r$ such that  
	\[
		r\leq h_r(\frac{\epsilon}{8},L,\Delta(L)),
	\]
	the rollout length $l$  such that 
	\[
	l\geq \tilde{h}_l(L,\frac{r\epsilon}{4mn}),
	\] 
	and the sample size $N$ such that 
	\[
	N\geq h_s(\epsilon,\delta,L,l,r),	
	\]
	where 
	\[
	\begin{aligned}
		&h_s(\epsilon,\delta,L,l,r):=\max \left\{\hs(\frac{\epsilon}{8},\frac{\delta}{2},L),\hshf(\frac{\epsilon}{2},\frac{\delta}{2},L,l,r)\right\}.
	\end{aligned}	
	\]
	Then,with at least $1-\delta$ probability, we have the following error bound: 
	\[
	\left\|\nabla C(L)- \td \right\|_F\leq \epsilon.	
	\]
	where the model-free gradient estimator is defined as: 
	\begin{equation}
		\label{eq:gp}
			\td =\frac{1}{N}\sum_{i=1}^N\frac{mn}{r^2}\left[\sum_{t=0}^{l-1}\xuti\T \Lambda_{t+1}^{(i)}\xuti \right]	U_i,
	\end{equation}
	and $\{x_t^{(i)},u_t^{(i)},\Lambda_{t+1}^{(i)}\}_{t=0}^{l-1}$ is a trajectory data using feedback matrix $L+U_i$. 
\end{lemma}
\begin{proof}
	We break the difference between model-free and true gradient into two terms:  
	\begin{equation}\label{eq:222}
	\|\nabla C(L) - \td\|_F\leq \|\nabla C(L) - \hdl\|_F + \|\hdl - \td\|_F.
	\end{equation}
	Let us define two sets
	\[
	\begin{aligned}
	&G:=\left\{\|\nabla C(L) - \hdl\|_F\leq \frac{\epsilon }{2}\right\},\\
	&F:=\left\{\| \hdl-\td \|_F\leq \frac{\epsilon }{2}\right\}.
	\end{aligned}	
	\]
	Note that
	\[
		\left\{\|\nabla C(L) - \td\|_F > \epsilon \right\} \subseteq G^c \cup F^c.
	\]
	
	To handle the first term in~\eqref{eq:222}, we use Lemma~\ref{lemma:5.8}. 
	Since $r\leq h_r(\frac{\epsilon}{8},L,\Delta(L))$ and $N\geq h_s(\epsilon,\delta,L,l,r)\geq  \hs(\frac{\epsilon}{8},\frac{\delta}{2},L)$, 
	Lemma~\ref{lemma:5.8} implies  
	\[
	\Prob(G^c)=\Prob\left(\|\nabla C(L)-\hdl\|_F > \frac{\epsilon}{2}\right)\leq \frac{\delta}{2}	.
	\]
	
	To handle the second term in~\eqref{eq:222}, we use Lemma~\ref{lemma:5.9}. 
	Since  $r\leq h_r(\frac{\epsilon}{8},L,\Delta(L))\leq \Delta(L)$ and $N\geq h_s(\epsilon,\delta,L,l,r) \geq \hshf(\frac{\epsilon}{2},\frac{\delta}{2},L,l,r)$, 
	Lemma~\ref{lemma:5.9} implies
	\[
	\Prob(F^c)=\Prob\left(\|\hdl-\td\|_F > \frac{\epsilon}{2}\right)\leq \frac{\delta}{2}	.
	\]
	Combining the above results, we obtain
	\[
		\begin{aligned}
			\Prob\left(\|\nabla C(L) -\td\|_F \leq \epsilon\right)&=1-\Prob\left(\|\nabla C(L) -\td\|_F > \epsilon\right)\\
			&\geq 1- \Prob\left(G^c \right)-\Prob\left(F^c \right)\\
			&\geq 1 - \delta 	.
		\end{aligned}
	\]
\end{proof}
\subsection{Proof of theorem~\ref{thm:samplebased}}\label{sec:pot2}
Combing above preparing lemmas, we are now ready to prove the global convergence of the model-free gradient method. 
\begin{thmn}[\ref{thm:samplebased}]\label{thm:thm222}
	Assuming that Assumptions~\ref{ass:1},~\ref{ass:2} and~\ref{ass:3} hold, 
	and $\{U_i\}_{i=1}^N$are independent and identically distributed samples from $\SSbb_r$. 
	Let $$\widehat{\nabla C}(L_k)=\frac{1}{N}\sum_{i=1}^N\frac{mn}{r^2}\left[\sum_{t=0}^{l-1}\xuti\T \Lambda_{t+1}^{(i)}\xuti \right]	U_i$$ be the model-free gradient term, 
	where $\{x_t^{(i)},u_t^{(i)},\Lambda_{t+1}^{(i)}\}_{t=0}^{l-1}$  is a sample trajectory generated by feedback matrix $L_k+U_i$ for each $i=1,\ldots,N$. 
	We  use the model-free policy gradient update: 
	\[
	L_{k+1}=L_k-\eta\widehat{\nabla C}(L_k)
	\]
	where $L_0\in U_{ad}$ is chosen arbitrarily. 
	We define the sublevel set 
	$$S:=\left\{L\in\R^{m\times n}\big| C(L)\leq 2C(L_0)\right\}.$$ 

	For any $\epsilon>0$ and failure probability $\delta>0$, 
	if we choose the step size $\eta$  such that $0<\eta<\eta^*$, 
	the exploration radius $r$  such that $r\leq \hrgd(\frac{\epsilon'}{4})$, 
	the rollout length $l$ such that $l\geq \hlgd(\frac{r\epsilon'}{4mn})$, 
	and the sample size $N$ such that $N\geq \hs^* (\epsilon',\delta',l,r)$.
	Then, with probability at least $1-\delta$, at most at the $(k+1)$-th update step, we have the following performance guarantee:
	\[
	C(L_{k+1})-C(L^*)<\epsilon,
	\]
	where 
	\[
		\begin{aligned}
		&\bhdelta:=\min_{L\in S}\hdelta(L),\\
		&\overline{h_c}:=\max_{L\in S}h_c(L,\bhdelta),\\	
		&\epsilon':=\min\left\{\frac{\overline{\hdelta}}{\eta},\frac{\mu^2\us_{\br}\epsilon}{2\|\Sigma_{L^*}\|C(L_0)\overline{h_c}},\epsilon\right\},	\\
		&k= \left\lceil \frac{\|\Sigma_{L^*}\|}{\eta\mu^2\us_{\br}}	\log\frac{C(L_0)-C(L^*)}{\epsilon}\right\rceil,\\
		&\delta':=\frac{2\delta}{(k+1)(k+2)},\\
		&\hrgd(\frac{\epsilon'}{8}):=\min_{L\in S}h_r(\frac{\epsilon'}{8},L,\bhdelta),\\
		&h_r(\frac{\epsilon'}{8},L,\bhdelta):=\min\left\{\bhdelta,\frac{1}{h_c(L,\bhdelta)},\frac{\epsilon'}{8\widehat{\hgrad}(L,\bhdelta)}\right\},	\\
		&\hlgd(\frac{r\epsilon'}{4mn}):=\max_{L\in S}\widetilde{h}_l(L,\frac{r\epsilon'}{4mn}),\\
		&\hs^* (\epsilon',\delta',l,r):=\max_{L\in S}h_s(\epsilon',\delta',L,l,r).
		\end{aligned}
	\]
\end{thmn}
\begin{proof}$\quad$

	\noindent {\bf Step 1: }  

	The well-definedness and sample complexity of $\eta^*$ have been shown in Theorem \ref{thm:mb}.
	The well-definedness and sample complexity of  $\bhdelta$, $\overline{h_c}$, $\epsilon'$, $\hrgd$, $\hlgd$, $\hs^*$ follow from 
	Lemmas~\ref{lemma:wdbhd},~\ref{lemma:wdbhc},~\ref{lemma:wde},~\ref{lemma:wdhrgd},~\ref{lemma:wdhlgd}, and~\ref{lemma:wdhss} respectively.
	
	\noindent {\bf Step 2: }  

	In order to simplify our argument, we define the following sets:
	\[
		\begin{aligned}
			&D_j:=\left\{C(L_{j})-C(L^*)<\epsilon\right\},\\
			&\Theta:=\cup_{j=0}^k D_j,\\
			&E_j:=\left\{ C(L_{j+1})\leq C(L_0) \right\}\cap\left\{C(L_{j+1})-C(L^*)\leq \left(1-\eta\frac{\mu^2\us_{\br}}{\|\Sigma_{L^*}\|}\right)(C(L_j)-C(L^*))\right\}.	
		\end{aligned}
	\]
	where $j\geq 0$ and $E_{-1}=\Omega$.

	\noindent{\bf Claim: } When $0\leq j\leq k$, it holds that
	\[
		\Prob\left(E_j\cup \Theta \big| E_{j-1}\right)\geq\Prob\left(E_j\cup D_j \big| E_{j-1}\right)\geq 1 - \delta'.
	\]
	This claim means that if $j$-th update is successful, 
	then, with at least probability $1-\delta'$, 
	either the $(j+1)$-th update is successful: 
	\[
	C(L_{j+1})-C(L^*)\leq \left(1-\eta\frac{\mu^2\us_{\br}}{\|\Sigma_{L^*}\|}\right)(C(L_j)-C(L^*))\quad \implies\quad C(L_{j+1})\leq C(L_{j})\leq C(L_0) ,
	\] 
	or convergence has been attained early:
	\[
	C(L_{j})-C(L^*)<\epsilon.	
	\]

	\noindent{\it Proof of the claim: } 
	First note that the condition $E_{j-1}$ implies $C(L_j)\leq C(L_0)<\infty$ (when $j=0$ this holds automatically which is consistent with $E_{-1}=\Omega$). 
	Let 
	\[
		L_j'=L_j-\eta\nabla C(L_j),\qquad L_{j+1}=L_j-\eta \widehat{\nabla C}(L_j).	
	\]
	Since $r\leq \hrgd(\frac{\epsilon'}{8})\leq \min\left\{\hdelta(L_j),(L_c(L_j,\bhdelta))^{-1}\right\}$, Lemma~\ref{lemma:cp} gives
	\[
		|C(L_j+U_i)-C(L_j)|\leq h_c(L_j,\bhdelta)\|U_i\|C(L_j)\leq C(L_j).	
	\]
	Thus 
	$$C(L_j+U_i)<2C(L_j)\leq 2C(L_0) \implies L_j+U_i \in S\quad \forall U_i\in S_F(0,r).$$ 
	By Theorem \ref{thm:mb} ,
	\[
	C(L_j')<C(L_j)\leq C(L_0)\implies L_j'\in S.	
	\]
	Hence, every $L$ encountered in the $(j+1)$-th update step are legal, except $L_{j+1}$. 

	Observe that $L_j'-L_{j+1}=\eta(\nabla C(L_j)-\widehat{\nabla C}(L_j))$. 
	Using Lemma \ref{lemma:egwfmfhr} and definition of $\epsilon'$, 
	we have 
	\[
	\|L_{j+1}-L_j'\|=\eta\|\nabla C(L_j)-\td_j\|\leq 	\eta \epsilon'\leq\bhdelta\leq \hdelta(L_j'),
	\]
	with probability at least $1-\delta'$. 
	Combing this result with Lemma~\ref{lemma:cp} we obtain that 
	\[
	|C(L_{j+1})-C(L_j')|\leq C(L_j')h_c(L_j',\bhdelta)\eta \epsilon'\leq \eta\frac{\mu^2 \us_{\br}}{\|\Sigma_{L^*}\|}\epsilon.	
	\]
	with at least probability $1-\delta'$, i.e.,  
	\begin{equation}
		\label{eq:thm22}
		\Prob\left(\left\{|C(L_{j+1})-C(L_j')|\leq \eta\frac{\mu^2 \us_{\br}}{\|\Sigma_{L^*}\|}\epsilon\right\}\bigg|E_{j-1}\right)	\geq 1-\delta'.
	\end{equation}
	Note that on the set $\left\{|C(L_{j+1})-C(L_j')|\leq \eta\frac{\mu^2 \us_{\br}}{\|\Sigma_{L^*}\|}\epsilon\right\}\cap D_j^c$, 
	using Theorem~\ref{thm:mb}, we have 
	\[
	\begin{aligned}
	C(L_{j+1})-C(L^*)&\leq \left|C(L_{j+1})-C(L_j')\right|+C(L_j')-C(L^*)\\
	&\leq 	(\eta\frac{\mu^2 \us_{\br}}{\|\Sigma_{L^*}\|}+1-2\eta\frac{\mu^2 \us_{\br}}{\|\Sigma_{L^*}\|})(C(L_j)-C(L^*))	\\
	&\leq 	(1-\eta\frac{\mu^2 \us_{\br}}{\|\Sigma_{L^*}\|})(C(L_j)-C(L^*)),
	\end{aligned}	
	\] 
	which implies
	\[
	C(L_{j+1})\leq C(L_j)\leq C(L_0).
	\]
	In other words, 
	\[
	E_j\supseteq 	\left\{|C(L_{j+1})-C(L_j')|\leq \eta\frac{\mu^2 \us_{\br}}{\|\Sigma_{L^*}\|}\epsilon\right\}\cap D_j^c.
	\]
	Therefore by~\eqref{eq:thm22}, we obtain
	\[
	\begin{aligned}
		\Prob\left(E_j\cup \Theta\big| E_{j-1}\right)&\geq\Prob\left(E_j\cup D_j\big| E_{j-1}\right)\\
		&\geq \Prob\left(\left(\left\{|C(L_{j+1})-C(L_j')|\leq \eta\frac{\mu^2 \us_{\br}}{\|\Sigma_{L^*}\|}\epsilon\right\}\cap D_j^c\right)\cup D_j \big| E_{j-1}\right) \\
		&= \Prob\left(\left\{|C(L_{j+1})-C(L_j')|\leq \eta\frac{\mu^2 \us_{\br}}{\|\Sigma_{L^*}\|}\epsilon\right\} \cup D_j \big| E_{j-1}\right) \\
		&\geq \Prob\left(\left\{|C(L_{j+1})-C(L_j')|\leq \eta\frac{\mu^2 \us_{\br}}{\|\Sigma_{L^*}\|}\epsilon\right\} \bigg| E_{j-1}\right)\\
		&\geq 1-\delta'.
	\end{aligned}	
	\]
	This conclude the proof of our claim.

	\noindent {\bf Step 3: }

	Step 2 implies that 
	\[
		\Prob\left(E_j^c\cap \Theta^c \big| E_{j-1}\right)\leq \delta',\quad 0\leq j\leq k.
	\] 
	Specially at the first step, we have
	$$\Prob(E_0^c\cap \Theta^c)\leq \delta',\quad \Prob\left(E_1^c\cap \Theta^c \big| E_{0}\right)\leq \delta',$$
	then 
	\[
	\begin{aligned}
	\Prob\left(E_1^c\cap \Theta^c\right)&=\Prob\left(E_1^c\cap \Theta^c\big| E_0\right)\Prob\left(E_0\right)+\Prob\left(E_1^c\cap \Theta^c\big| E_0^c\right)\Prob\left(E_0^c\right)\\
	&\leq \Prob\left(E_1^c\cap \Theta^c\big| E_0\right)+\Prob\left(E_1^c\cap \Theta^c\cap E_0^c\right)\\
	&\leq \Prob\left(E_1^c\cap \Theta^c\big| E_0\right)+\Prob\left(E_0^c\cap \Theta^c\right)\\
	&\leq 2\delta'.
	\end{aligned}
	\]
	Inductively, we obtain
	\[
	\begin{aligned}
	\Prob\left(E_j^c\cap \Theta^c\right)&=\Prob\left(E_j^c\cap \Theta^c\big| E_{j-1}\right)\Prob\left(E_{j-1}\right)+\Prob\left(E_j^c\cap \Theta^c\big| E_{j-1}^c\right)\Prob\left(E_{j-1}^c\right)\\
	&\leq \Prob\left(E_j^c\cap \Theta^c\big| E_{j-1}\right)+\Prob\left(E_j^c\cap \Theta^c\cap E_{j-1}^c\right)\\
	&\leq \delta'+\Prob\left(E_{j-1}^c\cap \Theta^c\right)\\
	&\leq (j+1)\delta'.
	\end{aligned}
	\]

	For the event $\omega\in \left(\cap_{j=0}^{k}E_j\right)\cup \Theta$, at most at the ($k+1$)-th update step, we have
	\[
		C(L_{k+1})-C(L^*)\leq \left(1-\eta\frac{\mu^2\us_{\br}}{\|\Sigma_{L^*}\|}\right)^{k+1}(C(L_0)-C(L^*))< \epsilon,
	\]
	and by union bound
	\[
	\Prob\left(\left(\left(\cap_{j=0}^{k}E_j\right)\cup \Theta\right)^c \right)=\Prob\left(\cup_{j=0}^{k}\left(E_j^c\cap \Theta^c\right)\right)\leq\sum_{j=0}^k(j+1)\delta'= \delta.
	\]
	This conclude the proof of this theorem.
\end{proof}
\section{Conclusions}

In conclusion, this work has demonstrated the convergence of policy gradient methods in both model-based and model-free settings for LQ systems with random parameters. 
The proof techniques developed in this work can be extended to analyze the convergence of similar algorithms, such as natural policy gradient and projected policy gradient methods. 
Furthermore, examining the optimal sample complexities of parameters $(N,r,l,k,\eta^*)$ and
relaxing the assumption of almost sure boundedness on random parameters is directions for future research. 

\vspace{0.3cm}
\noindent{\bf Acknowledgments} 

I would like to express my  sincere  gratitude to Prof. Kai Du from Fudan University 
for valuable discussions, suggestions, and guidance during the course of this research.  

The author's research is supported by the National Key R{\&}D Program of China (No.~2022ZD0116401). 

\bibliographystyle{plain}
\bibliography{ref.bib}

\begin{thebibliography}{10}

\bibitem{aoki1976}
M.~Aoki.
\newblock {\em Optimal Control and System Theory in Dynamic Economic Analysis}.
\newblock Number Vol. 1 in A Series of Volumes in Dynamic Economics : Theory
  and Applications. North Holland Publishing Company, 1976.

\bibitem{athans1977}
Michael Athans, Richard Ku, and Stanley Gershwin.
\newblock The uncertainty threshold principle: Some fundamental limitations of
  optimal decision making under dynamic uncertainty.
\newblock {\em IEEE Transactions on Automatic Control}, 22(3):491--495, 1977.

\bibitem{beghi1998}
Alessandro Beghi and Domenico D'alessandro.
\newblock Discrete-time optimal control with control-dependent noise and
  generalized riccati difference equations.
\newblock {\em Automatica}, 34(8):1031--1034, 1998.

\bibitem{de1982}
Willem~L De~Koning.
\newblock Infinite horizon optimal control of linear discrete time systems with
  stochastic parameters.
\newblock {\em Automatica}, 18(4):443--453, 1982.

\bibitem{dean2020}
Sarah Dean, Horia Mania, Nikolai Matni, Benjamin Recht, and Stephen Tu.
\newblock On the sample complexity of the linear quadratic regulator.
\newblock {\em Foundations of Computational Mathematics}, 20(4):633--679, 2020.

\bibitem{drenick1964}
R~Drenick and L~Shaw.
\newblock Optimal control of linear plants with random parameters.
\newblock {\em IEEE Transactions on Automatic Control}, 9(3):236--244, 1964.

\bibitem{du2022}
Kai Du, Qingxin Meng, and Fu~Zhang.
\newblock A q-learning algorithm for discrete-time linear-quadratic control
  with random parameters of unknown distribution: convergence and
  stabilization.
\newblock {\em SIAM Journal on Control and Optimization}, 60(4):1991--2015,
  2022.

\bibitem{fazel2018}
Maryam Fazel, Rong Ge, Sham Kakade, and Mehran Mesbahi.
\newblock Global convergence of policy gradient methods for the linear
  quadratic regulator.
\newblock In {\em International Conference on Machine Learning}, pages
  1467--1476. PMLR, 2018.

\bibitem{gravell2020}
Benjamin Gravell, Peyman~Mohajerin Esfahani, and Tyler Summers.
\newblock Learning optimal controllers for linear systems with multiplicative
  noise via policy gradient.
\newblock {\em IEEE Transactions on Automatic Control}, 66(11):5283--5298,
  2020.

\bibitem{hambly2021}
Ben Hambly, Renyuan Xu, and Huining Yang.
\newblock Policy gradient methods for the noisy linear quadratic regulator over
  a finite horizon.
\newblock {\em SIAM Journal on Control and Optimization}, 59(5):3359--3391,
  2021.

\bibitem{horn2012}
Roger~A Horn and Charles~R Johnson.
\newblock {\em Matrix analysis}.
\newblock Cambridge university press, 2012.

\bibitem{kalman1961}
R.~E. Kalman.
\newblock Control of randomly varying linear dynamical systems.
\newblock {\em Proceedings of Symposia in Applied Mathematics}, pages 287--298,
  1961.

\bibitem{ku1977}
Richard Ku and Michael Athans.
\newblock Further results on the uncertainty threshold principle.
\newblock {\em IEEE Transactions on Automatic Control}, 22(5):866--868, 1977.

\bibitem{lai2023}
Jing Lai, Junlin Xiong, and Zhan Shu.
\newblock Model-free optimal control of discrete-time systems with additive and
  multiplicative noises.
\newblock {\em Automatica}, 147:110685, 2023.

\bibitem{morozan1983}
Toader Morozan.
\newblock Stabilization of some stochastic discrete--time control systems.
\newblock {\em Stochastic Analysis and Applications}, 1(1):89--116, 1983.

\bibitem{simchowitz2018}
Max Simchowitz, Horia Mania, Stephen Tu, Michael~I Jordan, and Benjamin Recht.
\newblock Learning without mixing: Towards a sharp analysis of linear system
  identification.
\newblock In {\em Conference On Learning Theory}, pages 439--473. PMLR, 2018.

\bibitem{tiedemann1984}
AR~Tiedemann and WL~De~Koning.
\newblock The equivalent discrete-time optimal control problem for
  continuous-time systems with stochastic parameters.
\newblock {\em International Journal of Control}, 40(3):449--466, 1984.

\bibitem{tu2019}
Stephen Tu and Benjamin Recht.
\newblock The gap between model-based and model-free methods on the linear
  quadratic regulator: An asymptotic viewpoint.
\newblock In {\em Conference on Learning Theory}, pages 3036--3083. PMLR, 2019.

\bibitem{ver2018}
Roman Vershynin.
\newblock {\em High-dimensional probability: An introduction with applications
  in data science}, volume~47.
\newblock Cambridge university press, 2018.

\bibitem{wain2019}
Martin~J Wainwright.
\newblock {\em High-dimensional statistics: A non-asymptotic viewpoint},
  volume~48.
\newblock Cambridge University Press, 2019.

\end{thebibliography}

\newpage

\appendix
\section{Proof of technical results}\label{sec:tech}
Except for the lemmas discussing the concentration inequalities, the proofs for the following lemmas mainly follow the proof techniques in~\cite{fazel2018,gravell2020}, with the relevant quantities replaced by those in the current problem.  
All proofs are presented for the sake of completeness and rigor.
\subsection{Proofs in Section~\ref{sec:cigd}}\label{sec:po3.1}

First of all, we define the value function under linear feedback control $u_.=-Lx_.$ at $x_0$ as:
\begin{equation}
	\label{eq:vfulc}
	V_L(x_0)=\E\left[J(x_0,-Lx_.)\big| x_0\right]=x_0\T P_Lx_0.
\end{equation}
The dynamics programming implies that $P_L$ is the positive semi-definite solution to the following stochastic Lyapunov equation:
\begin{equation}
\label{eq:arelinear}
    P_L=\overline{Q}+L\T \overline{R} L+\overline{(A-BL)\T P_L(A-BL)}.
\end{equation}

It follows that the cost function $C(L)$ can be rewritten as:
\begin{equation}
\label{eq:closedformcost}
    C(L)=\E_{x_0}\left[V_L(x_0)\right]=\E_{x_0}\left[x_0\T P_Lx_0\right]=\tr(P_L\Sigma_0).
\end{equation}
Recall that in Lemma~\ref{lemma:gd} we have defined that  
\begin{align}\label{eq:RLEL}
    R_L&:=\rl,\\
    E_L&:=R_L L-\bbpla.
\end{align}
The closed form expression~\eqref{eq:closedformcost} of $C(L)$ and the fact that $P_L$ in the expression satisfies the equation~\eqref{eq:arelinear} suggest the following representation of gradient term. 
\begin{lemma}[Policy Gradient Expression]
	The policy gradient is:
	\[
	\nabla C(L)=2E_L\Sigma_L=2\overline{(R+B\T P_L B)L-B\T P_L A}\Sigma_L	.
	\] 
\end{lemma}
\begin{proof}
    Plugging the right hind side of Equation~\eqref{eq:arelinear} in $C(L)=\tr(P_L\Sigma_0)$ gives
	\begin{equation}
        C(L)=\tr(\bq \Sigma_0)+\tr(L\T \br L\Sigma_0)+\tr(\overline{(A-BL)\T P_L (A-BL)}\Sigma_0).
    \end{equation}
	By chain rule, taking derivative with respect to $L$ gives
	\[
	\begin{aligned}
		\nabla_L C(L)&=\nabla_L \tr(P_L\Sigma_0)\\
		&=\nabla_L\left(\tr(\bq \Sigma_0)+\tr(L\T \br L\Sigma_0)+\tr(\overline{(A-BL)\T P_L (A-BL)}\Sigma_0)\right)\\
		&=2\br L\Sigma_0-2\overline{B\T P_L(A-BL)}\Sigma_0+\tr(\overline{(A-B L)\T \nabla_{L}P_{L} (A-B L)}\Sigma_0)\\
		&=2\overline{(R+B\T P_L B)L-B\T P_L A}\Sigma_0 +\E\left[x_0\T (A_1-B_1L)\T \nabla_{L}P_{L} (A_1-B_1L)x_0\right] \\
		&=2\overline{(R+B\T P_L B)L-B\T P_L A}\Sigma_0+\E\left[x_1\T \nabla_{L}P_{L} x_1\right] \\
		&=2\overline{(R+B\T P_L B)L-B\T P_L A}\Sigma_0+\tr(\nabla_{L}P_{L} \Sigma_1).
	\end{aligned}	
	\]
	Inductively, we obtain
\[
	\nabla_L C(L)=2\left(\overline{(R+B\T P_L B)L-B\T P_L A}\right)\sum_{i=0}^t\Sigma_i+\tr(\nabla_{L}P_{L} \Sigma_{t+1}) .
\]
Then, letting $t\to \infty$, monotone convergence theorem and the condition $L\in U_{ad}$ imply
	\[
		\nabla_L C(L)=2\left(\overline{(R+B\T P_L B)L-B\T P_L A}\right) \Sigma_L=2 E_L \Sigma_L.
	\]
\end{proof}

Next, we define the state-action function (Q function) and the advantage function.
\begin{definition}[Q function \& the advantage]$\quad$

	\begin{itemize}
		\item Let the Q function be the cost of the policy starting with $x_\tau=x$, taking action $u_\tau=u$ and then proceeding with $u_t=-Lx_t,t> \tau$ onwards:
		\[
			Q_L(x,u,\tau):=\E\left[x_\tau\T Q_{\tau+1}x_\tau+u_\tau\T R_{\tau+1}u_\tau+\E\left[V_L(A_{\tau+1}x_\tau+B_{\tau+1}u_\tau)\big| x_\tau,u_\tau\right]\big| x_\tau=x,u_\tau=u\right].
		\]
		\item Let the advantage be the change in cost starting at state $x_\tau=x$ and taking $u_\tau=u$ as a one step deviation from linear policy $u_.=-Lx_.$:
		\[
		A_L(x,u,\tau):=Q_L(x,u,\tau)-V_L(x).	
		\]
	\end{itemize}
\end{definition}
\begin{remark}\label{ark:1}
	Since random parameters are i.i.d., the Q function is time independent and can be written as 
	\[
		Q_L(x,u,\tau)=Q_L(x,u)=x\T\bq x+u\T\br u+ \overline{V_L(Ax+Bu)}.
	\]
\end{remark}
\begin{lemma}[Value difference lemma]
	\label{alemma:vd}
	Suppose $L,L'\in U_{ad}$. 
	Let $\{x_t, u_t\}$ and  $\{x'_t, u'_t\}$ be state 
	and action trajectories generated by $L$ and $L'$ 
	respectively. 
	Then the value difference satisfy:
	\[
	\begin{aligned}
	V_{L'}(x)-V_L(x)=\E\left[\sum_{t=0}^{\infty}A_L(x'_t,u'_t,t)\big|x'_0=x\right].
	\end{aligned}	
	\]  
	And the advance satisfies:
	\begin{equation}
		\label{eq:vd}
		A_L(x,-L'x)=x\T(L'-L)\T R_L (L'-L)x+2x\T(L'-L)\T E_L(L'-L)x.
	\end{equation}
\end{lemma}
\begin{proof}
	Let $c_t(x_0)$ and $c_t(x_0')$ denote the cost at step $t$ in a trajectory generated by $L$ and $L'$ starting form $x_0$ and $x'_0$ respectively. 
	\[
		\begin{aligned}
		c_t(x_0)&=x_t\T Q_{t+1}x_t+u_t\T R_{t+1}u_t,\\
		x_{t+1}&=A_{t+1}x_t+B_{t+1}u_t,
		\end{aligned}
	\]
	with $u_t=-Lx_t$.
	\[
		\begin{aligned}
		c_t(x'_0)&=(x'_t)\T Q_{t+1}x'_t+(u'_t)\T R_{t+1}u'_t,\\
		x'_{t+1}&=A_{t+1}x'_t+B_{t+1}u'_t,
		\end{aligned}
	\]
	with $u'_t=-L'x'_t$.

	Telescoping the sum and using Fubini's Theorem in the last equality, we get
	\[
	\begin{aligned}
	V_{L'}(x)-V_L(x)&=\E\left[\sum_{t=0}^{\infty}c_t(x_0')\big|x'_0=x\right]-V_L(x)\\
		&=\E\left[\sum_{t=0}^{\infty}(c_t(x_0')-V_L(x'_t))\big|x'_0=x\right]+\E\left[\sum_{t=0}^{\infty}V_L(x'_t)\big|x'_0=x\right]-V_L(x)\\
		&=\E\left[\sum_{t=0}^{\infty}(c_t(x_0')-V_L(x'_t))\big|x'_0=x\right]+\E\left[\sum_{t=0}^{\infty}V_L(x'_{t+1})\big|x'_0=x\right]\\
		&=\E\left[\sum_{t=0}^{\infty}(c_t(x_0')+V_L(x'_{t+1})-V_L(x'_t))\big|x'_0=x\right]\\
		&=\sum_{t=0}^{\infty}\E\left[c_t(x_0')+V_L(x'_{t+1})-V_L(x'_t)\big|x'_0=x\right].
	\end{aligned}	
	\]
	Then using properties of conditional expectation,  independency and Remark~\ref{ark:1}, we obtain
	\[
	\begin{aligned}
		 &\E\left[c_t(x_0')+V_L(x'_{t+1})-V_L(x'_t)\big|x'_0=x\right]\\
		=&\E\left[\E\left[c_t(x_0')+V_L(x'_{t+1})-V_L(x'_t)\big| x'_t,u'_t,x'_0\right]\big|x'_0=x\right]\\
		=&\E\left[\E\left[(x'_t)\T Q_{t+1}x'_t+(u'_t)\T R_{t+1}u'_t+V_L(A_{t+1}x'_{t}+B_{t+1}u'_t)-V_L(x'_t)\big| x'_t,u'_t,x'_0\right]\big|x'_0=x\right]\\
		=&\E\left[(x'_t)\T \bq x'_t+(u'_t)\T \br u'_t+ \overline{V_L(A x'_{t}+B u'_t)}-V_L(x'_t)\big|x'_0=x\right]\\
		=&\E\left[Q_L(x'_t,u'_t)-V_L(x'_t)\big|x'_0=x\right].
	\end{aligned}	
	\] 
	Without loss of generality, assume $V_{L'}\geq V_L$, i.e. $P_{L'}\succeq P_L$.
	Then the Principle of Optimality gives $$A_L(x,-L'x)=Q_L(x,-L'x) - V_L(x)\geq 0.$$

	Therefore by monotone convergence theorem, it holds that 
	\[
		V_{L'}(x)-V_L(x)=\sum_{t=0}^{\infty}\E\left[Q_L(x'_t,u'_t)-V_L(x'_t)\big|x'_0=x\right]=\E\left[\sum_{t=0}^{\infty} Q_L(x'_t,u'_t)-V_L(x'_t)\big|x'_0=x\right].
	\]

	For the second part of the proof. We expand and substitute in definition
	\[
		\begin{aligned}
			A_L(x,-L'x)&=Q_L(x,-L'x)-V_L(x)\\
			&=x\T (\bq +(L' )\T \br L')x + \overline{V_L ((A-BL')x)} - V_L(x)\\
			&=x\T (\bq +(L' )\T \br L')x+x \T \overline{(A-BL')\T P_L(A-BL')} x - x \T P_L x\\
			&\overset{(a)}{=}x\T (\bq+(L'-L+L)\T \br (L'-L+L))x\\
			&\qquad+x\T\overline{(A-BL-B(L'-L))\T P_L(A-BL-B(L'-L))}x\\
			&\qquad-x\T(\bq+L\T \br L+\overline{(A-BL)\T P_L(A-BL)})x\\
			&=x\T(L'-L)\T(\br+\overline{B\T P_LB})(L'-L)x\\
			&\qquad +2x\T (L'-L)\T ((\br+\overline{B\T P_LB})L-\overline{B\T P_L A})x,
		\end{aligned}
	\]
	where in $(a)$ we use Lyapunov equation~\eqref{eq:arelinear} for $P_L$.
\end{proof}
\begin{proof}[Proof of Lemma~\ref{lemma:gd}]
		First by the advantage expression in Lemma~\ref{alemma:vd}, we have 
		\begin{equation}
			\label{eq:gn1}
			\begin{aligned}
			A_L(x,-L'x)&=2\tr(xx\T (L'-L)\T E_L)+\tr(xx\T (L'-L)\T R_L (L'-L))	\\
			&=\tr(xx\T (L'-L+R_L^{-1}E_L)\T R_L (L'-L+R_L^{-1}E_L))\\
			&\qquad -\tr(xx\T E_L\T R_L^{-1}E_L)\\
			&\geq -\tr(xx\T E_L\T R_L^{-1}E_L),
			\end{aligned}
		\end{equation}
		with equality when \(L'=L-R_L^{-1}E_L\) (under the condition that \(\br>0\)).
	
		For the upper bound, let $L'=L^*$ and $x_t^*$, $u_t^*$ be the sequence generated by $L^*$. Using Lemma \ref{alemma:vd} again, we have
		\[
		\begin{aligned}
		C(L)-C(L^*)&=-\E\left[\sum_{t=0}^{\infty}A_L(x_t^*,u_t^*)\right]	\\
		&\leq \E\left[\sum_{t=0}^{\infty}\tr(x^*_t(x^*_t)\T E_L\T R_L^{-1} E_L)\right]\\
		&\leq \|\Sigma_{L^*}\|\tr(E_L\T R_L^{-1}E_L)\\
		&\leq \frac{\|\Sigma_{L^*}\|}{\us_{\br}}\tr(E_L\T E_L)\\
		&\leq \frac{\|\Sigma_{L^*}\|}{4(\us_{\Sigma_L})^2\us_{\br}}\tr(\nabla C(L)\T \nabla C(L))\\
		&\leq \frac{\|\Sigma_{L^*}\|}{4\mu^2\us_{\br}}\tr(\nabla C(L)\T \nabla C(L)).
		\end{aligned}	
		\]
		For the lower bound, let $L'=L-R_L^{-1}E_L$ where the equality holds in (\ref{eq:gn1}). Let $x'_t$, $u'_t$ be the sequence generated under $L'$. Using that $C(L^*)\leq C(L')$, it holds that 
		\[
		\begin{aligned}
			C(L)-C(L^*)&\geq C(L)-C(L')\\
			&=-\E\left[\sum_{t=0}^{\infty}A_L(x'_t,-L'x'_t)\right]\\
			&=\E\left[\sum_{t=0}^{\infty}\tr(x'_t(x'_t)\T E_L\T R_L^{-1}E_L)\right]\\
			&=\tr(\Sigma_{L'}E_L\T R_L^{-1}E_L)\\
			&\geq \frac{\mu}{\|R_L\|}\tr(E_L\T E_L).
		\end{aligned}	
		\] 
\end{proof}
\begin{proof}[Proof of Lemma~\ref{lemma:as}]
		Using Lemma~\ref{alemma:vd}, we obtain
		\[
		\begin{aligned}
			C(L')-C(L)&=\E\left[\sum_{t=0}^{\infty}A_L(x'_t,-L'x'_t)\right]\\
			&=\E\left[\sum_{t=0}^{\infty}\left(2\tr(x'_t(x'_t)\T (L'-L)\T E_L)+\tr(x'_t(x'_t)\T (L'-L)\T R_L(L'-L))\right)\right]\\
			&=2\tr(\Sigma_{L'}(L'-L)\T E_L)+\tr(\Sigma_{L'}(L'-L)\T R_L (L'-L)).
		\end{aligned}	
		\]
\end{proof}
The following Lemmas bound some key quantities with model parameters on sublevel set $S_{\gamma}:=\{L\in\R^{n\times m}\big|C(L)\leq \gamma\}$.
\begin{lemma}
    \label{lemma:cb}
    Assume Assumptions~\ref{ass:1} and~\ref{ass:2} hold and $L\in U_{ad}$. Then it holds that: 
    \[
	\|P_L\|\leq \frac{C(L)}{\mu},\qquad \|\Sigma_L\|\leq \frac{C(L)}{\us_{\bq}}.	
	\]
\end{lemma}
\begin{proof}
	Using \eqref{eq:closedformcost}, the cost is lower bounded as 
	\[
	C(L)=\tr(\Sigma_0P_L)\geq \sigma_{\min}(\Sigma_0)\|P_L\|=\mu\|P_L\|	,
	\]
	which gives the first inequality. 
	
	Let $Q_L:=\E\left[Q+L\T R L\right]$.
	Then by monotone convergence theorem, the cost is also lower bounded as 
	\[
	\begin{aligned}
	C(L)&=\E\left[\sum_{t=0}^{\infty}x_t\T(Q_{t+1}+L\T R_{t+1}L)x_t\right]	\\
		&=\sum_{t=0}^{\infty}\E\left[x_t\T(Q_{t+1}+L\T R_{t+1}L)x_t\right]\\
		&=\sum_{t=0}^{\infty}\E\left[\E\left[x_t\T(Q_{t+1}+L\T R_{t+1}L)x_t\big| x_t\right]\right]\\
		&=\sum_{t=0}^{\infty}\E\left[x_t\T(\bq +L\T \br L)x_t \right]\\		
		&=\E\left[\sum_{t=0}^{\infty}x_t\T Q_L x_t\right]\\
		&=\E\left[\sum_{t=0}^{\infty}\tr(x_tx_t\T Q_L)\right]\\
		&=\tr(\Sigma_LQ_L)\\
		&\geq \sigma_{\min}(\bq)\|\Sigma_L\|,
	\end{aligned}	
	\]
	which gives the second inequality.	
\end{proof}

\begin{lemma}
    \label{cor:2}
    Under the same assumptions as in Lemma~\ref{lemma:cb}. It holds that
    \[\beta(L)\leq 2\frac{\gamma}{\us_{\bq}}(\|\br\|+\bbs\frac{\gamma}{\mu}) \]
    for any $L\in S_{\gamma}$
\end{lemma}
\begin{proof}
    Using definition of $R_L$ and properties of norm, we obtain
    \[
        \beta(L)=2\|\Sigma_L\|\|R_L\|\leq 2\|\Sigma_L\|(\|\br\|+\bbs\|P_L\|).
    \]
    Then, using Lemma~\ref{lemma:cb}.
\end{proof}
\subsection{Proofs in Section~\ref{sec:paos}}\label{sec:po3.3}
Recall that $(\SN,\left<\cdot,\cdot\right>_F)$  is a Hilbert space on which we define a linear operator $$\FF_L(X)=\E\left[(A-BL)X(A-BL)\T\right]$$ to address the stability issue of the closed loop system~\eqref{eq:cls}. 
Next, we define another linear operator $\TT_L$ as 
\[
    \TT_L(X)=X+\E\left[\sum_{t=0}^{\infty}\prod_{i=0}^t(A_{i+1}-B_{i+1}L) X\left(\prod_{i=0}^t(A_{i+1}-B_{i+1}L)\right)\T\right].  
\]
where $\prod_{i=0}^t(A_{i+1}-B_{i+1}L)=(A_{t+1}-B_{t+1}L)\cdots(A_1-B_1L).$

The conjugate operators of $\FF_L$ and $\TT_L$ can be derived accordingly: 
\begin{align}
    \FF_L^*(X)&=\E\left[(A-BL)\T X(A-BL)\right],\\
    \TT_L^*(X)&=X+\E\left[\sum_{t=0}^{\infty}\left(\prod_{i=0}^t(A_{i+1}-B_{i+1}L)\right)\T X\left(\prod_{i=0}^t(A_{i+1}-B_{i+1}L)\right)\right].
\end{align}
Next lemma shows that operators defined above can characterize several key quantities. 
\begin{lemma}
    \label{lemma:oac}
    Suppose $L\in U_{ad}$, then it holds that, 
    \begin{enumerate}
        \item the state covariance matrix and aggregate covariance matrix can be expressed as 
	\begin{equation}
		\label{eq:oac1}
				\Sigma_{t}=\FF^t_L(\Sigma_0),\qquad \Sigma_L=\TT_L(\Sigma_0).
	\end{equation}
    \item
    \begin{equation}\label{eq:or}
        \TT_L=\sum_{t=0}^{\infty}\FF_L^t=(I_n -\FF_L)^{-1},
    \end{equation}
    where $I_n$ denotes the identity operator on $\SN$.
    \item $P_L$, the solution to equation~\eqref{eq:arelinear}, can be expressed as 
    \[
	P_L=\TT_L^*(\bqlrl).
	\]
    \item The cost $C(L)$ can be written as 
    \[
        C(L) = \left<\Sigma_0, \TT_L^*(\bqlrl)\right>_F=\left<\Sigma_L, \bqlrl\right>_F.
    \]
    \end{enumerate}
\end{lemma} 
\begin{proof}
	$\quad$

	{\bf For $1$}. We prove the first equality in \eqref{eq:oac1} by induction. The base case when $t=0$ is clear.  For $t+1$ case,
	\[
	\begin{aligned}
	\Sigma_{t+1}&=\E\left[x_{t+1} x_{t+1}\T \right]	\\
	&=\E\left[\E\left[(A_{t+1}-B_{t+1}L)x_{t} x_{t}\T (A_{t+1}-B_{t+1}L)\T \big|x_t\right]\right]\\
	&=\E\left[\FF_L(x_{t} x_{t}\T )\right]\\
	&=\FF_L(\Sigma_t).
	\end{aligned}	
	\]
	Then using inductive hypotheses. 
	
	For second equality in \eqref{eq:oac1}.
	We claim that: for any $X\in \SN$
	\[
	\E\left[\prod_{i=0}^t(A_{i+1}-B_{i+1}L) X\left(\prod_{i=0}^t(A_{i+1}-B_{i+1}L)\right)\T\right]=\FF_L^{t+1}(X).	
	\] 
	Now, we prove this claim by induction. When $t=0$ the base case is clear by definition of $\FF_L$. Then  
	by independency and inductive hypotheses, we obtain that 
	\[
	\begin{aligned}
		&\E\left[\prod_{i=0}^t(A_{i+1}-B_{i+1}L) X\left(\prod_{i=0}^t(A_{i+1}-B_{i+1}L)\right)\T\right]\\
		=&\E\left[(A_{t+1}-B_{t+1}L)\E\left[\prod_{i=0}^{t-1}(A_{i+1}-B_{i+1}L) X\left(\prod_{i=0}^{t-1}(A_{i+1}-B_{i+1}L)\right)\big| A_{t+1},B_{t+1}\T\right](A_{t+1}-B_{t+1}L)\T \right]\\
		=&\E\left[(A-BL)\FF_L^{t}(X)(A-BL)\T \right]\\
		=&\FF_L^{t+1}(X).
	\end{aligned}	
	\] 

	With this claim, we can prove the second equality. 
	Since any symmetric matrix can be represented as subtraction of two positive semi-definite matrices, we can without loss of generality assume $X$ is positive semi-definite. 
	For any $z\in \R^n$, by monotone convergence theorem
	\[
	\begin{aligned}
		z\T \TT_L(X)z&=\sum_{t=0}^{\infty}z\T\E\left[\prod_{i=0}^t(A_{i+1}-B_{i+1}L) X\left(\prod_{i=0}^t(A_{i+1}-B_{i+1}L)\right)\T \right]z + z\T X z\\
		&=\sum_{t=1}^{\infty}z\T\FF_L^{t}(X)z + z\T X z\\
		&=z\T \left(\sum_{t=0}^{\infty}\FF_L^{t}(X)\right) z.
	\end{aligned}	
	\]
	Combining above, we obtain 
	\begin{equation}
		\label{eq:oac2}
			\Sigma_L=\sum_{t=0}^{\infty}\Sigma_t=	\sum_{t=0}^{\infty} \FF_L^{t}(\Sigma_0)=\TT_L(\Sigma_0).
	\end{equation}
	This concludes the first statement in this Lemma.

	{\bf For $2$}. Note that 
	\[
	\sum_{t=0}^{N}\FF_L^t=I+\left(\sum_{t=0}^{N-1}\FF_L^t\right)\circ\FF_L.	
	\]
	Taking limit we get 
	\[
	\TT_L=I_n+\TT_L\circ \FF_L \implies \TT_L\circ (I_n-\FF_L)=I_n	.
	\]
	Similarly, 
	\[
		(I_n-\FF_L)\circ\TT_L =I_n,
	\]
	Thus the second statement in this Lemma is proved.

	{\bf For $3$}. By monotone convergence theorem and property of conditional expectation, we get  
	\[
		\begin{aligned}
			x_0\T P_L x_0&=\E\left[\sum_{t=0}^{\infty}x_t\T (Q_{t+1}+L\T R_{t+1}L)x_t\big| x_0\right]\\
			&=\sum_{t=0}^{\infty}\E\left[\E[x_t\T (Q_{t+1}+L\T Q_{t+1}L)x_t \big| x_t,x_0]\big| x_0\right]\\
			&=\sum_{t=0}^{\infty}\E\left[x_t\T(\bqlrl) x_t | x_0\right]\\
			&=\sum_{t=0}^{\infty}\E\left[x_0\T(A_1-B_1L)\T \cdots(A_t-B_tL)\T (\bqlrl)(A_t-B_tL)\cdots (A_1-B_1L)x_0 | x_0\right]\\
			&=x_0\T\E\left[\sum_{t=0}^{\infty}(A_1-B_1L)\T \cdots(A_t-B_tL)\T (\bqlrl)(A_t-B_tL)\cdots (A_1-B_1L) \right]x_0\\
			&=x_0\T \TT_L^*(\bqlrl) x_0, \quad \forall x_0\in\R^n.
		\end{aligned}
	\] 
	Hence 
	\begin{equation}
		\label{eq:oac3}
		P_L = \TT_L^*(\bqlrl).
	\end{equation}
	This concludes the third statement in this Lemma.

	{\bf For $4$}. By~\eqref{eq:oac3} and~\eqref{eq:closedformcost}, it holds that 
	\[
		C(L)=\tr(\Sigma_0\T P_L)=\left<\Sigma_0,P_L\right>_F=\left<\Sigma_0,\TT_L^*(\bqlrl)\right>_F.
	\]
	Finally, using definition of adjoint operator and~\eqref{eq:oac2}, we have
	\[
		C(L)=\left<\Sigma_0,\TT_L^*(\bqlrl)\right>_F=\left<\TT_L(\Sigma_0), \bqlrl\right>_F=\left<\Sigma_L, \bqlrl\right>_F.
	\] 
\end{proof}
Note that the infinite summation in~\eqref{eq:or} is well-defined, since $\rho(\FF_L)<1$ by definition of  $U_{ad}$. 
\begin{remark}
	Our goal is to bound the perturbation $\|\Sigma_L-\Sigma_{L'}\|$. Now by~\eqref{eq:oac1}, we can represent it using operator $\FF_L$ .
	\[
		\|\Sigma_L-\Sigma_{L'}\|=\|(\TT_L-\TT_{L'})(\Sigma_0)\|=\|(I_n-\FF_L)^{-1}-((I_n-\FF_{L'})^{-1})(\Sigma_0)\|
	\] for any $L,L'\in U_{ad}$.
\end{remark}
Next, two lemmas bound the perturbations of $\FF_L$ and $\TT_L$.
\begin{lemma}[$\FF_L$ perturbation]
	\label{lemma:fp}
	Suppose $L,L'\in U_{ad}$. The following perturbation bound holds
	\[
	\|\FF_L-\FF_{L'}\|\leq 2\overline{\|A-BL\|\|B\|}\|L-L'\|+\bbs \|L-L'\|^2.	
	\]
\end{lemma}
\begin{proof}
	Let $\Delta=L'-L$. For any $X\in\SN$, we have
	\[
	\begin{aligned}
	(\FF_L-\FF_{L'})(X)&=\E\left[(A-BL)X(A-BL)\T\right]	-\E\left[(A-BL')X(A-BL')\T\right]	\\
	&=\E\left[(A-BL)X(A-BL'+B\Delta)\T\right]	-\E\left[(A-BL')X(A-BL')\T\right]	\\
	&=\E\left[(A-BL)X(A-BL')\T +(A-BL)X(B\Delta)\T-(A-BL')X(A-BL')\T\right]	\\
	&=\E\left[(B\Delta)X(A-BL')\T +(A-BL)X(B\Delta)\T\right]	\\
	&=\E\left[(B\Delta)X(A-BL-B\Delta)\T +(A-BL)X(B\Delta)\T\right]	\\
	&=\E\left[(B\Delta)X(A-BL)\T +(A-BL)X(B\Delta)\T-(B\Delta)X(B\Delta)\T \right].
	\end{aligned}	
	\]
	The operator norm is 
	\[
		\begin{aligned}
		\|\FF_L-\FF_{L'}\|&=\sup_{X\in\SN}\frac{\|(\FF_L-\FF_{L'})(X)\|}{\|X\|}\\
		&\leq 2\E\left[\|A-BL\|\|B\|\right]\|L-L'\|+\E\left[\|B\|^2\right]\|L-L'\|^2.
		\end{aligned}
	\]
\end{proof}
\begin{lemma}[$\TT_L$ perturbation]
	\label{lemma:tp}
		Suppose $L,L'\in U_{ad}$ and $\|\TT_L\|\|\FF_L-\FF_{L'}\|\leq \frac{1}{2}$, then 
		\[
		\begin{aligned}
		\|(\TT_L-\TT_{L'})(X)\|&\leq 2\|\TT_L\|\|\FF_L-\FF_{L'}\|\|\TT_L(X)\|\\
		&\leq 2\|\TT_L\|^2\|\FF_L-\FF_{L'}\|\|X\|	.
		\end{aligned}	
		\]
\end{lemma}
\begin{proof}
	Define $\AAA=I_n-\FF_L$, and $\BBB=\FF_{L'}-\FF_L$. In this case by~\eqref{eq:or}, $\AAA^{-1}=\TT_L$ and $(\AAA-\BBB)^{-1}=\TT_{L'}$. Hence, the condition 
	$\|\TT_L\|\|\FF_L-\FF_{L'}\|\leq 1/2$ translate to the condition $\|\AAA^{-1}\|\|\BBB\|\leq 1/2.$ 

	Observe that 
	\[
		\begin{aligned}
			(\TT_L-\TT_{L'})(X)&=\left(\AAA^{-1}-(\AAA-\BBB)^{-1}\right)(X)	\\
			&=\left(I_n-(I_n-\AAA^{-1}\circ \BBB)^{-1}\right)(\AAA^{-1}(X))\\
			&=\left(I_n-(I_n-\AAA^{-1}\circ \BBB)^{-1}\right)(\TT_L(X)).
		\end{aligned}
	\]
	Since $(I_n-\AAA^{-1}\circ \BBB)^{-1}=I_n+\AAA^{-1}\circ\BBB\circ (I_n-\AAA^{-1}\circ \BBB)^{-1}$,
	\[
	\|(I_n-\AAA^{-1}\circ \BBB)^{-1}\|	\leq 1+\|\AAA^{-1}\circ \BBB\|\|(I_n-\AAA^{-1}\circ \BBB)^{-1}\|\leq 1+1/2\|(I_n - \AAA^{-1}\circ \BBB)^{-1}\|, 
	\]
	which implies $\|(I_n-\AAA^{-1}\circ \BBB)^{-1}\|\leq 2$. Hence, 
	\[
	\begin{aligned}
		\|I_n-(I_n-\AAA^{-1}\circ \BBB)^{-1}\|=\|\AAA^{-1}\circ\BBB\circ (I_n-\AAA^{-1}\circ \BBB)^{-1}\|\leq 2\|\AAA^{-1}\|\|\BBB\|,
	\end{aligned}	
	\]
	and hence, 
	\[
		\|I_n-(I_n-\AAA^{-1}\circ \BBB)^{-1}\|\leq 2\|\TT_L\|\|\FF_L-\FF_{L'}\|.
	\]
	Combing above 
	\[
	\begin{aligned}
		\|(\TT_L-\TT_{L'})(X)\|&\leq\|I_n-(I_n-\AAA^{-1}\circ \BBB)^{-1}\|\|\TT_L(X)\|\\
		&\leq2\|\TT_L\|\|\FF_L-\FF_{L'}\|\|\TT_L(X)\|\\
		&\leq 2\|\TT_L\|^2\|\FF_L-\FF_{L'}\|\|X\|	.
	\end{aligned}	
	\]
\end{proof}
Next Lemma bound the operator norm of $\TT_L$ with cost and parameters.
\begin{lemma}[$\TT_L$ norm bound]
	\label{lemma:tnb}
	The following bound holds for any mean-squared stabilizing matrices $L$:
	\[
	\|\TT_L\|\leq \frac{C(L)}{\mu\cdot \us_{\bq}}	.
	\] 
\end{lemma}
\begin{proof}
	For any unit norm vector $v\in \R^n$ and unit operator norm matrix $X$, by monotone convergence theorem, 
		\begin{align*}
			v\T \TT_L(X)v&=v\T Xv+\E\left[\sum_{t=0}^{\infty}v\T \prod_{i=0}^t(A_{i+1}-B_{i+1}L) X\left(\prod_{i=0}^t(A_{i+1}-B_{i+1}L)\right)\T v\right]\\
			&= v\T Xv + \sum_{t=0}^{\infty}\E\left[v\T \prod_{i=0}^t(A_{i+1}-B_{i+1}L) X\left(\prod_{i=0}^t(A_{i+1}-B_{i+1}L)\right)\T v\right]\\
			&=\tr(vv\T X) + \sum_{t=0}^{\infty}\E\left[\tr(\left(\prod_{i=0}^t(A_{i+1}-B_{i+1}L)\right)\T vv\T \prod_{i=0}^t(A_{i+1}-B_{i+1}L) X)\right]\\
			&=\sum_{t=0}^{\infty}\E\left[\tr(\Sigma_0^{1/2}\left(\prod_{i=0}^t(A_{i+1}-B_{i+1}L)\right)\T vv\T \prod_{i=0}^t(A_{i+1}-B_{i+1}L) \Sigma_0^{1/2}\Sigma_0^{-1/2}X\Sigma_0^{-1/2})\right]\\
			&\qquad +\tr(\Sigma_0^{1/2}vv\T \Sigma_0^{1/2}\Sigma_0^{-1/2}X\Sigma_0^{-1/2})\\
			&\leq \|\Sigma_0^{-1/2}X\Sigma_0^{-1/2}\|\left(v\T \Sigma_0 v +\sum_{t=0}^{\infty}\E\left[v\T \prod_{i=0}^t(A_{i+1}-B_{i+1}L) \Sigma_0 \left(\prod_{i=0}^t(A_{i+1}-B_{i+1}L)\right)\T v\right]\right)\\
			&=\|\Sigma_0^{-1/2}X\Sigma_0^{-1/2}\|(v\T \sum_{t=0}^{\infty} \FF_L^t (\Sigma_0) v) \\
			&=\|\Sigma_0^{-1/2}X\Sigma_0^{-1/2}\|(v\T \Sigma_L v) \\
			&\leq \frac{\|\Sigma_L\|}{\mu}\\
			&\leq \frac{C(L)}{\mu \cdot \us_{\bq}},
		\end{align*}
	where in the last inequality we use the cost bounds lemma \ref{lemma:cb}.
\end{proof}

With these lemma we can prove a weak version of $\Sigma_L$ perturbation bound. The following lemma is weak in the sense that 
it requires $L$ and $L'$ both in $U_{ad}$. Actually, Lemma~\ref{lemma:paos} is a strong version of this lemma in which the the requirement that $L'\in U_{ad}$ is redundant, 
i.e. if $L'$ is close enough to $L\in U_{ad}$, then $L'\in U_{ad}$ automatically.
\begin{lemma}[Weak $\Sigma_L$ perturbation]
	\label{lemma:spw}
	Suppose $L,L'\in U_{ad}$ and
	\[
	\|L'-L\|\leq \hdelta(L) ,
	\]
	with 
	\[
		\hdelta(L):=\frac{\us_{\bq}\cdot \mu}{4(\bbs)^{1/2}\left(1+(\abs)^{1/2}\right)C(L)}.
	\]
	It holds that 
	\[
		\begin{aligned}
			\|\Sigma_{L'}-\Sigma_L\|\leq 4\left(\frac{C(L)}{\us_{\bq}}\right)^2\frac{(\bbs)^{1/2}\left(1+(\abs)^{1/2}\right)}{\mu}\|L-L'\|
			\leq \frac{C(L)}{ \us_{\bq} }.
		\end{aligned}
	\]
\end{lemma}
\begin{proof}
	Lemma \ref{lemma:cb} gives $\mu\leq \|\Sigma_L\|\leq\frac{C(L)}{\us_{\bq}}$, hence
	\[
	(\bbs)^{1/2}\|L-L'\|\leq \frac{\us_{\bq}\cdot \mu}{4C(L)}\leq \frac{1}{4}	.
	\]
	Using $\FF_L$ perturbation Lemma \ref{lemma:fp} and Cauchy-Schwartz inequality, we have 
	\[
	\begin{aligned}
		\|\FF_L-\FF_{L'}\|&\leq 2\overline{\|A-BL\|\|B\|}\|L-L'\|+\bbs \|L-L'\|^2	\\
		&\leq 2(\abs)^{1/2}(\bbs)^{1/2}\|L-L'\|+\bbs \|L-L'\|^2	\\
		&\leq 2(\bbs)^{1/2}\left((\abs)^{1/2}+1/2(\bbs)^{1/2}\|L-L'\| \right)\|L-L'\|\\ 
		&\leq 2(\bbs)^{1/2}\left((\abs)^{1/2}+1 \right)\|L-L'\|.
	\end{aligned}	
	\]
	Using $\TT_L$ norm bound Lemma \ref{lemma:tnb}
	\[
		\|\TT_L\|\|\FF_L-\FF_{L'}\|\leq 2\frac{C(L)}{\us_{\bq}\cdot \mu}(\bbs)^{1/2}\left((\abs)^{1/2}+1 \right)\|L-L'\|\leq 1/2.
	\]
	Thus, by $\TT_L$ perturbation Lemma \ref{lemma:tp}, Lemma \ref{lemma:tnb} and Lemma \ref{lemma:cb}
	\[
	\begin{aligned}
	\|\Sigma_{L'}-\Sigma_L\|&=\|(\TT_{L'}-\TT_L)(\Sigma_0)\|\\
		&\leq 2\|\TT_L\|\|\FF_L-\FF_{L'}\|\|\TT_L(\Sigma_0)\|\\
		&\leq  2\frac{C(L)}{\us_{\bq}\cdot \mu}\left[2(\bbs)^{1/2}\left((\abs)^{1/2}+1 \right)\|L-L'\|\right]\frac{C(L)}{\us_{\bq}}\\
		&=4\left(\frac{C(L)}{\us_{\bq}}\right)^2\frac{(\bbs)^{1/2}\left(1+(\abs)^{1/2}\right)}{\mu}\|L-L'\|.
	\end{aligned}	
	\]
\end{proof}

We say an operator $\TT: \SN \to \SN $ is matrix monotone if $\TT(X)\succeq \TT(Y)$ whenever $X \succeq Y$. The following lemma shows that $\FF_L^t$ is matrix monotone for each $t\geq 0$.
\begin{lemma}
	\label{lemma:mof}
	Suppose $L\in U_{ad}$, then $\FF_L^t$ is matrix monotone for each $t\geq 0$.
\end{lemma}
\begin{proof}
	The proof is similar to the proof of Theorem 3.1 in~\cite{de1982}. 
\end{proof}
\begin{lemma}[$\Sigma_L$ trace bound]
	\label{lemma:stb}
	Suppose $L\in U_{ad}$, then 
\[
\tr(\Sigma_L)\geq \frac{\mu}{1-\rho(\FF_L)}.	
\]
\end{lemma}
\begin{proof}  
	By property of spectral radius~\cite{horn2012} and Lemma 2.1c in~\cite{de1982}, we have 
	\[
		\rho(\FF_L)^t=\rho(\FF_L^t)\leq \|\FF_L^t\|=\|\FF_L^t(I_n)\|.
	\]
	Thus
	\[
		\tr(\FF_L^t(I_n))\geq \|\FF_L^t(I_n)\|\geq \rho(\FF_L)^t.
	\]
 	By definition of $\mu$ we have
	\[
	\mu I_n \preceq \Sigma_0,
	\]
	then Lemma \ref{lemma:mof} implies 
	\[
		\mu\cdot \FF_L^t(I_n) \preceq \FF_L^t(\Sigma_0),
	\]
	for each $t\geq 0$. 
	Hence, 
	\[
	\mu \sum_{t=0}^{\infty}\FF_L^t(I)\preceq 	\sum_{t=0}^{\infty}\FF_L^t(\Sigma_0)=\Sigma_L.
	\]
	Combing above, we obtain that 
	\[
	\tr(\Sigma_L)\geq \mu \sum_{t=0}^{\infty}\tr(\FF_L^t(I))\geq \mu \sum_{t=0}^{\infty}\rho(\FF_L)^t=\frac{\mu}{1-\rho(\FF_L)}.
	\]
\end{proof}
With above preparing lemmas, we can prove Lemma~\ref{lemma:paos} now.
\begin{proof}[Proof of Lemma~\ref{lemma:paos}]
		First, we prove the first statement, i.e., if $\rho(\FF_L)<1$ and 
		\[
		\|L'-L\|\leq \hdelta(L),
		\]
		then $\rho(\FF_{L'})<1.$
	
		Suppose  $L''\neq L$ with $\rho(L'')<1$ and 
		\[
		\|L''-L\|\leq \hdelta(L).
		\]
		Using weak version perturbation Lemma \ref{lemma:spw} we obtain
		\[
		|\tr(\Sigma_{L''}-\Sigma_L)|\leq n\|\Sigma_{L''}-\Sigma_L\|	\leq n\frac{C(L)}{\us_{\bq}},
		\]
		hence
		\[
		\tr(\Sigma_{L''})\leq  \tr(\Sigma_L)	+n\frac{C(L)}{\us_{\bq}}.
		\]
		We denote $K:=\tr(\Sigma_L)	+n\frac{C(L)}{\us_{\bq}}$.
		
		$\Sigma_L$ trace bound Lemma \ref{lemma:stb} gives 
		\[
		\tr(\Sigma_L)\geq \frac{\mu}{1-\rho(\FF_L)}=\frac{2K\cdot \varepsilon}{1-\rho(\FF_L)}	,
		\]
		where $\varepsilon:=\frac{\mu}{2K}$.
		Thus 
		\[
		\rho(\FF_L)\leq 1-	\frac{2K\cdot \varepsilon}{\tr(\Sigma_L)}=1-\frac{2\varepsilon(\tr(\Sigma_L)	+n\frac{C(L)}{\us_{\bq}})}{\tr(\Sigma_L)}\leq 1-2\varepsilon<1-\varepsilon.
		\]
	Now suppose $\exists L'$ with $\rho(\FF_{L'})\geq 1$ and 
	\[
	\|L'-L\|\leq \hdelta(L)	.
	\]
	Since spectral radius is continuous, by intermediate value theorem $\exists L'''$ on tha path between $L$ and $L'$ such that 
	$\rho(\FF_{L'''})=1-\epsilon<1$ and $\| L'''-L\|\leq \hdelta(L)$. So we have 
	\[
	|\tr(\Sigma_{L'''}-\Sigma_L)|\leq 	n\frac{C(L)}{\us_{\bq}}\quad \implies \quad \tr(\Sigma_{L'''})\leq K.
	\]
	However, 
	\[
	\tr(\Sigma_{L'''})\geq \frac{\mu}{1-\rho(\FF_{L'''})}=\frac{2\varepsilon\cdot K}{1-(1-\varepsilon)}=2K.	
	\]
	This is a contradiction.
	
	The second part of this lemma follows from Lemma \ref{lemma:spw} immediately.
\end{proof}
\subsection{Proofs in Section~\ref{sec:caombm}}\label{sec:po3.4}
\begin{proof}[Proof of Lemma~\ref{lemma:osu}]
	By almost smoothness Lemma \ref{lemma:as} 
	\[
	\begin{aligned}
		C(L')-C(L)&=-4\eta \tr(\Sigma_{L'}\Sigma_L E_L\T E_L)+4\eta^2\tr(\Sigma_{L'}\Sigma_LE_L\T R_LE_L\Sigma_L)\\
		&=-4\eta \tr(\Sigma_{L}\Sigma_L E_L\T E_L)-4\eta \tr((\Sigma_{L'}-\Sigma_L)\Sigma_L E_L\T E_L)\\
		&\qquad +4\eta^2\tr(\Sigma_{L'}\Sigma_LE_L\T R_LE_L\Sigma_L)\\
		&\leq -4\eta \tr(\Sigma_L E_L\T E_L\Sigma_{L})+4\eta \|\Sigma_{L'}-\Sigma_L\|\tr(\Sigma_L E_L\T E_L)\\
		&\qquad +4\eta^2\|\Sigma_{L'}\|\|R_L\|\tr(\Sigma_LE_L\T E_L\Sigma_L)\\
		&\leq -4\eta \tr(\Sigma_L E_L\T E_L\Sigma_{L})+4\eta \frac{\|\Sigma_{L'}-\Sigma_L\|}{\mu}\tr(\Sigma_L E_L\T E_L\Sigma_L)\\
		&\qquad +4\eta^2\|\Sigma_{L'}\|\|R_L\|\tr(\Sigma_LE_L\T E_L\Sigma_L)\\
		&=-4\eta\left( 1-\frac{\|\Sigma_{L'}-\Sigma_L\|}{\mu}-\eta\|\Sigma_{L'}\|\|R_L\|\right)\tr(\Sigma_L E_L\T E_L\Sigma_L)\\
		&=-\eta\left( 1-\frac{\|\Sigma_{L'}-\Sigma_L\|}{\mu}-\eta\|\Sigma_{L'}\|\|R_L\|\right)\tr(\nabla C(L)\T \nabla C(L))\\
		&\leq -4\eta\left( 1-\frac{\|\Sigma_{L'}-\Sigma_L\|}{\mu}-\eta\|\Sigma_{L'}\|\|R_L\|\right)\frac{\mu^2 \us_{\br}}{\|\Sigma_{L^*}\|}(C(L)-C(L^*)),
	\end{aligned}	
	\]
	where in the last inequality we use Lemma \ref{lemma:gd}. 
	By the assumed condition on the step size and Lemma~\ref{lemma:cb}, we get 
	\[
	\|L'-L\|=\eta\|\nabla C(L)\| \leq \frac{1}{16}\left(\frac{\us_{\bq}\cdot \mu}{C(L)}\right)^2\frac{\|\nabla C(L)\|}{\|\nabla C(L)\|(\bbs)^{1/2}\left(1+(\abs)^{1/2}\right)}\leq \hdelta(L),
	\]
	which is the condition for Lemma \ref{lemma:spw}. 

	Next using the assumed condition on $\eta$ again, we have
	\[
		\frac{\|\Sigma_{L'}-\Sigma_L\|}{\mu}	\leq 4\eta \left(\frac{C(L)}{\us_{\bq}\cdot \mu}\right)^2(\bbs)^{1/2}\left(1+(\abs)^{1/2}\right)\|\nabla C(L)\|\leq \frac{1}{4}.
	\]
	Using last claim and Lemma \ref{lemma:cb} 
	\[
	\|\Sigma_{L'}\|\leq \|\Sigma_{L'}-\Sigma_L\|+\|\Sigma_L\|\leq \frac{1}{4}\mu+\frac{C(L)}{\us_{\bq}}\leq \frac{\|\Sigma_{L'}\|}{4}+\frac{C(L)}{\us_{\bq}},	
	\]
	and thus $\|\Sigma_{L'}\|\leq \frac{4C(L)}{3\us_{\bq}}$. Hence, 
	\[
		1-\frac{\|\Sigma_{L'}-\Sigma_L\|}{\mu}-\eta\|\Sigma_{L'}\|\|R_L\|\geq 1-\frac{1}{4}-\eta\frac{4C(L)\|R_L\|}{3\us_{\bq}}\geq \frac{3}{4}-\frac{1}{12}>\frac{1}{2}.
	\]
	Therefore, 
	\[
		C(L')-C(L^*)=C(L')-C(L)+(C(L)-C(L^*))\leq \left(1-2\eta \frac{\mu^2\us_{\br}}{\|\Sigma_{L^*}\|}\right)	(C(L)-C(L^*)).
	\]
\end{proof}
The following Lemma~\ref{lemma:bocag} and~\ref{lemma:boss} indicate that some crucial quantities can be upper bounded on $S_{\gamma}$ by model parameters and $\gamma$. 
\begin{lemma}
    \label{lemma:bocag}
     We have the following upper bounds: 
     \begin{enumerate}
        \item \[
	\|\nabla C(L)\|\leq \|\nabla C(L)\|_F\leq h_1(L),	
	\]
    where 
    \[
        h_1(L):=2\frac{C(L)}{\us_{\bq}}\sqrt{\frac{\|R_L\|}{\mu}(C(L)-C(L^*))} .
    \]
    \item \[
	\|L\|\leq \frac{1}{\us_{\br}}	\left( \sqrt{\frac{\|R_L\|}{\mu}(C(L)-C(L^*))}+\overline{\|B\T P_L A\|}\right).
	\]
     \end{enumerate}
	
\end{lemma}
\begin{proof}
	For $1$. Using Lemma \ref{lemma:cb} we have 
	\[
	\begin{aligned}
	\|\nabla C(L)\|^2\leq \|\nabla C(L)\|_F^2=4\tr(\Sigma_L E_L\T E_L \Sigma_L)\leq 4\|\Sigma_L\|^2\tr(E_L\T E_L)\leq 4\left(\frac{C(L)}{\us_{\bq}}\right)^2\tr(E_L\T E_L).	
	\end{aligned}	
	\]
	By Lemma \ref{lemma:gd} we obtain
	\[
	\|\nabla C(L)\|^2\leq 	4\left(\frac{C(L)}{\us_{\bq}}\right)^2\frac{\|R_L\|}{\mu}(C(L)-C(L^*)),	
	\]
	which is the first inequality.

	For $2$. Recall in~\eqref{eq:RLEL} we defined that  
	\[
	\begin{aligned}
		R_L&:=\rl,\\
		E_L&:=R_L L-\bbpla.
	\end{aligned}
	\]
	Hence, using Lemma \ref{lemma:gd} again we get
	\[
	\begin{aligned}
	\|L\|\leq \|R_L^{-1}\|\|R_LL\|&\leq \frac{1}{\us_{\br}}(\|E_L\|+\|\overline{B\T P_L A}\|)\\
		&\leq \frac{1}{\us_{\br}}(\tr(E_L\T E_L )^{1/2}+\|\overline{B\T P_L A}\|)\\
		&\leq \frac{1}{\us_{\br}}\left( \sqrt{\frac{\|R_L\|}{\mu}(C(L)-C(L^*))}+\|\overline{B\T P_L A}\|\right).
	\end{aligned}	
	\]
\end{proof}
\begin{lemma}
    \label{lemma:boss}
    For any $L\in S_{\gamma}$, it holds that
	\[
	\begin{aligned}
		&\|R_L\|\leq M_{R_L}(\gamma),\\
		&\|\nabla C(L)\|\leq M_{\nabla}(\gamma),\\
		&\|L\|\leq M_L(\gamma),\\
		&\abs \leq M_{ABL}(\gamma),
	\end{aligned}	
	\]
	where 
	\[
	\begin{aligned}
		&M_{R_L}(\gamma)=\|\br\|+\bbs\frac{\gamma}{\mu}\leq \cc_{B,R,\frac{1}{\mu}}(1+\gamma),\\
		&M_{\nabla}(\gamma)=2\frac{\gamma}{\us_{\bq}}\sqrt{\frac{M_{R_L}(\gamma)}{\mu}(\gamma-C(L^*))}\leq \cc_{B,Q,R,\frac{1}{\mu}}(1+\gamma)^2,\\
		&M_L(\gamma)=\frac{1}{\us_{\br}}	\left( \sqrt{\frac{M_{R_L}(\gamma)}{\mu}(\gamma-C(L^*))}+\overline{\|A\|\|B\|}\frac{\gamma}{\mu}\right)\leq \cc_{A,B,R,\frac{1}{\mu}}(1+\gamma),\\
		&M_{ABL}(\gamma)=2(\overline{\|A\|^2}+\bbs M_L(\gamma)^2)\leq \cc_{A,B,R,\frac{1}{\mu}}(1+\gamma)^2.
	\end{aligned}	
	\]
\end{lemma}
\begin{proof}
	For $\|R_L\|$, by Lemma \ref{lemma:cb} we have 
	\[
	\begin{aligned}
	\|R_L\|&=\|\br+\overline{B\T P_L B}\|	\leq \|\br\|+\bbs\|P_L\|\leq\|\br\|+\bbs\frac{\gamma}{\mu}=M_{R_L}(\gamma),
	\end{aligned}	
	\]and 
	\[
		M_{R_L}(\gamma)\leq \cc_{B,R,\frac{1}{\mu}}(1+\gamma).
	\]
	For $\nabla C(L)$, by Lemma \ref{lemma:bocag} we have 
	\[
		\|\nabla C(L)\|\leq 2\frac{C(L)}{\us_{\bq}}\sqrt{\frac{\|R_L\|}{\mu}(C(L)-C(L^*))}\leq 2\frac{\gamma}{\us_{\bq}}\sqrt{\frac{M_{R_L}(\gamma)}{\mu}(\gamma-C(L^*))}=M_{\nabla}(\gamma),
	\]and we see that 
	\[
		M_{\nabla}(\gamma) = 2\frac{\gamma}{\us_{\bq}}\sqrt{\frac{M_{R_L}(\gamma)}{\mu}(\gamma-C(L^*))}\leq \cc_{B,Q,R,\frac{1}{\mu}}(1+\gamma)^2.
	\]
	For $\|L\|$, by Lemma \ref{lemma:bocag} we get 
	\[
	\begin{aligned}
		\|L\|&\leq \frac{1}{\us_{\br}}	\left( \sqrt{\frac{\|R_L\|}{\mu}(C(L)-C(L^*))}+\overline{\|B\T P_L A\|}\right)\\
		&\leq\frac{1}{\us_{\br}}	\left( \sqrt{\frac{M_{R_L}(\gamma)}{\mu}(\gamma-C(L^*))}+\overline{\|A\|\|B\|}\frac{\gamma}{\mu}\right)\\
		&=M_L(\gamma),
	\end{aligned}	
	\]and
	\[
		M_L(\gamma)=\frac{1}{\us_{\br}}	\left( \sqrt{\frac{M_{R_L}(\gamma)}{\mu}(\gamma-C(L^*))}+\overline{\|A\|\|B\|}\frac{\gamma}{\mu}\right)\leq \cc_{A,B,R,\frac{1}{\mu}}(1+\gamma).
	\]
	For $\abs$, by Lemma \ref{lemma:bocag} we have
	\[
	\begin{aligned}
	\abs&=\E\|A-BL\|^2\leq 2( \overline{\|A\|^2}+\bbs\|L\|^2)	\\
		&\leq 2(\overline{\|A\|^2}+\bbs M_L(\gamma)^2)\\
		&=M_{ABL}(\gamma),
	\end{aligned}	
	\]and 
	\[
		M_{ABL}(\gamma)\leq \cc_{A,B,R,\frac{1}{\mu}}(1+\gamma)^2.
	\]
\end{proof}
\subsection{Proofs in Section~\ref{sec:acaswfh}}\label{sec:po4.1}
\begin{proof}[Proof of Lemma~\ref{lemma:acaswfhr}]
	For $1$. By Lemma~\ref{lemma:oac}, it holds that 
	\[
	\Sigma_L=\TT_L(\Sigma_0).	
	\] 
	Since $\FF_L$ is linear , for any $z\in \R^n$ by monotone convergence theorem, we get 
	\[
	\begin{aligned}
		z\T \left(\Sigma_L-\Sigma^{(l)}_L\right) z&=\E\left[z\T \left(\sum_{t=l}^{\infty} x_t x_t\T\right)z\right]\\
		&=\E\left[\sum_{t=l-1}^{\infty} z\T(A_{t+1}-B_{t+1}L)x_t x_t\T (A_{t+1}-B_{t+1}L)z\right]\\
		&=\sum_{t=l-1}^{\infty}\E\left[\E\left[z\T (A_{t+1}-B_{t+1}L)x_t x_t\T (A_{t+1}-B_{t+1}L)\T z \big| x_t \right]\right]\\
		&=\sum_{t=l-1}^{\infty}\E\left[z\T \FF_L(x_t x_t\T) z\right]\\
		&=z\T \FF_L(\E[\sum_{t=l-1}^{\infty}x_t x_t\T]) z\\
		&=\cdots\\
		&=z\T \FF_L^l(\Sigma_L) z,
	\end{aligned}	
	\]
	hence 
	\[
		\Sigma^{(l)}_L=\Sigma_L-\FF_L^l(\Sigma_L).
	\]
	Lemma \ref{lemma:mof} and Lemma \ref{lemma:cb} ensures 
	\[
	\sum_{i=0}^{l-1}\tr(\FF_L^{i}(\Sigma_0))=\tr(\sum_{i=0}^{l-1}\FF_L^{i}(\Sigma_0))\leq 	\tr(\sum_{i=0}^{\infty}\FF_L^{i}(\Sigma_0))=\tr(\Sigma_L)\leq n\frac{C(L)}{\us_{\bq}},
	\]
	then there mus exists $0\leq j\leq l-1$ such that 
	\[
		\tr(\FF_L^{j}(\Sigma_0))\leq \frac{nC(L)}{l\cdot \us_{\bq}}.
	\]
	Using Lemma \ref{lemma:mof} and Lemma \ref{lemma:cb} again, we have  
	\[
	\Sigma_L=	\TT_L(\Sigma_0)\preceq\|\TT_L\|\Sigma_0\preceq \frac{C(L)}{\mu \cdot \us_{\bq}}\Sigma_0.
	\]
	Lemma~\ref{lemma:mof} ensures that $\FF_L^j$ is matrix monotone, thus  
	\[
	\tr(\FF_L^{j}(\Sigma_L))\leq 	\frac{C(L)}{\mu \cdot \us_{\bq}}\tr(\FF_L^{j}(\Sigma_0))\leq \frac{nC(L)^2}{l \mu (\us_{\bq})^2 }.
	\]
	It follows that 
	\[
		\|\Sigma_L-\Sigma^{(l)}_L\|\leq \|\Sigma_L-\Sigma^{(j)}_L\|	=\|\FF_L^{j}(\Sigma_L)\|\leq \tr(\FF_L^{j}(\Sigma_L)) \leq \frac{nC(L)^2}{l \mu (\us_{\bq})^2 },
	\]
	and 
	\[
		\tr(\Sigma_L-\Sigma^{(l)}_L)\leq \tr(\Sigma_L-\Sigma^{(j)}_L)	=\tr(\FF_L^{j}(\Sigma_L))\leq \frac{nC(L)^2}{l \mu (\us_{\bq})^2 }.	
	\]
	Therefore as long as 
	\[
	l\geq 	\frac{nC(L)^2}{\epsilon \mu\cdot (\us_{\bq})^2 },
	\]
	it holds that  
	\[
		\|\Sigma_L-\Sigma^{(l)}_L\|	\leq \epsilon.
	\]

	For $2$, using the routine conditional expectation trick, we obtain 
	\[
	\begin{aligned}
		C^{(l)}(L)&=\E[\sum_{t=0}^{l-1}x_t\T (Q_{t+1}+L\T R_{t+1} L)x_t]\\
		&=\E[\sum_{t=0}^{l-1}x_t\T\bqlrl x_t]\\
		&=\E[\sum_{t=0}^{l-1}\tr(x_t x_t\T\bqlrl )]\\
		&=\tr(\E[\sum_{t=0}^{l-1}x_t x_t\T]\bqlrl)\\
		&=\left< \Sigma^{(l)}_L,\bqlrl \right>_F.
	\end{aligned}	
	\]
	Similarly, using monotone convergence theorem, we have 
	\[
		C(L)=\left< \Sigma_L,\bqlrl \right>_F.
	\]
	Hence, 
	\[
		\begin{aligned}
			C(L)-C^{(l)}(L)&=\left< \Sigma_L-\Sigma^{(l)}_L,\bqlrl \right>_F\\
			&\leq \|\bqlrl\|\tr(\Sigma_L-\Sigma^{(l)}_L)\\
			&\leq\|\bqlrl\|\frac{nC(L)^2}{l \mu (\us_{\bq})^2 }\\
			&\leq (\|\bq\|+\bnr \|L\|^2)\frac{nC(L)^2}{l \mu (\us_{\bq})^2 }.
		\end{aligned}
	\]
	Therefore as long as 
	\[
		l\geq 	\frac{nC(L)^2(\|\bq\|+\bnr \|L\|^2)}{\epsilon \mu\cdot (\us_{\bq})^2},
	\]
	it holds that 
	\[
		C(L)\geq C^{(l)}(L)	\geq C(L)-\epsilon.
	\]
\end{proof}
\subsection{Proofs in Section~\ref{sec:paogac}}\label{sec:po4.2}
\begin{proof}[Proof of Lemma~\ref{lemma:cp}]
		As in the proof of Lemma \ref{lemma:spw}, the assumption implies that
		\[
			\begin{aligned}
				 &\|\TT_L\|\|\FF_L-\FF_{L'}\|\leq 1/2,\\
				 &\|\FF_L-\FF_{L'}\|\leq 2(\bbs)^{1/2}(1+(\abs)^{1/2})\|L-L'\|.
			\end{aligned}
		\]
		Using \eqref{eq:closedformcost} we have  
		\[
				C(L)=\left<\Sigma_0,P_L\right>_F, \qquad C(L')=\left<\Sigma_0,P_{L'}\right>_F,
		\]
		thus
		\begin{equation}
			\label{eq:cd}
		\begin{aligned}
		|C(L')-C(L)|&= |\left<\Sigma_0,P_{L'}-P_L\right>_F|\\
		&\leq \tr(\E[x_0x_0\T])\|P_{L'}-P_L\|\\
		&=\bnxs\|P_{L'}-P_L\|.
		\end{aligned}	
		\end{equation}
		
		By linearity of $\TT_L$, Lemma~\ref{lemma:tp} Lemma~\ref{lemma:oac} and the fact that $\|\TT^*\|=\|\TT\|$ for operator $\TT$, we obtain
		\begin{equation}
		\label{eq:pp}
		\begin{aligned}
		\|P_{L'}-P_L\|&=\|\TT_{L'}^*(\bqlrlp)-\TT_L^*(\bqlrlp)+\TT_L^*(\bqlrlp)-\TT_L^*(\bqlrl)\|\\
		&=\|\TT_{L'}^*(\bqlrlp)-\TT_L^*(\bqlrlp)+\TT_L^*(\blrlp-\blrl)\|\\
		&\leq \|\TT_{L'}^*-\TT_L^*\|\|\bqlrlp\|+\|\TT_L^*\|\|\blrlp-\blrl\|\\
		&=\|\TT_{L'}-\TT_L\|\|\bqlrlp\|+\|\TT_L\|\|\blrlp-\blrl\|\\
		&\leq 2\|\TT_L\|^2\|\FF_{L'}-\FF_L\|\|\bqlrl+\blrlp-\blrl\|+\|\TT_L\|\|\blrlp-\blrl\|\\
		&\leq 2\|\TT_L\|^2\|\FF_{L'}-\FF_L\|\|\bqlrl\|+\|\TT_L\|\|\blrlp-\blrl\|+\|\TT_L\|\|\blrlp-\blrl\|\\
		&=2\|\TT_L\|^2\|\FF_{L'}-\FF_L\|\|\bqlrl\|+2\|\TT_L\|\|\blrlp-\blrl\|.
		\end{aligned}			
		\end{equation}
	
	 For the first term, by Lemma \ref{lemma:fp} we obtain
	 \[
	\begin{aligned}
		&2\|\TT_L\|^2\|\FF_{L'}-\FF_L\|\|\bqlrl\|\\
		\leq &4\|\TT_L\|^2(\bbs)^{1/2}(1+(\abs)^{1/2})(\|\bq\|+\bnr \|L\|^2)\|L'-L\|.
	\end{aligned}	
	 \]
	 For the second term, 
	 \[
		\begin{aligned}
		&2\|\TT_L\|\|\blrlp-\blrl\|\\
		=&2\|\TT_L\|\|\overline{(L'-L)\T R (L'-L)}+\overline{(L')\T R L}+\overline{L\T R L'}-\blrl -\blrl \|\\
		\leq & 2\|\TT_L\|(\bnr \|L'-L\|^2+2\bnr \|L\|\|L'-L\|)\\
		\leq&2\|\TT_L\|\bnr (\|L'-L\|+2 \|L\|)\|L'-L\|\\
		\leq&2\|\TT_L\|\bnr (\Delta +2 \|L\|)\|L'-L\|.
		 \end{aligned}
	 \]
	 Combing above inequalities and Equation (\ref{eq:cd}) and Lemma \ref{lemma:tnb}, we have 
	 \[
		\begin{aligned}
			&|C(L')-C(L)|\\
			\leq& \bnxs \|\TT_L\|\left(4\|\TT_L\|(\bbs)^{1/2}(1+(\abs)^{1/2})(\|\bq\|+\bnr \|L\|^2)\right.\\
			&\left. +2\bnr (\Delta(L)+2 \|L\|)\right)\|L'-L\|\\
			\leq & C(L)\frac{\bnxs}{\mu\us_{\bq}}\left(4\frac{C(L)}{\mu\us_{\bq}}(\bbs)^{1/2}(1+(\abs)^{1/2})(\|\bq\|+\bnr \|L\|^2)\right. \\
			&\left.+2\bnr (\Delta(L)+2 \|L\|)\right)\|L'-L\|.
		\end{aligned}
	 \]
\end{proof}
\begin{proof}[Proof of Lemma~\ref{lemma:gp}]
	Recall that 
	\[
	\nabla C(L)=2E_L\Sigma_L, 
	\]
	where $E_L=\rl \cdot L-\bbpla$.
	Thus  
	\[
	\nabla C(L')-\nabla C(L)=2(E_{L'}\Sigma_{L'}-E_L\Sigma_L)=2\left[(E_{L'}-E_L)\Sigma_{L'}+E_L(\Sigma_{L'}-\Sigma_L)\right].	
	\]
	First, we bound the second term. Lemma \ref{lemma:gd} gives 
	\[
	\|E_L\|_F^2=\tr(E_L\T E_L)\leq \frac{\|R_L\|}{\mu}(C(L)-C(L^*))	.
	\]
	Thus using Lemma \ref{lemma:paos}, we obtain
	\[
	\|E_L(\Sigma_{L'}-\Sigma_L)\|\leq 4\sqrt{\frac{\|R_L\|}{\mu}(C(L)-C(L^*))}\left(\frac{C(L)}{\us_{\bq}}\right)^2\frac{(\bbs)^{1/2}(1+(\abs)^{1/2})}{\mu}\|L'-L\|.
	\]
	Next we bound the first term. Observe that 
	\[
	\|\Sigma_{L'}\|-\|\Sigma_L\|\leq \|\Sigma_{L'}-\Sigma_L\|\leq \frac{C(L)}{\us_{\bq}}, 	
	\]
	hence 
	\[
	\|\Sigma_{L'}\|\leq \|\Sigma_L\|+\frac{C(L)}{\us_{\bq}}\leq 2\frac{C(L)}{\us_{\bq}}.
	\]
	Equation (\ref{eq:pp}) in the proof of previous Lemma gives 
	\[
		\begin{aligned}
			\|P_{L'}-P_L\|\leq& 	\frac{C(L)}{\mu\cdot \us_{\bq}}\left[2\bnr (\|L-L’\|+2 \|L\|)\right.\\
			&+\left. 4\frac{C(L)}{\mu\cdot \us_{\bq}}(\bbs)^{1/2}(1+(\abs)^{1/2})(\|\bq\|+\bnr\|L\|^2)\right]\|L'-L\|.
		\end{aligned}
	\]
	 Note that 
 \[
	\begin{aligned}
		E_{L'}-E_L&=\overline{R+B\T P_{L'} B}\cdot L' -\overline{B\T P_{L'} A}-\rl\cdot L + \bbpla\\
		&=\br(L'-L)-\overline{B\T (P_{L'}-P_L) A} +\overline{B\T (P_{L'}-P_L) B}\cdot L'+\overline{B\T P_{L} B}(L'-L).
	\end{aligned}
 \]
 Hence, 
 \[
	\begin{aligned}
		\|E_{L'}-E_L\|\leq& \|\br\|\|L'-L\|+\overline{\|B\|\|A\|}\|P_{L'}-P_L\|+(\|L-L’\|+\|L\|)\bbs \|P_{L'}-P_L\|\\
		&+\bbs\|P_L\|\|L'-L\|\\
		\leq&\left(\|\br\|+\bbs\frac{C(L)}{\mu}+\left(\overline{\|B\|\|A\|}+\bbs(\|L-L’\|+\|L\|) \right)\frac{C(L)}{\mu\us_{\bq}}\left[2\bnr (\|L-L’\|+2 \|L\|)\right.\right.\\
		&\left.\left.+4\frac{C(L)}{\mu\cdot \us_{\bq}}(\bbs)^{1/2}(1+(\abs)^{1/2})(\|\bq\|+\bnr\|L\|^2)\right] \right)\|L'-L\|.
	\end{aligned}
 \]
 
 We define 
 \[
	\begin{aligned}
		\hgrad(L,\Delta):=&4\frac{C(L)}{\us_{\bq}}\left(\left(\overline{\|B\|\|A\|}+\bbs(\Delta+\|L\|) \right)\frac{C(L)}{\mu\us_{\bq}}\left[2\bnr (\Delta+2 \|L\|)\right.\right.\\
		& \left. \left.+4\frac{C(L)}{\mu\cdot \us_{\bq}}(\bbs)^{1/2}(1+(\abs)^{1/2})(\|\bq\|+\bnr\|L\|^2)\right]+\|\br\|+\bbs\frac{C(L)}{\mu} \right)\\
		&+8\sqrt{\frac{\|R_L\|}{\mu}(C(L)-C(L^*))}\left(\frac{C(L)}{\us_{\bq}}\right)^2\frac{(\bbs)^{1/2}(1+(\abs)^{1/2})}{\mu}.
	\end{aligned}
 \]
 Combining above we obtain 
 \[
\|\nabla C(L')-\nabla C(L)\|\leq 2(\|(E_{L'}-E_L)\Sigma_{L'}\|+\|E_L(\Sigma_{L'}-\Sigma_L)\|)\leq \hgrad(L,\Delta(L))\|L'-L\|.
 \] 
 Finally, we see that
 \[
	\begin{aligned}
	\|\nabla C(L')-\nabla C(L)\|_F\leq& \sqrt{m\land n}\|\nabla C(L')-\nabla C(L)\|\\
		\leq &\sqrt{m\land n} \cdot \hgrad(L,\Delta(L))\|L'-L\|\\
		\leq&\sqrt{m\land n} \cdot \hgrad(L,\Delta(L))\|L'-L\|_F.
	\end{aligned}
 \]
\end{proof}

\subsection{Proofs in Section~\ref{sec:sgosp}}\label{sec:po4.3}
\begin{proof}[Proof of Lemma~\ref{lemma:xtisg}]
    First, using~\eqref{eq:ass32}, 
	\begin{equation}
        \label{eq:xisg1}
        \begin{aligned}
		\|A-BL\|^p&\leq (\|A\|+\|B\| \|L\|)^p\\
		&\leq \left[K(1+\|L\|)\right]^p \quad \text{a.s.} \quad \forall p\geq 1. 
		\end{aligned}
    \end{equation}

	By independency, induction, and~\eqref{eq:xisg1}, we have 
	\[
		\begin{aligned}
			\|\overline{x_{t}}\|_2&=\|\E\left[(A_t-B_t L)x_{t-1}\right]\|\\
			&=\|\E\left[(A_t-B_t L)\right]\overline{x_{t-1}}\|_2\\
			&\leq \E\|A-BL\|\cdot \|\overline{x_{t-1}}\|_2	\\
			&\leq (1+\|L\|)K \|\overline{x_{t-1}}\|_2\\
			&\leq \left((1+\|L\|)K\right)^{t}\|\overline{x_0}\|_2.
		\end{aligned}
	\]

	Next, we claim that for any $u\in \R^n$ with $\|u\|_2=1$, for any $t\geq 1$  it holds that 
	\begin{equation}\label{eq:cl}
	\|\left<u,x_{t}-\overline{x_{t}}\right>\|_{L^p}^p \leq \left[K(1+\|L\|)\right]^{tp}\left(c_p^{2t}\sigma_0^pp^{p/2}+(\Sigma_{i=1}^{2t}c_p^i )\|\overline{x_0}\|_2^p\right),\quad \forall p\geq 1,
	\end{equation}
	where $c_p=2^{p-1}$.

	We prove this claim by induction. For base case $t=1$, 
	\[
	\begin{aligned}
	&\|\left<u,x_{1}-\overline{x_{1}}\right>\|_{L^p}^p\\
	=&\E\left[\left|\left<u,x_{1}-\overline{x_{1}}\right>\right|^p\right]\\
	\leq& c_p\left(\E\left[\left|\left<u,x_{1}\right>\right|^p\right]+\|\overline{x_{1}}\|_2^p\right)\\
	\leq& c_p\left(\E\left[\left|\left<u,(A_{1}-B_{1} L)x_0\right>\right|^p\right]+\|\overline{x_{1}}\|_2^p\right)\\
	\leq& c_p\left(\E\left[\left|\left<(A_{1}-B_{1} L)^*u,x_0\right>\right|^p\right]+\|\overline{x_{1}}\|_2^p\right)\\
	\leq& c_p\left(c_p\left(\E\left[\left|\left<(A_{1}-B_{1} L)^*u,x_0-\overline{x_0}\right>\right|^p\right]+\E\left[\left|\left<(A_{1}-B_{1} L)^*u,\overline{x_0}\right>\right|^p\right]\right)+\|\overline{x_{1}}\|_2^p\right)\\
	\leq& c_p\left(c_p\E\left[\|(A_{1}-B_{1} L)^*u\|_2^p\left|\left<\frac{(A_{1}-B_{1} L)^*u}{\|(A_{1}-B_{1} L)^*u\|_2},x_0-\overline{x_0}\right>\right|^p\right]\right.\\
	&\qquad  \left.+c_p\E\|A-B L\|^p\|\overline{x_0}\|_2^p+((1+\|L\|)K)^p\|\overline{x_0}\|_2^p\right)\\
	\leq& c_p\left(c_p\E\left[\|A_{1}-B_{1} L\|^p\E\left[\left|\left<\frac{(A_{1}-B_{1} L)^*u}{\|(A_{1}-B_{1} L)^*u\|_2},x_0-\overline{x_0}\right>\right|^p \bigg| A_{1},B_{1}\right]\right]\right.\\
	&\qquad \left.+c_p\left[K(1+\|L\|)\right]^{p}\|\overline{x_0}\|_2^p+\left[K(1+\|L\|)\right]^{p}\|\overline{x_0}\|_2^p\right)\\
	\leq& c_p\left(c_p\left[K(1+\|L\|)\right]^{p}\sigma_0^pp^{p/2}+(c_p+1)\left[K(1+\|L\|)\right]^{p}\|\overline{x_t}\|_2^p \right) \\
	\leq& \left[K(1+\|L\|)\right]^{p}\left(c_p^{2}\sigma_0^pp^{p/2}+(\Sigma_{i=1}^{2}c_p^i )\|\overline{x_0}\|_2^p\right),\quad \forall p\geq 1.
	\end{aligned}	
	\]

	For the $t+1$ case, using inductive hypotheses, we have 
	\[
	\begin{aligned}
	&\|\left<u,x_{t+1}-\overline{x_{t+1}}\right>\|_{L^p}^p\\
	\leq& c_p\left(\E\left[\left|\left<u,x_{t+1}\right>\right|^p\right]+\|\overline{x_{t+1}}\|_2^p\right)\\
	\leq& c_p\left(\E\left[\left|\left<(A_{t+1}-B_{t+1} L)^*u,x_t\right>\right|^p\right]+\|\overline{x_{t+1}}\|_2^p\right)\\
	\leq& c_p\left(c_p\left(\E\left[\left|\left<(A_{t+1}-B_{t+1} L)^*u,x_t-\overline{x_t}\right>\right|^p\right]+\E\left[\left|\left<(A_{t+1}-B_{t+1} L)^*u,\overline{x_t}\right>\right|^p\right]\right)+\|\overline{x_{t+1}}\|_2^p\right)\\
	\leq& c_p\left(c_p\E\left[\|(A_{t+1}-B_{t+1} L)^*u\|_2^p\left|\left<\frac{(A_{t+1}-B_{t+1} L)^*u}{\|(A_{t+1}-B_{t+1} L)^*u\|_2},x_t-\overline{x_t}\right>\right|^p\right]\right.\\
	&\qquad  \left.+c_p\E\|A-B L\|^p\|\overline{x_t}\|_2^p+\left[(1+\|L\|)K\right]^{(t+1)p}\|\overline{x_0}\|_2^p\right)\\
	\leq& c_p\left(c_p\E\left[\|A_{t+1}-B_{t+1} L\|^p\E\left[\left|\left<\frac{(A_{t+1}-B_{t+1} L)^*u}{\|(A_{t+1}-B_{t+1} L)^*u\|_2},x_t-\overline{x_t}\right>\right|^p \bigg| A_{t+1},B_{t+1}\right]\right]\right.\\
	&\qquad \left.+(c_p+1)\left[(1+\|L\|)K\right]^{(t+1)p}\|\overline{x_0}\|_2^p\right)\\
	\leq& c_p\left(c_p\left[K(1+\|L\|)\right]^{tp+p}\left(c_p^{2t}\sigma_0^pp^{p/2}+(\Sigma_{i=1}^{2t}c_p^i )\|\overline{x_0}\|_2^p\right)+(c_p+1)\left[(1+\|L\|)K\right]^{(t+1)p}\|\overline{x_0}\|_2^p \right) \\
	=&  \left[K(1+\|L\|)\right]^{(t+1)p}\left(c_p^{2t+2}\sigma_0^pp^{p/2}+(\Sigma_{i=1}^{2t}c_p^{i+2} +c_p(c_p+1))\|\overline{x_0}\|_2^p\right) \\
	=& \left[K(1+\|L\|)\right]^{(t+1)p}\left(c_p^{2(t+1)}\sigma_0^pp^{p/2}+(\Sigma_{i=1}^{2(t+1)}c_p^i )\|\overline{x_0}\|_2^p\right),\quad \forall p\geq 1.
	\end{aligned}	
	\]
	which concludes our claim.
Using~\eqref{eq:cl} and the fact that $\Sigma_{i=1}^{2t}c_p^i =\frac{c_p^{2t+1}-c_p}{c_p-1}\leq 2c_p^{2t}$, we obtain
\[
\begin{aligned}
	\|\left<u,x_{t}-\overline{x_{t}}\right>\|_{L^p}^p \leq c_p^{2t}\left[K(1+\|L\|)\right]^{tp}\left(\sigma_0^p+2\|\overline{x_0}\|_2^p\right)p^{p/2},\quad \forall p\geq 1.
\end{aligned}	
\]
Next, taking $p$-th root and using the fact that $c_p^{1/p}\leq 2$, we obtain 
\[
	\|\left<u,x_{t}-\overline{x_{t}}\right>\|_{L^p}\leq \left[4K(1+\|L\|)\right]^{t}\left(\sigma_0+2\|\overline{x_0}\|_2\right)p^{1/2},\quad \forall p\geq 1.
\]
\end{proof}


\begin{lemma}[$\|x_t\|_2^2$ is sub-exponential]
    \label{lemma:sen}
	Assume Assumption \ref{ass:3p} holds and the state process $\{x_t\}$ is generated by feedback matrix $L$. Then for each $t$, $\|x_t\|_2^2$ is sub-exponential random variable:
    $$\|x_t\|_2^2\sim SE((2n{\rm e}\beta_t(L))^2, 2n{\rm e}\beta_t(L))$$
	where 
	\[
		\beta_t(L):=16\left(\sigma_t(L)^{2}+\left[K(1+\|L\|)\right]^{2t}\E\left[\|x_0\|_2^2\right]\right).	
	\]
\end{lemma}
\begin{proof}
	By Lemma~\ref{lemma:xtisg}, $x_t$ is a sub-Gaussian random vector with parameter $\sigma_t(L)$.
	Thus 
	\(
	(x_t)_i	=\left<e_i ,x_t\right>
	\)
    is a sub-Gaussian random variable with parameter $\sigma_t(L)$, 
	i.e. 
	\[
		\|(x_t)_i	-\overline{(x_t)_i	}\|_{L^p}\leq \sigma_t(L) p^{1/2},\quad \forall p\geq 1.
	\]
	By induction and independency, we have
	\[
		\begin{aligned}
			\E\|x_t\|_2^2\leq& \E[\|A_t-B_tL\|^2\|x_{t-1}\|_2^2]\\
			= &\E\left[\|A_t-B_tL\|^2\right]\E\left[\|x_{t-1}\|_2^2\right]\\
			\leq &\left[(1+\|L\|)K\right]^{2}\E[\|x_{t-1}\|_2^2]\\
            \leq &\left[(1+\|L\|)K\right]^{2t}\E[\|x_{0}\|_2^2].
		\end{aligned}
	\]
	Hence, we have 
	\[
        	\begin{aligned}
			\|(x_t)_i^2-\overline{(x_t)_i^2}\|_{L^p}^p&=\E\left|(x_t)_i^2-\E\left[(x_t)_i^2\right]\right|^p\\
			&\leq 2^{p-1}\left(\E|(x_t)_i|^{2p}+\E\left[(x_t)_i^2\right]^p\right)\\
			&\leq 2^{p-1}\left(\E|(x_t-\overline{x_t})_i+\overline{(x_t)_i}|^{2p}+\E\left[(x_t)_i^2\right]^p\right)\\
			&\leq 2^{p-1}\left(2^{2p-1}(\E|(x_t-\overline{x_t})_i|^{2p}+\E\left[(x_t)_i\right]^{2p})+\E\left[(x_t)_i^2\right]^p\right)\\
			&\leq 2^{p-1}\left(2^{2p-1}\sigma_t(L)^{2p}(2p)^{p}+(2^{2p-1}+1)\E\left[(x_t)_i^2\right]^{p}\right)\\
			&\leq 2^{p-1}\left(2^{3p-1}\sigma_t(L)^{2p}p^p+2^{3p-1}\left[K(1+\|L\|)\right]^{2tp}\E\left[\|x_0\|_2^2\right]^p\right)\\
			&\leq 2^{4p-2}\left(\sigma_t(L)^{2p}+\left[K(1+\|L\|)\right]^{2tp}\E\left[\|x_0\|_2^2\right]^p\right)p^p,
		\end{aligned}
    \]
	i.e. 
	\[
		\|(x_t)_i^2-\overline{(x_t)_i^2}\|_{L^p}\leq 16\left(\sigma_t(L)^{2}+\left[K(1+\|L\|)\right]^{2t}\E\left[\|x_0\|_2^2\right]\right)p:=\beta_t(L) p.
	\]
	Thus
    \[
      \left\| \|x_t\|_2^2-\E\left[\|x_t\|_2^2\right] \right\|_{L^p}\leq \sum_{i=1}^n\|(x_t)_i^2-\overline{(x_t)_i^2}\|_{L^p}  \leq n\beta_t(L) p.
    \]
    Finally, we have
	\begin{equation}
        \label{aeq:sen}
		\begin{aligned}
			\E[e^{\lambda(\|x_t\|_2^2-\E\left[\|x_t\|_2^2\right]) }]&\leq 1+\sum_{p=2}^{\infty}\frac{|\lambda|^p\left\| \|x_t\|_2^2-\E\left[\|x_t\|_2^2\right]\right\|_{L^p}^p}{(p/{\rm e})^p}\\
			& \leq 1+\sum_{p=2}^{\infty}(n|\lambda|{\rm e} \beta_t(L))^p	\\
			& \leq 1+\frac{\lambda^2 (n{\rm e}\beta_t(L))^2}{1-n|\lambda|{\rm e}\beta_t(L)}\\
			& \leq 1+2\lambda^2(n{\rm e}\beta_t(L))^2\\
			& \leq e^{\frac{(2n{\rm e}\beta_t(L))^2}{2}\lambda^2},\quad \text{ provided } \quad |\lambda|\leq \frac{1}{2n{\rm e}\beta_t(L)}.
		\end{aligned}
	\end{equation}
	Hence $\|x_t\|_2^2-\E\left[\|x_t\|_2^2\right]\sim SE((2n{\rm e}\beta_t(L))^2,2n{\rm e}\beta_t(L))$.
\end{proof}

    \begin{proof}[Proof of Lemma~\ref{lemma:xmx}]
    In the proof of Lemma \ref{lemma:sen} we have 
    \[
        \E\|x_t\|_2^2\leq ((1+\|L\|)K)^{2t}\E[\|x_0\|_2^2],
    \]
    and 
    \[
        \left\|\|x_t\|_2^2-\overline{\|x_t\|_2^2}\right\|_{L^p}\leq n \beta_t(L)p.
    \]
     
    Thus, we have
    \[
    \begin{aligned}
    &\|x_t\T M x_t-\overline{x_t\T M x_t}\|_{L^p}^p\\
    =&	\E\left[|x_t\T M x_t- \overline{x_t\T M x_t} |^p\right]\\
    \leq& 2^{p-1}\left(\E\left[|x_t\T M x_t|^p\right]+ (\E[x_t\T M x_t ])^p\right)\\
    \leq &2^{p-1}\|M\|^p\left(\E\left[(\|x_t\|_2^2)^p\right]+ (\E[\|x_t\|_2^2 ])^p\right)\\
    \leq& 2^{p-1}\|M\|^p\left[2^{p-1}\E\left[|\|x_t\|_2^2-\overline{\|x_t\|_2^2}|^p\right]+(2^{p-1}+1)(\E[\|x_t\|_2^2 ])^p\right]\\
    \leq& 2^{p-1}\|M\|^p\left[2^{p-1}n^p \beta_t(L)^pp^p+2^p\left[(1+\|L\|)K\right]^{2tp}(\E[\|x_0\|_2^2])^p\right]\\
    \leq& 2^{2p-1}\|M\|^p\left(n^p \beta_t(L)^p+\left[(1+\|L\|)K\right]^{2tp}(\E[\|x_0\|_2^2])^p\right)p^p,\quad \forall p\geq 1.
    \end{aligned}	
    \]
    Then
    \[
        \begin{aligned}
            \|x_t\T M x_t-\overline{x_t\T M x_t}\|_{L^p}&\leq 4\|M\|\left(n \beta_t(L)+\left[(1+\|L\|)K\right]^{2t}(\E[\|x_0\|_2^2])\right)p\\
             &=\gamma_t(L,M) p, \quad \forall p\geq 1.
        \end{aligned}
    \]
    Therefore , similar proof as \eqref{aeq:sen} gives
    \[
        x_t\T M x_t \sim SE((2{\rm e}\gamma_t(L,M))^2,2{\rm e}\gamma_t(L,M) ).
    \]
    \end{proof}
The following lemma shows a constant matrix multiplying a sub-exponential random variable is a sub-exponential random matrix.
\begin{lemma}
	\label{lemma:sem}
	Let $\xi$ be a centered sub-exponential random variable:
	\[
	\|\xi\|_{L^p}\leq \gamma p,\quad \forall p\geq 1.
	\]
	Let $Q\in \SN$ be a symmetric constant matrix. Then $\xi Q$ is a sub-exponential random matrix: 
	\[
		\xi Q\sim SE((2{\rm e}\gamma  \|Q\|)^2I_n,2{\rm e}\gamma  \|Q\|)	.	
	\] 
\end{lemma}
\begin{proof}
	 Because $\xi$ is centered, using definition of matrix exponential and Fubini's theorem, we obtain
	\[
		\begin{aligned}
			\E\left[e^{\lambda \xi Q}\right]	&=\E\left[I_n+\sum_{p=2}^{\infty}\frac{(\lambda \xi Q)^p}{p!} \right]\\
			&=I_n +\sum_{p=2}^{\infty}\frac{\lambda^p \E[\xi^p]}{p!}Q^p\\
			&\preceq \left(1+\sum_{p=2}^{\infty}\frac{|\lambda|^p \|\xi\|_p^p}{(p/{\rm e})^p}\|Q\|^p\right)I_n\\
			&\preceq \left(1+\sum_{p=2}^{\infty}|\lambda|^p (\gamma {\rm e}\|Q\|)^p\right)I_n\\
			&= \left(1+\frac{(\lambda \gamma {\rm e} \|Q\|)^2}{1-|\lambda| \gamma {\rm e} \|Q\|}\right)I_n\\
			&\preceq \exp(2(\lambda \gamma {\rm e} \|Q\|)^2)I_n\\
			&=\exp( \frac{4(\gamma {\rm e} \|Q\|)^2I_n}{2}\lambda^2),\quad |\lambda|<\frac{1}{2\gamma {\rm e} \|Q\|}.
		\end{aligned}
	\]
		Therefore 
		\[
		\xi Q\sim SE((2\gamma {\rm e} \|Q\|)^2I_n,2\gamma {\rm e} \|Q\|)	.
		\]
\end{proof}

\begin{proof}[Proof of Lemma~\ref{lemma:concen}]
	Let 
	\[
	Q_i:=\begin{bmatrix}
		0_{m\times m}&U_i\\
		U_i\T & 0_{n\times n}
	\end{bmatrix}.	
	\]
	Matrix theory knowledge gives 
	\[
	\|\frac{1}{N}\sum_{i=1}^N \xi_i U_i\|=\|\frac{1}{N}\sum_{i=1}^N \xi_i Q_i\|,
	\]
	thus 
	\[
	\Prob(\|\frac{1}{N}\sum_{i=1}^N \xi_i U_i\|\geq \epsilon)=\Prob(\|\frac{1}{N}\sum_{i=1}^N \xi_i Q_i\|\geq \epsilon).
	\]
	Hence it suffices to consider 
	\[
		\Prob(\|\frac{1}{N}\sum_{i=1}^N \xi_i Q_i\|\geq \epsilon)	.
	\]
	
	Lemma~\ref{lemma:sem} implies $\xi_i Q_i\sim SE((2{\rm e}\|Q_i\|\gamma)^2I_{m+n},2{\rm e}\|Q_i\|\gamma)$ is sub-exponential random matrices. Thus 
	\begin{equation}
        \label{eq:concen1}
        	\E[\exp(\lambda \xi_i Q_i)]	\preceq  \exp(\frac{(2{\rm e}\|Q_i\|\gamma)^2I_{m+n}}{2}\lambda^2)\preceq  \exp(\frac{(2{\rm e}\|Q_i\|_F\gamma)^2I_{m+n}}{2}\lambda^2)=\exp(\frac{(2{\rm e}r\gamma)^2I_{m+n}}{2}\lambda^2),
    \end{equation}
    for $|\lambda|\leq \frac{1}{2{\rm e}r\gamma}$.

    By matrix Chernoff trick and Lieb's inequality (Lemma 6.12 and Lemma 6.13 in~\cite{wain2019}), 
    we obtain 
	\[
		\Prob(\|\frac{1}{N}\sum_{i=1}^N \xi_i Q_i\|\geq \epsilon)\leq 2{\rm e}^{-\lambda N \epsilon}\tr({\rm e}^{\sum_{i=1}^N \log \E[\exp(\lambda \xi_i Q_i)]})
	\]
    for $\lambda\in [ 0,\frac{1}{2{\rm e}r\gamma} )$.

	Since matrix logarithm is matrix monotone~\cite{wain2019}, using~\eqref{eq:concen1} we have
	\[
		\log \E[\exp(\lambda \xi_i Q_i)]	\preceq  2({\rm e}r\gamma)^2I_{m+n}\lambda^2,\quad \forall |\lambda|\leq \frac{1}{2{\rm e}r\gamma}.
	\]
	Therefore, 
	\[
	\begin{aligned}
		\Prob\left(\|\frac{1}{N}\sum_{i=1}^N \xi_i Q_i\|\geq \epsilon\right)&\leq 2{\rm e}^{-\lambda N \epsilon}\tr({\rm e}^{\sum_{i=1}^N \log \E[\exp(\lambda \xi_i Q_i)]})\\
		&\leq 2{\rm e}^{-\lambda N \epsilon}\tr({\rm e}^{N 2({\rm e}r\gamma)^2I_{m+n}\lambda^2})\\
		&\leq 2(m+n){\rm e}^{N (-\lambda\epsilon+ 2({\rm e}r\gamma)^2\lambda^2)},\quad \forall \lambda\in\left[ 0,\frac{1}{2{\rm e}r\gamma} \right).
	\end{aligned}	
	\]
    In different cases of $\epsilon$, we optimize $\inf_{\lambda\in[0,(2{\rm e}r\gamma)^{-1})} \left\{ -\lambda\epsilon+ 2({\rm e}r\gamma)^2\lambda^2 \right\} $ and get
    \[
        \Prob\left(\|\frac{1}{N}\sum_{i=1}^N \xi_i Q_i\|\geq\epsilon\right)\leq\begin{cases}
            2(m+n)\exp(-\frac{N\epsilon^2}{2(2{\rm e}r\gamma)^2})\quad \text{if } 0<\epsilon\leq 2{\rm e}r\gamma,\\
            2(m+n)\exp(-\frac{N\epsilon}{2(2{\rm e}r\gamma)})\quad \text{if } 2{\rm e}r\gamma<\epsilon. 
        \end{cases}
    \]
	Note that 
    \[
    \max\left\{-\frac{N\epsilon^2}{2(2{\rm e}r\gamma)^2}, -\frac{N\epsilon}{2(2{\rm e}r\gamma)}\right\} \leq -\frac{N\epsilon^2}{2\left(\epsilon(2{\rm e}r\gamma)+(2{\rm e}r\gamma)^2\right)},\quad \forall \epsilon>0.  
    \]
    Therefore, 
	\[
		\Prob\left(\|\frac{1}{N}\sum_{i=1}^N U_i^t\|\geq \epsilon\right)=\Prob\left(\|\frac{1}{N}\sum_{i=1}^N \xi_i Q_i\|\geq \epsilon\right)\leq 2(m+n)\exp(-\frac{\epsilon^2N}{2\left(\epsilon (2{\rm e}r\gamma)+(2{\rm e}r\gamma)^2\right)}) 
	\]
    for all $\epsilon>0$.
\end{proof}
\begin{proof}[Proof of Corollary~\ref{cor:concen}]
    First note that 
    \[
        \Prob\left(\|\frac{1}{N}\sum_{i=1}^N U_i^t\|\leq \epsilon\right)=1- \Prob\left(\|\frac{1}{N}\sum_{i=1}^N U_i^t\|>\epsilon\right)\geq 1-  \Prob\left(\|\frac{1}{N}\sum_{i=1}^N U_i^t\|\geq \epsilon\right)\geq 1-\delta,
    \]
    then we solve 
    \[
        2(m+n)\exp(-\frac{\epsilon^2N}{2\left(\epsilon (2{\rm e}r\gamma)+(2{\rm e}r\gamma)^2\right)}) \leq \delta
    \]
    and get 
    \[
        N\geq \frac{2}{\epsilon^2}\left((2{\rm e}r\gamma)^2+\epsilon (2{\rm e}r\gamma)\right)\log(\frac{2(m+n)}{\delta}).
    \]
    Finally, since $\|U\|_F\leq \sqrt{m\land n}\|U\|$, we have
    \[
        \Prob\left(\|\frac{1}{N}\sum_{i=1}^N U_i^t\|_F\geq \epsilon\right) \leq  \Prob\left(\|\frac{1}{N}\sum_{i=1}^N U_i^t\|\geq \frac{\epsilon}{\sqrt{m \land n}}\right) .
    \]
\end{proof}

Next lemma is the Bernstein inequality for bounded non-symmetric random matrices whose details can be found in~\cite{wain2019} Exercise 6.10.
\begin{lemma}[Matrix Bernstein inequality: non-symmetric version]
	\label{alemma:mbinsv}
	Let  $\{H_i\}_{i=1}^N\subset \R^{d_1 \times d_2}$ be a sequence of independent non-symmetric random matrices with 
	\[
	\E[H_i]=H, \quad \|H_i-H\|\leq b \quad a.s.	
	\]
	and 
	\[
	\max\left\{\left\|\frac{1}{N}\sum_{i=1}^N(\E[H_i\T H_i]-H\T H )\right\|,\left\|\frac{1}{N}\sum_{i=1}^N(\E[H_iH_i\T ]-HH\T  ) \right\|\right\}	\leq \sigma^2,
	\]
	then for any  $\epsilon>0$, 
	\[
	\Prob\left(\|\frac{1}{N}\sum_{i=1}^N(H_i-H)\|\geq \epsilon\right)\leq 2(d_1+d_2)\exp(-\frac{N\epsilon^2}{2(\sigma^2+b\epsilon)}).
	\]
	Furthermore, for any  $\epsilon>0$ and failure probability $0<\delta\leq 1$, if 
	\[
	N\geq \hsnsop(\epsilon,\delta,b,\sigma^2):=\frac{2}{\epsilon^2}(\sigma^2+\epsilon b)\log(\frac{2(d_1+d_2)}{\delta}),
	\]
	it holds that 
	\[
		\Prob\left(\|\frac{1}{N}\sum_{i=1}^N(H_i-H)\|\leq \epsilon\right)\geq 1-\delta.
	\]
	For Frobenius norm case, if 
	\[
	N\geq\hsnsf(\epsilon,\delta,b,\sigma^2):= \frac{2d_1\land d_2}{\epsilon^2}	(\sigma^2+\frac{\epsilon b}{\sqrt{d_1\land d_2}})\log(\frac{2(d_1+d_2)}{\delta}),
	\]
	it holds that 
	\[
		\Prob\left(\|\frac{1}{N}\sum_{i=1}^N(H_i-H)\|_F\leq \epsilon\right)\geq 1-\delta.
	\]
\end{lemma}

\subsection{Proofs in Section~\ref{sec:saag}}\label{sec:appsaag}
\begin{proof}[Proof of Lemma~\ref{lemma:egwfmihrr}]
	First, for any $U\in S_F(0,r)$ fixed, let $L'=L+U$. Note that  
    $$\|L'-L\|\leq \|L'-L\|_F=\|U\|_F=r\leq h_r(\epsilon,L,\Delta(L))\leq \Delta(L) \leq\hdelta(L).$$
	
	Since $\|L'-L\|\leq \Delta(L) \leq \hdelta(L)$ and $r\leq h_r(\epsilon,L,\Delta(L))\leq \frac{1}{h_c(L,\Delta(L))}$, Lemma \ref{lemma:cp} gives 
	\[
	|C(L+U)-C(L)|\leq C(L)h_c(L,\Delta(L))r\leq C(L)	,
	\]
	hence $C(L+U)\leq 2C(L)$. Above result is true for any $U\in S_F(0,r)$, which ensures the stability of cost in smoothing~\eqref{eq:smoothing}.
	Because $r\leq h_r(\epsilon,L,\Delta(L))\leq \frac{\epsilon}{\widehat{\hgrad}(L,\Delta(L))}$, Lemma~\ref{lemma:gp} implies 
	\[
	\|\nabla C(L+U)-\nabla C(L)\|_F\leq \widehat{\hgrad}(L,\Delta(L))r\leq \epsilon,	
	\]
	for any $U\in S_F(0,r)$. 
	Hence we obtain
	\[
        \|\nabla C(L)-\nabla C_r(L)\|_F\leq \E\|\nabla C(L+U)-\nabla C(L)\|_F\leq \epsilon.	
    \]
\end{proof}
\begin{proof}[Proof of Lemma~\ref{lemma:egwfmihr}]
	Note that
	\begin{equation}\label{eq:xx}
		\|\nabla C(L)-\hd\|_F\leq \|\nabla C(L)-\nabla C_r(L)\|_F+\|\nabla C_r(L) - \hd\|_F.
	\end{equation}
	By Lemma~\ref{lemma:egwfmihrr}, the first term in~\eqref{eq:xx} can be bounded as
	\begin{equation}
        \label{eq:egwfmihr1}
        \|\nabla C_r(L)-\nabla C(L)\|_F\leq  \frac{\epsilon}{2}.	
    \end{equation}
	
	For the second term in~\eqref{eq:xx}, we set $H_i=\frac{mn}{r^2}C(L+U_i)U_i$, so 
	\[
	\hd=\frac{1}{N}\sum_{i=1}^NH_i	.
	\]
	Equation~\eqref{eq:zooo} gives  
	\[
	H=\E[H_i]=\nabla C_r(L).
	\]
	Observe 
	\[
	\|H_i\|	\leq \|H_i\|_F=\|\frac{mn}{r^2}C(L+U_i)U_i\|_F\leq \frac{2mnC(L)}{r},
	\]
	then using~\eqref{eq:egwfmihr1} and Lemma \ref{lemma:bocag}, we have 
	\[
		\|H\|\leq \|H\|_F=\|\nabla C_r(L)\|_F\leq \frac{\epsilon}{2}+\|\nabla C(L)\|_F\leq \frac{\epsilon}{2}+h_1(L).	
	\]
	Thus 
	\[
	\|H_i-H\|\leq\|H_i-H\|_F\leq 	\frac{2mnC(L)}{r}+\frac{\epsilon}{2}+h_1(L)=b_0(L).
	\]
	Note that 
	\[
	\|\E\left[H_iH_i\T \right]-HH\T\|\leq \E\|H_iH_i\T\|+\|HH\T\|\leq (\frac{2mnC(L)}{r})^2+(\frac{\epsilon}{2}+h_1(L))^2
	\]
	and the same bound holds for $\|\E\left[H_i\T H_i\right]-H\T H\|$, hence 
	\[
	\sigma_0^2(L):=	\left(\frac{2mnC(L)}{r}\right)^2+\left(\frac{\epsilon}{2}+h_1(L)\right)^2.
	\]
	Finally using Lemma \ref{alemma:mbinsv}, if $$N\geq \hsnsf(\frac{\epsilon}{2},\delta,b_0(L),\sigma_0^2(L)):=\hs(\frac{\epsilon}{2},\delta,L),$$ it holds that 
	\[
		\Prob\left(\|\nabla C_r(L) - \hd\|_F\leq \frac{\epsilon}{2}\right)= \Prob\left(\|\frac{1}{N}\sum_{i=1}^N(H_i-H)\|_F\leq \frac{\epsilon}{2}\right)\geq 1-\delta.
	\]
\end{proof}
\begin{proof}[Proof of Lemma~\ref{lemma:5.8}]
	Note that 
	\[
	\|\nabla C(L) - \hdl\|_F	\leq \|\nabla C(L) -\hd\|_F + \|\hd-\hdl\|_F.
	\]
	Since $r\leq h_r(\frac{\epsilon}{4},L,\Delta(L))$, 
	Lemma~\ref{lemma:egwfmihr} implies that 
	\[
	C(L+U)\leq 2C(L).	
	\]
	and that the first term can be bounded as 
	\[
		\|\nabla C(L) -\hd\|_F\leq \frac{\epsilon}{2}.
	\]
	with at least $1-\delta$ probability.

	For the second term, Lemma \ref{lemma:acaswfhr} implies that if 
	\[
	l\geq \frac{4mn^2C(L)(\|\bq\|+\bnr (\|L\|+r)^2)}{r\epsilon\mu(\us_{\bq})^2}\geq \max_{\|U\|_F\leq r}h_l(L+U,\frac{r\epsilon}{2mn})	,
	\] 
	we have  
	\[
		|C^{(l)}(L+U_i)-C(L+U_i)|<\frac{r\epsilon}{2mn},
	\]
    for any $U_i$. 
	Thus, the second term can be bounded as 
	\[
	\|\hd-\hdl\|_F=\|\frac{1}{N}\sum_{i=1}^N\frac{mn}{r^2}(C^{(l)}(L+U_i)-C(L+U_i))U_i\|_F\leq \frac{mn}{r}\frac{r\epsilon}{2mn}=\frac{\epsilon}{2}.	
	\]
\end{proof}
\subsection{Well-definedness of key quantities}\label{sec:po4.5}
In this subsection we prove the  well-definedness of some key quantities.
First we recall the notation of sublevel sets $S_{\gamma}:=\{L\in\R^{n\times m}\big|C(L)\leq \gamma\}$ and $S:=S_{2C(L_0)}$.
\begin{lemma}[Well-definedness of $\eta^*$]
	\label{lemma:eta}
	$$\eta^* = \cc_{A,B,Q,R,\mu} \left(1+C(L_0)\right)^{-5}>0$$ is well defined.
\end{lemma}
\begin{proof}
	As discussed in Remark~\ref{rk:thm1}, to establish the existence of a positive $\eta^*$ such that the step size condition~\eqref{eq:ss} 
    holds at each update position, we need to establish a positive lower bound on the right hand side of \eqref{eq:ss}. To this goal 
	it suffices to lower bounded
	\begin{equation}\label{eq:eta}
		\inf_{L\in S_{C(L_0)}}\frac{1}{16}\min\left\{\left(\frac{\us_{\bq}\cdot \mu}{C(L)}\right)^2\frac{1}{\|\nabla C(L)\|(\bbs)^{1/2}\left(1+(\abs)^{1/2}\right)},\frac{\us_{\bq}}{C(L)\|R_L\|}\right\}.
	\end{equation}
	For the first term, using Lemma~\ref{lemma:boss}, we have
	\[
	\begin{aligned}
		&\left(\frac{\us_{\bq}\cdot \mu}{C(L)}\right)^2\frac{1}{\|\nabla C(L)\|(\bbs)^{1/2}\left(1+(\abs)^{1/2}\right)}	\\
		\geq & \cc_{B,Q,\Sigma_0}(1+C(L_0))^{-2}M_{\nabla}(C(L_0))^{-1}(1+M_{ABL}(C(L_0))^{1/2})^{-1}\\
		\geq & \cc_{B,Q,\Sigma_0}(1+C(L_0))^{-2}(\cc_{B,Q,R,\Sigma_0}(1+\gamma)^2)^{-1}(1+\cc_{A,B,R,\Sigma_0}^{1/2}(1+\gamma))^{-1}\\
		=& \cc_{A,B,Q,R,\Sigma_0}(1+C(L_0))^{-5}.
	\end{aligned}	
	\]
	For the second term 
	\[
		\frac{\us_{\bq}}{C(L)\|R_L\|}\geq \cc_{Q}(1+C(L_0))^{-1}(M_{R_L}(C(L_0)))^{-1}\geq \cc_{B,Q,R,\Sigma_0}(1+C(L_0))^{-2}.
	\]
	Hence, \eqref{eq:eta} can be lower bounded by 
	\[
	 \cc_{A,B,Q,R,\Sigma_0}(1+C(L_0))^{-5}\wedge \cc_{B,Q,R,\Sigma_0}(1+C(L_0))^{-2}\geq \cc_{A,B,Q,R,\Sigma_0}\wedge \cc_{B,Q,R,\Sigma_0} (1+C(L_0))^{-5}.
	\]
	Therefore, we can chose 
	\[
		\eta^* = \cc_{A,B,Q,R,\mu} \left(1+C(L_0)\right)^{-5}>0.
	\]
\end{proof}
\begin{lemma}[Well-definedness of $\bhdelta$]
	\label{lemma:wdbhd}
	$$\bhdelta=\cc_{A,B,R,\Sigma_0}(1+C(L_0))^{-2}\leq \min_{L\in S}\hdelta(L)$$ is well defined.
\end{lemma}
\begin{proof}
	It suffices to prove $$\hdelta(L)=\frac{\us_{\bq}\cdot \mu}{4(\bbs)^{1/2}\left(1+(\abs)^{1/2}\right)C(L)}$$ is lower bounded on $S$. Equivalently, we need to show 
	\[
		4(\bbs)^{1/2}\left(1+(\abs)^{1/2}\right)C(L)
	\]
	is upper bounded on $S$.
	Indeed, using Lemma \ref{lemma:boss}, we have  
	\[\begin{aligned}
	&4(\bbs)^{1/2}\left(1+(\abs)^{1/2}\right)C(L) \\
	\leq &4(\bbs)^{1/2}\left(1+(M_{ABL}(2C(L_0)))^{1/2}\right)2C(L_0)\\
	\leq & \cc_{A,B,R,\frac{1}{\mu}}(1+C(L_0))^2.
	\end{aligned}\]
	Hence 
	\[
		\bhdelta=\cc_{A,B,R,\Sigma_0}(1+C(L_0))^{-2}
	\]
\end{proof}

\begin{lemma}[Well-definedness of $\overline{h_c}$]
	\label{lemma:wdbhc}
	$$\overline{h_c}=\cc_{\Sigma_0,A,B,Q,R}(1+C(L_0))^4 \geq \max_{L\in S}h_c(L,\bhdelta)$$ is well defined.
\end{lemma}
\begin{proof}
	It suffices to prove $$h_c(L)=\frac{\bnxs}{\mu\cdot \us_{\bq}}\left[2\bnr (\bhdelta +2 \|L\|)+4\frac{C(L)}{\mu\cdot \us_{\bq}}(\bbs)^{1/2}(1+(\abs)^{1/2})(\|\bq\|+\bnr\|L\|^2)\right]$$ is upper bounded on $S$.

Indeed,  using Lemma \ref{lemma:boss}, we have 
\[
\begin{aligned}
	&\frac{\bnxs}{\mu\cdot \us_{\bq}}\left[2\bnr (\bhdelta+2 \|L\|)+4\frac{C(L)}{\mu\cdot \us_{\bq}}(\bbs)^{1/2}(1+(\abs)^{1/2})(\|\bq\|+\bnr\|L\|^2)\right]\\
	\leq&\frac{\bnxs}{\mu\cdot \us_{\bq}}\left[\frac{8C(L_0)}{\mu\cdot \us_{\bq}}(\bbs)^{1/2}(1+(M_{ABL}(2C(L_0)))^{1/2})(\|\bq\|+\bnr M_L(2C(L_0))^2) \right. \\
	&\left.+2\bnr (\bhdelta+2 M_L(2C(L_0)))\right]\\
	\leq & \cc_{\Sigma_0,A,B,Q,R}(1+C(L_0))^4
\end{aligned}	
\]
\end{proof}

\begin{lemma}[Well-definedness of $\epsilon'$]
	\label{lemma:wde}
	$$\epsilon'=\min\left\{\frac{\overline{\hdelta}}{\eta},\frac{\mu^2\us_{\br}\epsilon}{2\|\Sigma_{L^*}\|C(L_0)\overline{h_c}},\epsilon\right\} = \cc_{A,B,Q,R,\Sigma_0,\Sigma_{L^*}}(1+C(L_0))^{-5} \epsilon$$ is well defined.
\end{lemma}
\begin{proof}
	Obviously, the well-definedness of $\epsilon'$ follows from Lemma~\ref{lemma:wdbhd}~\ref{lemma:wdbhc} and Theorem~\ref{thm:mb}.
\end{proof}

\begin{lemma}[Well-definedness of $\hrgd$]
	\label{lemma:wdhrgd}
	$$\hrgd(\frac{\epsilon'}{8})=\cc_{A,B,Q,R,\Sigma_0,\Sigma_{L^*}}(1+C(L_0))^{-12}\epsilon\leq \min_{L\in S}h_r(\frac{\epsilon'}{4},L)$$ is well defined.
\end{lemma}
\begin{proof}
	Since we can exchange the order of two $\min$s,  
	\[
		\hrgd(\frac{\epsilon'}{8})=\min\left\{\overline{\hdelta},\frac{1}{\overline{h_c}},\frac{\epsilon'}{8\max_{L\in S}\widehat{\hgrad}(L,\overline{\hdelta})}\right\}.
	\]
	
	It suffices to prove $\max_{L\in S}\widehat{\hgrad}(L,\overline{\hdelta})$ is upper bounded on $S$. 
	Indeed, it is upper bounded by 
	\[
		\begin{aligned}
			&\frac{8C(L_0)}{\us_{\bq}}\left(\left(\overline{\|B\|\|A\|}+\bbs(\bhdelta+M_L(2C(L_0))) \right)\frac{2C(L_0)}{\mu\us_{\bq}}\left[2\bnr (\bhdelta+2M_L(2C(L_0)))\right.\right.\\
			& \left. \left.+\frac{8C(L_0)}{\mu\cdot \us_{\bq}}(\bbs)^{1/2}(1+(M_{ABL}(2C(L_0)))^{1/2})(\|\bq\|+\bnr M_L(2C(L_0))^2)\right]+\|\br\|+\bbs\frac{2C(L_0)}{\mu} \right)\\
			&+8\sqrt{\frac{M_{R_L}(2C(L_0))}{\mu}(2C(L_0)-C(L^*))}\left(\frac{2C(L_0)}{\us_{\bq}}\right)^2\frac{(\bbs)^{1/2}(1+(M_{ABL}(2C(L_0)))^{1/2})}{\mu}\\
			\leq & \cc_{A,B,Q,R,\Sigma_0}(1+C(L_0))^7.
		\end{aligned}	
	\]
\end{proof}

\begin{lemma}[Well-definedness of $\hlgd$]
	\label{lemma:wdhlgd}
	$$\hlgd(\frac{r\epsilon'}{4mn})=\cc_{A,B,Q,R,m,n,\Sigma_0,\Sigma_{L^*}}(1+C(L_0))^{20}\frac{1}{\epsilon r}\geq \max_{L\in S}\widetilde{h}_l(L,\frac{r\epsilon'}{4mn})$$ is well defined.
\end{lemma}
\begin{proof}
	It suffices to prove $$\tilde{h}_l(L,\frac{r\epsilon'}{4mn})=\frac{8mn^2C(L)(\|\bq\|+\bnr (\|L\|+r)^2)}{r\epsilon'\mu(\us_{\bq})^2}$$ is upper bounded on $S$.
Indeed, by Lemma~\ref{lemma:boss}, it holds that 
\[
\begin{aligned}
	&\max_{L\in S}\tilde{h}_l(L,\frac{r\epsilon'}{4mn})\\
	\leq&\frac{16mn^2C(L_0)(\|\bq\|+\bnr (M_L(2C(L_0))+\hrgd(\epsilon'/4))^2)}{r\epsilon'\mu(\us_{\bq})^2}\\
	\leq & \cc_{A,B,Q,R,m,n,\Sigma_0,\Sigma_{L^*}}(1+C(L_0))^{20}\frac{1}{\epsilon r}
\end{aligned}	
\]
\end{proof}

\begin{lemma}[Well-definedness of $\hs^*$]
	\label{lemma:wdhss}
	$$\hs^* (\epsilon',\delta',l,r)\geq \max_{L\in S}h_s(\epsilon',\delta',L,l,r)$$ is well defined,
	where 
	\[
	\begin{aligned}
		&\hs^* (\epsilon',\delta',l,r)\\
		:=&\cc_{*}^{4l}(1+C(L_0))^{4l+14}\frac{l^2}{\epsilon^2 r^2}\left(\log\log(1/\epsilon)+\log(1/\delta)+\log(1/\eta)+\log(l)\right),
	\end{aligned}
\]where $\cc_{*}:=\cc_{A,B,Q,R,n,m,K,\sigma_0,\Sigma_0,\Sigma_{L^*}}.$
\end{lemma}
\begin{proof}
	First we show that $\delta'$ is well-defined. 
	Clearly, same as in Theorem~\ref{thm:mb}, 
	\[
	k =	\cc_{Q,R,\Sigma_0,\Sigma_{L^*}}\frac{\log(1/\epsilon)}{\eta}(1+C(L_0)).
	\] 
	Thus, 
	\[
	\delta' = \cc_{Q,R,\Sigma_0,\Sigma_{L^*}}\frac{\delta \eta^2}{\log(1/\epsilon)^2}(1+C(L_0))^{-2}.
	\]
	
	Our next goal is to prove 
	\[
	h_s(\epsilon',\delta',L,l,r):=\max \left\{\hs(\frac{\epsilon'}{8},\frac{\delta'}{2},L),\hshf(\frac{\epsilon'}{2},\frac{\delta'}{2},L,l,r)\right\}
	\] is upper bounded on $S$.

	Indeed, for any $L\in S$, it holds that
\[
\begin{aligned}
	&\hs(\frac{\epsilon'}{8},\frac{\delta'}{2},L)\\
	=&128\frac{n\land m}{(\epsilon')^2}\left[(\frac{2mnC(L)}{r})^2+(\frac{\epsilon'}{8}+h_1(L))^2+\frac{\epsilon' (\frac{2mnC(L)}{r}+\frac{\epsilon'}{8}+h_1(L))}{2\sqrt{m\land n}}\right]\log(\frac{4(m+n)}{\delta'})\\
	\leq&128\frac{n\land m}{(\epsilon')^2}\left[(\frac{4mnC(L_0)}{r})^2+(\frac{\epsilon'}{8}+M_{\nabla}(2C(L_0)))^2\right.\\
	&\left.+\frac{\epsilon' (\frac{4mnC(L_0)}{r}+\frac{\epsilon'}{8}+M_{\nabla}(2C(L_0)))}{2\sqrt{m\land n}}\right]\log(\frac{4(m+n)}{\delta'})\\
	\leq & \cc_{A,B,Q,R,m,n,\Sigma_0,\Sigma_{L^*}}(1+C(L_0))^{15}\frac{1}{\epsilon^2 r^2}\left(\log\log(1/\epsilon)+\log(1/\delta)+\log(1/\eta)\right).
\end{aligned}	
\]

For $\hshf(\frac{\epsilon'}{2},\frac{\delta'}{2},L,l,r)$, 
since
\[
	\hshf(\frac{\epsilon'}{2},\frac{\delta'}{2},L,l,r):=\frac{8(mnl)^2 m \land n}{(\epsilon')^2 r^2}\left((2{\rm e}\widetilde{\gamma_l}(L,r))^2+\frac{\epsilon'}{2mnl\sqrt{ m \land n}} (2{\rm e}r \widetilde{\gamma_l}(L,r))\right)\log(\frac{4(m+n)l}{\delta'}),		
\] where 
\[
		\begin{aligned}
			\widetilde{\gamma_l}(L,r)&= 4 K\left( 1 + (\|L\|+r)^2\right) \left(n\widetilde{\beta_l}(L,r)+[(1+r+\|L\|)K]^{2l}\E[\|x_0\|_2^2]\right),\\
			\widetilde{\beta_l}(L,r)&= 16\left(\widetilde{\sigma_l}(L,r)^{2}+\left[K(1+r+\|L\|)\right]^{2l}\E\left[\|x_0\|_2^2\right]\right),\\
			\widetilde{\sigma_l}(L,r)&=\left[4(1+r+\|L\|)K\right]^l\left(\sigma_0+2\|\overline{x_0}\|\right).
		\end{aligned}
\]
It suffices to upper bound $\max_{L\in S}\widetilde{\gamma_l}(L,r)$, 
$\max_{L\in S}\widetilde{\beta_l}(L,r)$ 
and $\widetilde{\sigma_l}(L,r)$. 
This fact follows from 

\[
	\begin{aligned}
		\max_{L\in S}\widetilde{\gamma_l}(L,r)&\leq 4 \left(\|\bq\| +\|\br\|(M_L(2C(L_0))+r)^2\right) \left(n\max_{L\in S}\widetilde{\beta_l}(L,r) +[(1+r+M_L(2C(L_0)))K]^{2l}\E[\|x_0\|_2^2]\right),\\
		\max_{L\in S}\widetilde{\beta_l}(L,r)&\leq 16\left(\max_{L\in S}\widetilde{\sigma_l}(L,r)^{2}+\left[K(1+r+M_L(2C(L_0)))\right]^{2l}\E\left[\|x_0\|_2^2\right]\right),\\
		\max_{L\in S}\widetilde{\sigma_l}(L,r)&\leq \left[4(1+r+M_L(2C(L_0)))K\right]^l\left(\sigma_0+2\|\overline{x_0}\|\right).
	\end{aligned}
\]
Note that 
\[
	\max_{L\in S}\widetilde{\gamma_l}(L,r) \leq (\cc_{A,B,Q,R,n,m,K,\sigma_0,\Sigma_0})^{2l} (1+C(L_0))^{2(l+1)}.	
\]
Therefore, 
\[
	\begin{aligned}
		&\hshf(\frac{\epsilon'}{2},\frac{\delta'}{2},L,l,r)	\\
		\leq &(\cc_{A,B,Q,R,n,m,K,\sigma_0,\Sigma_0,\Sigma_{L^*}})^{4l}(1+C(L_0))^{4l+14}\frac{l^2}{\epsilon^2 r^2}\left(\log\log(1/\epsilon)+\log(1/\delta)+\log(1/\eta)+\log(l)\right),
	\end{aligned}
\]
and hence 
\[
	\begin{aligned}
		&\hs^* (\epsilon',\delta',l,r)\\
		=&(\cc_{A,B,Q,R,n,m,K,\sigma_0,\Sigma_0,\Sigma_{L^*}})^{4l}(1+C(L_0))^{4l+14}\frac{l^2}{\epsilon^2 r^2}\left(\log\log(1/\epsilon)+\log(1/\delta)+\log(1/\eta)+\log(l)\right).
	\end{aligned}
\]
\end{proof}

\end{document}